\documentclass[a4paper]{amsart}
\usepackage{amssymb,amsmath,enumerate,amsthm,url}
\usepackage[all,cmtip]{xy}
\usepackage{graphicx}
\usepackage[pdftex,colorlinks]{hyperref}
\theoremstyle{plain}
\newtheorem{theorem}{Theorem}[section]
\newtheorem{lemma}[theorem]{Lemma}
\newtheorem{proposition}[theorem]{Proposition}
\newtheorem{fact}[theorem]{Fact}
\newtheorem{corollary}[theorem]{Corollary}
\theoremstyle{definition}
\newtheorem{definition}[theorem]{Definition}
\newtheorem*{definition*}{Definition}

\newtheorem*{notation*}{Notation}
\theoremstyle{remark}
\newtheorem{remark}[theorem]{Remark}
\newtheorem{example}[theorem]{Example}
\newtheorem{question}[theorem]{Question}
\newtheorem{claim}[theorem]{Claim}
\newtheorem{assumption}[theorem]{Assumption}
\newtheorem{parag}[subsubsection]{}
\begin{document}
\providecommand{\eq}{\operatorname{eq}}
\providecommand{\st}{\operatorname{st}}
\providecommand{\loc}{\operatorname{loc}}
\providecommand{\cl}{\operatorname{cl}}
\providecommand{\trd}{\operatorname{trd}}
\providecommand{\dim}{\operatorname{dim}}
\providecommand{\locus}{\operatorname{locus}}
\providecommand{\ccl}{\operatorname{ccl}}
\providecommand{\tp}{\operatorname{tp}}
\providecommand{\dcl}{\operatorname{dcl}}
\providecommand{\acl}{\operatorname{acl}}
\providecommand{\Cb}{\operatorname{Cb}}
\providecommand{\Aut}{\operatorname{Aut}}
\providecommand{\Mat}{\operatorname{Mat}}
\providecommand{\End}{\operatorname{End}}
\providecommand{\Lie}{\operatorname{Lie}}
\providecommand{\alg}{\operatorname{alg}}
\providecommand{\im}{\operatorname{im}}
\providecommand{\U}{\mathcal{U}}
\providecommand{\G}{\mathcal{G}}
\providecommand{\R}{\mathbb{R}}
\providecommand{\C}{\mathbb{C}}
\providecommand{\Z}{\mathbb{Z}}
\providecommand{\Q}{\mathbb{Q}}
\providecommand{\N}{\mathbb{N}}

\providecommand{\stR}{{ }^*\R}
\providecommand{\dl}{\delta}
\providecommand{\bdl}{\boldsymbol\dl}
\newcommand{\tuple}[1]{\overline{#1}}
\renewcommand{\a}{\tuple{a}}
\renewcommand{\b}{\tuple{b}}
\renewcommand{\c}{\tuple{c}}
\renewcommand{\d}{\tuple{d}}
\renewcommand{\k}{\tuple{k}}
\newcommand{\h}{\tuple{h}}
\newcommand{\g}{\tuple{g}}
\newcommand{\q}{\tuple{q}}
\newcommand{\x}{\tuple{x}}
\newcommand{\y}{\tuple{y}}
\newcommand{\z}{\tuple{z}}
\newcommand{\alphatup}{\tuple{\alpha}}
\newcommand{\gammatup}{\tuple{\gamma}}
\newcommand{\psitup}{\tuple{\psi}}
\newcommand{\pp}{\mathfrak{p}}
\newcommand{\qq}{\mathfrak{q}}
\renewcommand{\O}{\mathcal{O}}
\providecommand{\AA}{\mathbb{A}}
\let\strokeL\L
\newcommand{\Los}{\strokeL o\'s}
\renewcommand{\L}{\mathcal{L}}
\renewcommand{\P}{\mathbb{P}}
\newcommand{\opint}{{\operatorname{int}}}
\providecommand{\dimo}{\operatorname{d^0}}
\providecommand{\eps}{\epsilon}

\providecommand{\invlim}{\mathop{\varprojlim}\limits}

\newcommand{\powerset}{\mathcal{P}}

\newcommand{\Keq}{K^{\eq}}

\newcommand{\acleq}{\acl^{\eq}}

\providecommand{\defnstyle}[1]{\textbf{\emph{#1}}}

\newcommand{\Disjointunion}{\mathbin{\dot{\bigcup}}}

\def\Ind#1#2{#1\setbox0=\hbox{$#1x$}\kern\wd0\hbox to 0pt{\hss$#1\mid$\hss}
\lower.9\ht0\hbox to 0pt{\hss$#1\smile$\hss}\kern\wd0}
\def\ind{\mathop{\mathpalette\Ind\emptyset }}
\def\notind#1#2{#1\setbox0=\hbox{$#1x$}\kern\wd0\hbox to 0pt{\mathchardef
\nn=12854\hss$#1\nn$\kern1.4\wd0\hss}\hbox to
0pt{\hss$#1\mid$\hss}\lower.9\ht0 \hbox to
0pt{\hss$#1\smile$\hss}\kern\wd0}
\def\nind{\mathop{\mathpalette\notind\emptyset }}

\title
[Projective geometries arising from Elekes-Szab\'o problems]
{Projective geometries arising from Elekes-Szab\'o problems}

\author{Martin Bays}
\address{Martin Bays\\
Institut f\"{u}r Logik und Grundlagenforschung\\
Fachbereich Mathematik und Informatik\\
Universit\"{a}t M\"{u}nster\\
Einsteinstrasse 62\\
48149 M\"{u}nster\\
Germany}
\email{mbays@sdf.org}

\author{Emmanuel Breuillard}
\address{Emmanuel Breuillard\\
DPMMS\\
Centre for Mathematical Sciences\\
Wilberforce Road\\
Cambridge CB3 0WB\\
England}
\email{efjb2@dpmms.cam.ac.uk}

\date{\today}

\begin{abstract}We generalise the Elekes-Szab\'o theorem to arbitrary arity 
  and dimension and characterise the complex algebraic varieties without power 
  saving. The characterisation involves certain algebraic subgroups of 
  commutative algebraic groups endowed with an extra structure arising from a 
  skew field of endomorphisms. We also extend the Erd\H{o}s-Szemer\'edi 
  sum-product phenomenon to elliptic curves. Our approach is based on 
  Hrushovski's framework of pseudo-finite dimensions and the abelian group 
  configuration theorem.
\end{abstract}

\maketitle

\section{Introduction}


Let $V \subset \C^n$ be an irreducible algebraic set over $\C$,
let $N \in \N$,
and let $X_i \subset \C$ with $|X_i| \leq N$, $i=1\ldots n$.
Then it is easy to see that
\[ |V \cap \prod_{i=1}^n X_i| \leq O_V(N^{\dim V}) .\]
Indeed, this follows inductively from the observation that there exists an
algebraic subset $W \subset V$ of lesser dimension and a co-ordinate 
projection of the complement $V \setminus W \rightarrow \C^{\dim V}$ with 
fibres of finite size bounded by a constant.

Say $V$ \defnstyle{admits no power-saving} if the exponent $\dim V$ is 
optimal, i.e.\ if for no $\epsilon>0$ do we have a bound $|V \cap 
\prod_{i=1}^n X_i| \leq O_{V,\epsilon}(N^{{\dim V}-\epsilon})$ as the $X_i$ 
vary among finite subsets of $\C$ of size $\leq N$.


In an influential paper Elekes and Szab\'o \cite{ES-groups} classified the 
varieties which admit no power-saving in the case $n=3$. In order to state 
their main theorem, we first need the following definition:

\begin{definition}
  A \defnstyle{generically finite algebraic correspondence} between 
  irreducible algebraic varieties $V$ and $V'$ is a closed irreducible 
  subvariety of the product $\Gamma \subset V \times V'$ such that the 
  projections $\pi_V(\Gamma) \subset V$ and $\pi_{V'}(\Gamma) \subset V'$ are 
  Zariski dense, and $\dim(\Gamma) = \dim(V) = \dim(V')$.
\end{definition}

Suppose $W_1,\ldots ,W_n$ and $W_1',\ldots ,W_n'$ are irreducible algebraic 
varieties, and $V \subset \prod_{i=1}^n W_i$ and $V' \subset \prod_{i=1}^n 
W_i'$ are irreducible subvarieties.
Then we say $V$ and $V'$ are in \defnstyle{co-ordinatewise correspondence} 
if there is a generically finite algebraic correspondence $\Gamma \subset V 
\times V'$ and a permutation $\sigma \in \operatorname{Sym}(n)$ such that 
for each $i$, the closure of the
projection $(\pi_i \times \pi'_{\sigma i})(\Gamma) \subset W_i \times 
W_{\sigma i}'$ is a generically finite algebraic correspondence (between the 
closure of $\pi_i(V)$ and the closure of $\pi'_{\sigma i}(V')$).

\begin{theorem}[Elekes-Szab\'o \cite{ES-groups}] \label{thm:ES}
  An irreducible surface $V \subset \C^3$ admits no power-saving if and only 
  if either
  \begin{enumerate}[(i)]\item $V \subset \C^3$ is in co-ordinatewise correspondence 
      with the graph $\Gamma = \{ (g,h,g+h) : g,h \in G \} \subset G^3$ of the 
      group operation of a 1-dimensional connected complex algebraic group 
      $G$,
    \item \label{ES1triv} or $V$ projects to a curve, i.e.\ 
      $\dim(\pi_{ij}(V))=1$ for some $i\neq j \in \{1,2,3\}$.
  \end{enumerate}
\end{theorem}


Here we generalise these results to arbitrary $n$ and $V \subset \C^n$.

\begin{definition}\label{special-one-dim}
  An irreducible algebraic set $V \subset \C^n$ is \defnstyle{special} if it 
  is in co-ordinatewise correspondence with a product $\prod_i H_i \leq 
  \prod_i G_i^{n_i}$ of connected subgroups $H_i$ of powers $G_i^{n_i}$ of 
  1-dimensional complex algebraic groups, where $\sum_i n_i = n$.
\end{definition}

We prove:

\begin{theorem} \label{thm:main1}
  An irreducible algebraic set $V \subset \C^n$ admits no power-saving if and 
  only if it is special.
\end{theorem}

The case of Theorem~\ref{thm:main1} with $V \subset \C^3$ and $\dim(V)=2$ is 
precisely Theorem~\ref{thm:ES}. Indeed it is easy to verify that $V$ is special if and only if it is either of the form $(i)$ or of the form $(ii)$. The latter occurs exactly when the special subgroup $H \leq G^3$ can be taken to be a diagonal subgroup $\{x_i = x_j\}$, while the curve $\pi_{ij}(V)$ gives the correspondence.

The case $V \subset \C^4$ with $\dim(V)=3$ is a consequence of the results of
\cite{RSZ-ES4}.

A slightly stonger version of the case $V \subset \C^n$ with $\dim(V) = n-1$, 
asking also for some uniformity in the power-saving (c.f.\ 
Remark~\ref{rem:gaps}), was conjectured by de Zeeuw in \cite[Conjecture~4.3]{deZeeuw-surveyER}. 

The case $V \subset \C^4$ with $\dim(V) = 2$ solves \cite[Problem~4.4]{deZeeuw-surveyER}. 

\begin{example}
  $V := \{ (x,y,z,w) \in \C^4 : xzw=1=yz^2w^2 \}$ is special because it is a 
  subgroup of $(\C^*)^4$, and geometric progressions witness that it admits no 
  power-saving: setting $X = \{ 2^k : -M \leq k \leq M \}$, we find $|V \cap 
  X^4| \geq \Omega(M^2) \geq \Omega(|X|^2)$.
\end{example}

\begin{example}
  Let $E \subset \P^2(\C)$ be an elliptic curve, say defined by $\{ y^2 = 
  x(x-1)(x-\lambda) \}$. Then taking $x$ co-ordinates yields a surface $V 
  \subset \C^3$ in co-ordinatewise correspondence with the graph $\Gamma_+ 
  \subset E^3$ of the elliptic curve group law, and arithmetic progressions in 
  $E$ witness that $V$ admits no power-saving. This demonstrates the necessity 
  of taking correspondences in the definition of special.

  To demonstrate the necessity of taking products, suppose
  $E' \subset \P^2(\C)$ is another elliptic curve. Then taking $x$ 
  co-ordinates
  yields a 4-dimensional subvariety $W \subset \C^6$ in co-ordinatewise 
  correspondence with the product $\Gamma_+ \times \Gamma_{+'} \subset 
  E^3\times E'^3$ of the graphs of the two group laws, and again arithmetic 
  progressions witness that $W$ admits no power-saving. But if $E'$ is not 
  isogenous to $E$, then $W$ is not in co-ordinatewise correspondence with a 
  subgroup of a power of a single elliptic curve (see 
  Fact~\ref{fact:corrIsog}).
\end{example}

In fact we obtain a more general result, with arbitrary varieties in place of 
the complex co-ordinates. Again, this generalises the corresponding result of 
\cite{ES-groups}, who considered the case of a subvariety $V$ of $\C^d \times 
\C^d \times \C^d$ of dimension $2d$ and with dominant projections to pairs of 
co-ordinates, and showed that $V$ must be in correspondence with the graph of 
multiplication of some algebraic group $G$. In \cite{breuillard-wang} it was 
noted that this group must be commutative. Theorem~\ref{thm:main} below gives 
a complete classification of the subvarieties without power saving, showing in 
particular that the groups involved must be commutative. To state the result, 
we first introduce the following definition.

\begin{definition} \label{defn:taucgp}
  Let $W$ be a complex variety.
  Let $C,\tau \in \N$ with $C\ge \tau$.
  A finite subset $X \subset W$ is in \defnstyle{coarse $(C,\tau)$-general 
  position} in $W$ if for any proper irreducible complex closed subvariety $W' 
  \subsetneq W$ of complexity at most $C$,
  we have $|W' \cap X| \leq |X|^{\frac 1\tau}$. When $C=\tau$ we will simply 
  say that $X$ is \defnstyle{$\tau$-cgp} in $W$.
\end{definition}

The notion of the complexity of a subvariety of a fixed variety is defined in 
full generality in \ref{absVars} below. In the case that $W$ is affine, $W' 
\subset W$ has complexity at most $C$ if it can be defined as the zero set of 
polynomials of degree at most $C$.

Let $W_i$, $i=1,\ldots ,n$, be irreducible complex varieties each of dimension 
$d$, and let $V \subset \prod_{i=1}^n W_i$ be an irreducible subvariety.

Now let $C,\tau \in \N$ and consider finite subsets $X_i \subset W_i$ with 
$|X_i|\leq N^d$, $N \in \N$, and with each $X_i$ in coarse $(C,\tau)$-general 
position in $W_i$.
As a straightforward consequence of coarse general position, if $\tau>d$ and 
$C$ is sufficiently large depending on $V$ only, we will see in 
Lemma~\ref{lem:power-saving} that we have a trivial bound $$|V \cap 
\prod_{i=1}^n X_i| \leq O_V(N^{\dim(V)}).$$
We say that $V \subset \prod_i W_i$ \defnstyle{admits a power-saving by 
$\epsilon>0$} if for some $C,\tau \in \N$ depending on $V$ only, this bound 
can be improved to
$|V \cap \prod_{i=1}^n X_i| \leq O_{V,\epsilon}(N^{\dim(V)- \epsilon})$.
We say $V$ \defnstyle{admits no power-saving} if it does not admit a 
power-saving by $\epsilon$ for any $\epsilon>0$.

It is easy to see that if $V$ admits no power-saving, then $\dim(V)$ must be 
an integral multiple of $d$ (see Lemma~\ref{lem:power-saving}). In Theorem 
\ref{thm:main} below we give a complete classification of the varieties with 
no power-saving. To this end we introduce as earlier a notion of special 
varieties, which generalises the previous definition and is slightly more 
involved.

Let $G$ be a connected commutative complex algebraic group, and let $\End(G)$ 
be the ring of algebraic endomorphisms of $G$. We will denote by $\End^0(G)$ 
the $\Q$-algebra $\End^0(G):= \Q \otimes_{\Z} \End(G)$. For example, if $G$ 
is a torus $G=\mathbb{G}_m^r$, then $\End(G)=\Mat_r(\Z)$ and 
$\End_0(G)=\Mat_r(\Q)$, and if $G=\mathbb{G}_a^r$ is a vector group, then 
$\End(G)=\End^0(G)=\Mat_r(\C)$. In any case $\End(G)$ is a subring of 
$\End^0(G)$. 

\begin{definition} An algebraic subgroup of $G^n$ is called a 
  \defnstyle{special subgroup} if it has an ``$F$-structure'' for some 
  division subring $F$ of $\End^0(G)$, by which we mean that it is the 
  connected component of the kernel $\ker A \leq G^n$ of a matrix $A \in 
  \Mat_n(F \cap \End(G))$.
\end{definition}

For example $F$ could be trivial and equal to $\Q$, in which case the 
corresponding special subgroups will be the connected components of subgroups 
defined by arbitrary linear equations with integer coefficients in the $n$ 
co-ordinates of $G^n$.

\begin{remark} \label{remk:lieSpecial}
  It will be convenient for us to express this condition in terms of the Lie 
  algebra $\Lie(H)$ of the subgroup $H \leq G^n$, which is defined as the 
  tangent space at the identity as a $\C$-vector space.
  An algebraic endomorphism $\eta \in \End(G)$ induces by differentiation a 
  linear map $d\eta : \Lie(G) \rightarrow \Lie(G)$, making $\Lie(G)$ into an 
  $\End^0(G)$-module.
  Then a subgroup $H \leq G^n$ is a special subgroup if and only if $\Lie(H) = 
  \Lie(G) \otimes_F J \leq \Lie(G)^n$
  for some division subring $F \subset \End^0(G)$
  and some $F$-subspace $J \leq F^n$
  (where we make the obvious identifications between $\Lie(G)^n$, $\Lie(G^n)$ 
  and $\Lie(G)
  \otimes_F F^n$).
\end{remark}

\begin{definition}
  An irreducible closed subvariety $V \subset \prod_{i=1}^n W_i$ of a product 
  of irreducible varieties is \defnstyle{special}
  if it is in co-ordinatewise correspondence with a product $\prod_i H_i \leq 
  \prod_i G_i^{n_i}$ of special subgroups $H_i$ of powers $G_i^{n_i}$ of 
  commutative complex algebraic groups, where $\sum_i n_i = n$.
\end{definition}

This is consistent with Definition \ref{special-one-dim}, because when $G$ is one-dimensional, every connected algebraic subgroup of $G^n$ has an $F$-structure, where $F=End^0(G)$. See  Lemma~\ref{lem:1dimSpecial} below. 

We are now in a position to state our main theorem:

\begin{theorem} \label{thm:main}
  Suppose $W_1,...,W_n$ are irreducible complex algebraic varieties of the 
  same dimension.
  Then an irreducible subvariety $V \subset \prod_{i=1}^n W_i$ admits no 
  power-saving if and only if it is special.
\end{theorem}

\begin{example}
  \newcommand{\HQ}{\mathcal{H}_\Q}
  \newcommand{\HZ}{\mathcal{H}_\Z}
  Let $G := (\C^\times)^4$. Then $\End^0(G) = \Q \otimes_{\Z} \End(G) \cong \Q 
  \otimes_{\Z} \Mat_4(\Z) \cong \Mat_4(\Q)$, the ring of $4\times 4$ rational 
  matrices. This is certainly not a division ring, but for example the 
  quaternion algebra
  \[ \HQ = ( \Q[i,j,k] : i^2=j^2=k^2=-1;\; ij=k;\; jk=i;\; ki=j )\]
  embeds in $\Mat_4(\Q)$ via the left multiplication representation.
  This defines in particular an action of $\HZ = \Z[i,j,k] \subset \HQ$ on $G$ 
  by endomorphisms given by a ring homomorphism $\alpha: \HZ \to \End(G), x 
  \mapsto \alpha_x$ defined by:
  \begin{align*}
    \alpha_i(a,b,c,d) &= (b^{-1},a,d^{-1},c); \\
    \alpha_j(a,b,c,d) &= (c^{-1},d,a,b^{-1}); \\
    \alpha_k(a,b,c,d) &= (d^{-1},c^{-1},b,a).
  \end{align*}
  Then for instance $V := \{ (x,y,z_1,z_2,z_3) \in G^5 : z_1=x\cdot y,\; z_2=x 
  \cdot \alpha_i(y),\; z_3=x\cdot \alpha_j(y) \}$ is a special subgroup of 
  $G^5$.

  To see that $V$ admits no power-saving, we can consider ``approximate 
  $\HZ$-submodules'': let $H_N := \{ n + mi + pj + qk : n,m,p,q \in 
  \{-N,\ldots ,N\} \}$, let $g \in G$ be generic (i.e.\ $\Q(g)$ has 
  transcendence degree $4$), and let $X_N := \alpha_{H_N}(g) = \{ \alpha_h(g) 
  : h \in H_N \} \subset G$. Since $\alpha_{\HZ}(g)$ is a finitely generated 
  subgroup of $G$, one can show (it is a consequence of Laurent's Mordell-Lang 
  theorem for tori, see Remark~\ref{remk:converseGp}) that for $W \subsetneq 
  G$ a proper Zariski-closed subvariety, $|W \cap \alpha_{\HZ}(g)|$ is finite 
  and bounded by a function of the complexity of $W$. Hence for all $\tau$, 
  for all sufficiently large $N$, we have that $X_N$ is $\tau$-$cgp$ in $G$. 
  But $\alpha_i(X_N) = X_N = \alpha_j(X_N)$, and so $|X_N^5 \cap V| \geq 
  \Omega(|X_N|^2)$.

  We show in Subsection~\ref{subsect:sharp} that any special subgroup admits 
  no power-saving.
  The argument of this example goes through for many groups $G$ (see 
  Remark~\ref{remk:converseGp}),
  but extra complications arise with other groups - in particular in the case 
  of a power of the additive group $G=\mathbb{G}_a(\C)^d=(\C^d,+)$, where many 
  more division rings can arise and no Mordell-Lang type result holds.
\end{example}

\begin{remark} \label{rmk:mainCodim1}
  In the situation of \cite[Theorem~27]{ES-groups} that $V \subset \C^d \times 
  \C^d \times \C^d $ projects dominantly with generically finite fibres to 
  each pair of co-ordinates, if $V$ is special then it is in co-ordinatewise 
  correspondence with the special subgroup $H_0 := \{x_1+x_2+x_3=0\}$ of a 
  $d$-dimensional commutative algebraic group $G$.
  Indeed, we first obtain from Theorem~\ref{thm:main} that it is in 
  co-ordinatewise correspondence with the connected component of $\{\alpha_1 
  y_1+\alpha_2 y_2+\alpha_3 y_3=0\}$ for some self-isogenies $\alpha_i \in 
  \End(G)$; then, setting $x_i := \alpha_i y_i$, we see that this is in 
  co-ordinatewise correspondence with $H_0$.

  Similarly for general $n$, if $\dim(V) = (n-1)d$ (as in \cite{RSZ-ES4} for 
  instance) then $\{\sum_ix_i=0\}$ is the only kind of special subgroup which 
  needs to be considered.
  But if $V$ has higher codimension, endomorphisms are indispensable.
\end{remark}

\begin{remark}[Explicit power-saving] \label{rem:gaps}
  We can consider strengthening Theorem~\ref{thm:main} by replacing the 
  condition that $V$ admits no power-saving with the condition that $V$ does 
  not admit a power-saving by $\eta$, where the ``gap'' $\eta$ is a constant 
  $\eta=\eta(d,n) > 0$.

  The existence of such a gap in the case $n=3$ is part of
  \cite[Main Theorem]{ES-groups}, and for $n=3$ and $d=1$ an explicit value of 
  $\eta=\frac16$ for this gap was found independently by Wang
  \cite{Wang-gap} and Raz, Sharir, and de Zeeuw \cite{RSZ-ESR};
  furthermore, \cite{RSZ-ES4} finds a gap of $\eta=\frac13$ for the case of 
  $n=4$ and $d=1$ (under a non-degeneracy assumption). For $n=3$ and $d$ 
  arbitrary some explicit gaps were obtained by Wang and the second author 
  (see \cite{breuillard-wang}). None of the gaps are known to be optimal.

  Our techniques for the general situation go via the abstraction of 
  combinatorial geometries and are not adapted to even proving the existence 
  of a gap, still less calculating one. However, in Section \ref{sec:warmup} 
  we work out the case of $n=3$, which does not require the full power of this 
  abstraction, and we obtain there \emph{an explicit gap $\eta=\frac{1}{16}$ 
  for all $d$ (see Theorem \ref{cor:ES} below)} and also recover the 
  above-mentioned $\frac{1}{6}$ gap when $n=3$ and $d=1$.
\end{remark}

We draw as a corollary of Theorem~\ref{thm:main1} the following generalised 
sum-product phenomenon.

\begin{corollary}[Generalised sum-product phenomenon]\label{gen-sum-product}
  Let $(G_1,+_1)$ and $(G_2,+_2)$ be one-dimensional non-isogenous connected 
  complex algebraic groups,
  and for $i=1,2$ let $f_i : G_i(\C) \rightarrow \C$ be a rational map.
  Then there are $\epsilon,c>0$ such that if $A \subset \C$ is a finite set 
  lying in the range of each $f_i$, then setting $A_i=f_i^{-1}(A) \subset 
  G_i(\C)$ we have
  \[ \max( |A_1 +_1 A_1|, |A_2 +_2 A_2| ) \geq c|A|^{1+\epsilon} .\]

\end{corollary}

\begin{remark} The usual sum-product phenomenon is the case $(G_1,+_1) = 
  (\C,+)$ and $(G_2,+_2) = (\C \setminus \{0\}, \cdot)$,
  with $f_1$ and $f_2$ being the identity maps.

  If instead $G_2 = E \subset \P^2(\C)$ is an elliptic curve defined by $\{ 
  y^2 = x(x-1)(x-\lambda) \}$, then we may take $f_2$ to be the rational map 
  $[x:y:1] \mapsto x$. This case of the additive group and an elliptic curve 
  was previously considered for finite fields in 
  \cite{shparlinski-sumProdElliptic}.

  The constant $c$ (which must depend on the $f_i$'s) is necessary as both 
  $A_i$'s could be finite subgroups of bounded order. We believe however that 
  the power-saving $\epsilon>0$ above is uniform over all group laws (and also 
  independent of the $f_i$'s); proving this would require establishing an 
  explicit gap in Theorem \ref{thm:main1} for $d=1$ and $n=6$. We do not tackle 
  this issue here.
\end{remark}

We also obtain the following result on intersections of subvarieties with 
powers of an approximate subgroup, or just of a set with small doubling.

\begin{theorem} \label{thm:coherentApproxSubgroup}
  Let $G$ be a commutative complex algebraic group.
  Suppose $V$ is a subvariety of $G^n$ which is not a coset of a subgroup.
  Then there are $N,\tau, \eps, \eta>0$ depending only on $G$ and the 
  complexity of $V$ such that if $A \subset G$ is a finite subset such that 
  $A-A$ is $\tau$-$cgp$ and $|A+A| \leq |A|^{1+\eps}$ and $|A| \geq N$, then
  $|A^n \cap V| < |A|^{\frac{\dim(V)}{\dim(G)} - \eta}$.
\end{theorem}

Note that Theorem \ref{thm:main} yields right away that if no such $\eta>0$ 
exists, then $V$ must be special. So the point here is to show that under the 
small doubling assumption the special $V$'s are in fact cosets of algebraic 
subgroups. The result is reminiscent of the Larsen-Pink type estimates for 
approximate groups (see \cite{hrushovski-wagner} \cite[Prop. 5.5]{hrushovski}, 
\cite[Thm 4.1.]{bgt}), with a stronger conclusion (the power-saving $\eta>0$) 
and stronger hypothesis (coarse general position). This conclusion is also 
reminiscent of results in Diophantine geometry of Manin-Mumford or 
Mordell-Lang type, although our methods are completely unrelated; see Example 
\ref{MM} for further comments in this direction.

\subsection{Method of proof}
The proof of our main results, Theorems \ref{thm:main1} and \ref{thm:main}, 
relies on an initial ultraproduct construction starting from a sequence of 
finite subsets witnessing the absence of power-saving. This yields 
pseudo-finite cartesian products. The field-theoretic algebraic closure 
relation then induces an abstract projective geometry at the level of the 
ultraproduct and we show, as a consequence of known incidence bounds 
generalising the Szemer\'edi-Trotter theorem, that this geometry is modular, 
i.e. satisfies the Veblen axiom of abstract projective geometries and can 
therefore be co-ordinatised. The division rings appearing in Theorem 
\ref{thm:main} arise that way. In the one-dimensional ($d=1$) case, the 
projective geometries which embed in the geometry of algebraic closure in an 
algebraically closed field were characterised in \cite{EH-projACF}, and in an 
appendix we use similar techniques (primarily the abelian group configuration 
theorem) to characterise them in the higher dimensional case. The main 
combinatorial results above then follow.

Much of the strategy is an implementation of ideas due to Hrushovski appearing 
in \cite{Hr-psfDims}, where he introduced the formalism of coarse 
pseudo-finite dimension and outlined a proof of the original Elekes-Szab\'o 
theorem in those terms. 

More generally, our results are a consequence of specialising ideas of model 
theory to this combinatorial setting. We use the conventions and language of 
model theory throughout.
Nonetheless, our treatment requires very little model-theoretic background and 
everything we need is described and recalled in Section \ref{sec:setup}. It is 
also mostly self-contained, except for the use of the group configuration 
theorem, recalled in Section \ref{sec:warmup}, and the Szemer\'edi-Trotter 
type incidence bounds recalled in \S \ref{subsec:incidence}.

\subsection{Related work}
We remark here on how this paper relates to other recent works on applications 
of model theory to similar problems.

In an unreleased work in progress, Hrushovski, Bukh, and Tsimmerman consider 
expansion phenomena in pseudo-finite subsets of pseudo-finite fields of size 
comparable to that of the non-standard prime field. This context is quite 
different from that we consider, in particular because of the failure of 
Szemer\'edi-Trotter in this regime, but there may be some overlap in 
techniques; in particular, their analysis also proceeds via modularity and the 
abelian group configuration theorem.

Meanwhile, Chernikov and Starchenko \cite{CS-ES} recently proved a version of 
Theorem~\ref{thm:ES} in strongly minimal structures which are reducts of 
distal structures. This direction of generalisation is orthogonal to the one 
we consider here, where we restrict to the case of $ACF_0$ (this restriction 
is used in Lemma~\ref{lem:SzT} and in Proposition~\ref{prop:converse}).

\subsection{Organisation of the paper} In Section \ref{sec:setup} we set up 
our notation for the rest of the paper and present Hrushovski's notion of 
pseudo-finite dimension of internal sets and its basic properties. This 
section is entirely self-contained. We also recall the 
Szemer\'edi-Trotter-type bounds for arbitrary varieties and recast them in 
this language. In Section \ref{sec:warmup} we reprove the original 
Elekes-Szab\'o theorem using the group configuration theorem and the formalism 
of pseudo-finite dimensions. In higher dimensions we also recover the 
commutativity of the ambient group and obtain an explicit power saving of 
$\frac{1}{16}$. This section is not used in the proof of our main theorem, but 
can be read as an example of the method, worked out in a special case. In 
Section \ref{gpNecessity} we give a counter-example to the original 
Elekes-Szab\'o theorem when the assumption of general position is removed. 
Section \ref{sec:proj} contains the proof of the key point: the modularity of 
the projective geometry associated to a variety without power-saving. In 
Sections \ref{sec:varieties} and \ref{sec:asymptotic} we complete the proof of 
Theorems \ref{thm:main1} and \ref{thm:main} modulo the result proven in the 
Appendix. In particular we prove the converse (the ``if'' direction of the 
theorems), which requires some information regarding division subrings of 
matrices. We also derive Corollary \ref{gen-sum-product}. In Section 
\ref{sec:subgroups} we prove Theorem \ref{thm:coherentApproxSubgroup} and draw 
some connections with Diophantine geometry. Finally the appendix is devoted to 
the higher-dimensional version of \cite{EH-projACF}.

\subsection{Acknowledgements}
Thanks to Mohammad Bardestani, Elisabeth Bouscaren, Ben Green, Martin Hils, 
Udi Hrushovski, Jonathan Kirby, Oriol Serra, Pierre Simon and Hong Wang for 
helpful conversations.

We would also like to thank the Institut Henri Poincar\'e and the organisers 
of the
trimester ``Model theory, combinatorics and valued fields'', where some of the
work was done. The second author acknowledges support from ERC grant GeTeMo 
no. 617129.

\setcounter{tocdepth}{1}
\tableofcontents

\section{The non-standard setup}\label{sec:setup}

In this preliminary section, we set up our notation and introduce the key 
concepts, which will be used in the proof of the main results. We assume some 
familiarity with the notions of first order languages, formulas and 
ultraproducts as expounded for example in the first two chapters of 
\cite{Marker-MT}. No other more sophisticated model-theoretical concepts will 
be assumed. 

\subsection{Coarse pseudo-finite dimension}
\label{subsect:psfDim}

We begin with a self-contained presentation of Hrushovski's formalism of 
coarse pseudo-finite dimensions from
\cite{Hr-psfDims}, slightly adapted to our purposes.

\begin{parag}{\bf{Ultraproducts and internal sets.}} We will fix a 
  non-principal ultrafilter $\U$ on the set of natural numbers. We say that a 
  property of natural numbers holds for \defnstyle{$\U$-almost every} $s$ if 
  the set of natural numbers $s$ for which the property holds is an element of 
  $\U$.

  We form the ultraproduct $K = \prod_{s \rightarrow \U} K_s$ of countably many 
  algebraically closed fields $K_s$, $s\ge 0$, which by definition is the 
  cartesian product $\prod_{s \ge 0} K_s$ quotiented by the equivalence relation 
  $(x_s)_s \sim (y_s)_s$ if and only if $x_s=y_s$ for $\U$-almost every $s \ge 
  0$. The field $K$ is also algebraically closed.

  We will assume throughout {\bf internal characteristic zero}; namely, we 
  assume $\operatorname{char}(K_s) = 0$ for all $s$. This is required for the 
  incidence bounds used in Lemma~\ref{lem:SzT} below.
  (See \cite[Corollary~5.6]{Hr-psfDims} for discussion on how it ought to be
  possible to weaken this assumption).

  In fact for our purposes it makes no difference to simply make the following 

  \begin{assumption}
    We assume that $K_s=\C$ for all $s$.
  \end{assumption}

  We denote by $\stR := \R^\U$ the corresponding ultrapower of $\R$, and call 
  its elements \defnstyle{non-standard reals}. The real field $\R$ embeds 
  diagonally in $\stR$ and its elements are called standard reals. The order on 
  $\R$ extends to an order on $\stR$ by saying that $x<y$ if and only if $x_s 
  <y_s$ for $\U$-almost every $s \ge 0$.

  We let $\st : \stR \rightarrow \R\cup\{-\infty,\infty\}$ be the 
  \defnstyle{standard part map}, namely $\st(\xi)$ is $\infty$ (resp. $-\infty$) 
  if $\xi=(\xi_s)_{s \ge 0} \in \stR$ is larger (resp. smaller) than any 
  standard real, and otherwise it is the ultralimit along $\U$ of the sequence 
  $(\xi_s)_s$, namely the unique $z \in \R$ such that for each $\epsilon>0$, 
  $|z-\xi_s|<\epsilon$ holds for $\U$-almost every $s$.

  Let $n$ be a positive integer. We say that a subset $X \subset K^n$ is 
  \defnstyle{internal} if $X = \prod_{s \rightarrow \U} X^{K_s}$ for some
  subsets $X^{K_s} \subset K_s^n$.
\end{parag}

\begin{parag}{\bf{Saturation and compactness.}} \label{sat} A standard 
  property of ultraproducts over a countable index set is their 
  $\aleph_1$-compactness. Namely countable families of internal sets have the 
  finite intersection property. This means that for each positive integer $n$, 
  if $X_0 \supset X_1 \supset \ldots $ is a countable chain of internal 
  subsets of $K^n$ such that $\bigcap_{i \ge 0} X_i = \emptyset $, then $X_i = 
  \emptyset $ for some $i \ge 0$. Equivalently if an internal set $X \subset 
  K^n$ lies in the union of countably many internal sets, then it already lies 
  in the union of finitely many of them.
\end{parag}

\begin{parag}{\bf{Coarse pseudo-finite dimension.}}\label{basic-prop} 
  Throughout we will fix once and for all some infinite non-standard real $\xi 
  \in \stR$ with $\xi > \R$, which we call the \defnstyle{scaling constant}. 
  This choice corresponds to a choice of calibration for the large finite sets 
  involved in our main results. Given an internal set $X = \prod_{s 
  \rightarrow \U} X^{K_s} \subset K^n$, we define the non-standard cardinality 
  of $X$ by $|X| :=
  \prod_{s \rightarrow \U} |X^{K_s}| \in \stR \cup \{\infty\}$ and its 
  \defnstyle{coarse pseudo-finite
  dimension} $\bdl(X)$ by
  \[ \bdl(X) := \st\left({\frac{\log |X|}{\log \xi}}\right) \in \R_{\geq 0} \cup
  \{-\infty,\infty\} \]
  (for the empty set we adopt the convention $\bdl(\emptyset)=\log(0) = 
  -\infty$).

  \begin{example} Let $X^{K_s}:=\{(p,q) \in \N^2 : p+q<s\}$, 
    $Y^{K_s}:=\{1,\ldots,s^s\}$ and $\xi_s:=s$ for all $s \ge 1$. Then 
    $\bdl(X)=2$ and $\bdl(Y)=\infty$.
  \end{example}

  We note here the following immediate properties of the coarse dimension, for 
  internal sets $A,B \subset K^n$:

  \begin{enumerate}
    \item (non-negativity) $\bdl(A) \ge 0$ if $A$ is non-empty,
    \item (monotonicity) If $A \subset B$, then $\bdl(A) \le \bdl(B)$,
    \item (ultrametricity) $\bdl(A \cup B)=\max \{ \bdl(A), \bdl(B)\}$.
  \end{enumerate}

\end{parag}

\begin{parag}{\bf{Definable sets.}} In order to talk about definable subsets 
  of $K^n$ we fix a language $\L$, which extends the language of rings 
  $\L_{\operatorname{ring}}=(+,-,\cdot,0,1)$ by \emph{only countably many} 
  symbols. We assume each $K_s$ is an $\L$-structure, and equip the 
  ultraproduct $K$ with the corresponding $\L$-structure. If $C \subset K$ is 
  a countable set, we write $\L_C$ for the language with new constant symbols 
  for the elements of $C$.

  To every first order formula $\phi=\phi(\x)$ in the language $\L_C$ with free 
  variables $\x:=(x_1,\ldots,x_n)$ there corresponds a \defnstyle{definable set}
  $\phi(K): = \{ \k \in K^{n} : \phi(\k) \textnormal{ holds in } K\}$. We say
  that the set $\phi(K)$ is $C$-definable or definable \defnstyle{over} $C$; in 
  other words, $\phi(K)$ is definable by a formula with parameters from $C$. 
  When $C=\emptyset$, we say that $\phi(K)$ is definable over $\emptyset$, or 
  definable without parameters. 

  We set $$K^{< \infty} := \bigcup_{n \geq 1} K^n$$
  the set of all finite tuples of elements of $K$. The family 
  $\mathcal{D}_{n,C}$ of $C$-definable subsets of $K^{n}$ forms a boolean 
  algebra, which contains all algebraic sets defined over $C$ (i.e.\ solutions 
  of polynomial equations whose coefficients are elements of $C$) as well as a 
  countable number of prescribed subsets (the graphs of functions from $\L$ and 
  sets of tuples satisfying the relations whose symbols belong to $\L$) and 
  $\cup_{n \ge 1} \mathcal{D}_{n,C}$ is stable under co-ordinate projections 
  (image and pre-image). Equivalently, instead of starting with the language 
  $\L$ and considering the associated definable sets, we may begin by giving 
  ourselves for each $n$ a countable number of prescribed internal subsets of 
  $K^n$ and consider the smallest family of subsets of $K^{< \infty}$ which 
  contains them as well as all algebraic sets defined over $C$ and is stable 
  under union, complement and co-ordinate projections (image and pre-image).

  Clearly every definable set is internal. The converse is not true, however any 
  single internal set in $K^n$ (or a countable family of such) can be made 
  definable by expanding the language by adding an $n$-ary relation symbol for 
  that internal set. This will be done below in order to make $\bdl$ continuous 
  and further down in the paper when, in our combinatorial applications, we will 
  always add the ultraproduct of the finite sets $X_i$ to the class of definable 
  sets.

  \begin{remark}[{\bf Notation for tuples}] We will often write a couple $(a,b) 
    \in K^2$ as $ab$, or given two tuples $\a \in K^n$ and $\b \in K^m$, we will 
    denote the tuple $(\a,\b) \in K^{n+m}$ simply by $\a\b$, concatenating the 
    two tuples.
  \end{remark}

\end{parag}

\begin{parag}{\bf{Types, $\bigwedge$-definable sets and coarse dimension of a 
  tuple.}}
  The \defnstyle{type} $\tp(\a)$ of a tuple $\a \in K^n$ is the family of all 
  formulae in $n$ variables in the language $\L$ (that is, without parameters) 
  satisfied by $\a$. The intersection of all $\emptyset$-definable subsets 
  containing $\a$ will be denoted by $\tp(\a)(K)$. Similarly if $C$ is a 
  \emph{countable} subset of $K$, we denote by $\tp(\a/C)$ the (countable) 
  family of all formulae in $\L_C$ with $n$ variables satisfied by $\a$ and we 
  set:

  $$\tp(\a/C)(K):= \bigcap_{\phi \in \tp(\a/C)} \phi(K) \;=\; \bigcap_{\a \in 
  Y,\; Y \textnormal{ } C\textnormal{-definable} }Y.$$

  By a  \defnstyle{$\bigwedge$-definable set} over $C$ (say ``type-definable'', 
  or ``wedge-definable''), we mean a subset $X$ of $K^n$, for some $n$, which is 
  the intersection of countably many $C$-definable sets. Such sets need not be 
  internal. We say a set is \defnstyle{$\bigwedge$-internal} if it is the 
  intersection of a countable collection of internal sets; so 
  $\bigwedge$-definable sets are $\bigwedge$-internal.

  For a $\bigwedge$-internal set $X \subset K^n$, we define
  \begin{equation}
    \bdl(X):=\inf \{ \bdl(Y) : Y \supset X,\; Y  \textnormal{ internal}
    \}.
  \end{equation}
  
  It is an immediate consequence of $\aleph_1$-compactness that if $X \subset 
  K^n$ is the intersection of a countable decreasing chain $X_0 \supset X_1 
  \supset \ldots $ of internal subsets of $K^n$ then $\bdl(X)=\inf_i 
  \bdl(X_i)$. In particular if $X \subset K^n$ is a $\bigwedge$-definable set 
  over a countable set $C \subset K$, then 
  \begin{equation}\label{typedefdel}\bdl(X)=\inf \{ \bdl(Y) : Y \supset X,\; Y  
    \textnormal{ definable over }C \}.
  \end{equation}

  The set $\tp(\a/C)(K)$ is $\bigwedge$-definable, so this allows to define 
  $\bdl(\a/C)$, the \defnstyle{coarse dimension of the tuple} $\a$ over $C$, as 
  $\bdl(\tp(\a/C)(K))$. Namely for $\a \in K^n$:
  \begin{equation}\label{defdel} \bdl(\a/C)=\inf\{\bdl(Y) : \a \in Y \subset 
    K^n,\; Y  \textnormal{ definable over }C \}.
  \end{equation}
  Abusing notation we will write $\bdl(\a)$ for $\bdl(\a/\emptyset)$, and 
  similarly if $C\subset K^{< \infty}$ (as opposed to just $K$) we will denote 
  by $\tp(\a/C)$ and $\bdl(\a/C)$ the type and coarse dimension of $\a$ over 
  the subset $C'\subset K$ of all co-ordinates of tuples from $C$. Further note 
  that if $C_1 \subset C_2 \subset K^{<\infty}$, then 
  \begin{equation}\label{mono2}\bdl(\a/C_2) \leq \bdl(\a/C_1).
  \end{equation}

  Given a $\bigwedge$-definable set $X \subset K^n$ over a countable set $C 
  \subset K$, we clearly have $\bdl(\a/C) \leq \bdl(X)$ for every $\a \in X$. An 
  important consequence of $\aleph_1$-compactness, which will be used several 
  times in the proofs, is the existence of some tuple realising the dimension:

  \begin{fact}[``existence of an independent realisation''] \label{fact:ideal}
    If $X \subset K^n$ is a $\bigwedge$-definable set over a countable set $C 
    \subset K$, then $X$ contains some $\a \in X$ with $$\bdl(\a/C) = \bdl(X).$$
  \end{fact}

  It is for this that we require countability of the language.

  \begin{proof} Note that for any $\a \in X$, we have $\bdl(\a/C) < \bdl(X)$ if 
    and only if there is a $C$-definable subset $Z \subset K^n$ such that $\a 
    \in Z$ and $\bdl(Z)<\bdl(X)$. Consider the family of all $C$-definable 
    subsets $Z$ with $\bdl(Z)< \bdl(X)$. It is enough to show that their union 
    does not contain $X$. But if this were the case, by $\aleph_1$-compactness 
    (see \ref{sat}), $X$ would be contained in the union of finitely many of 
    them, say $Z_1,\ldots,Z_m$. So $\bdl(\bigcup_1^m Z_i) \geq \bdl(X)$ by 
    monotonicity. However $\bdl(\bigcup_1^m Z_i) = \max \bdl(Z_i)$ by the 
    ultrametricity property (see \ref{basic-prop} above), and hence is 
    $<\bdl(X)$, a contradiction.
  \end{proof}

\end{parag}




Finally we record the following straightforward observation:

\begin{fact}\label{set-dep} For a tuple $\a=(a_1,\ldots,a_n) \in K^n$ and a 
  countable set $C \subset K$, the coarse dimension $\bdl(\a/C)$ depends only 
  on the set of co-ordinates $\{a_1,\ldots,a_n\} \subset K.$
\end{fact}

\begin{proof} Indeed it is invariant under any permutation of the 
  co-ordinates, because these induce bijections of $K^n$ and thus preserve 
  cardinality. And furthermore if $X$ is an internal set in $K^n$ such that 
  the last two co-ordinates $x_{n-1}$ and $x_n$ coincide for all $x \in X$, 
  then $\bdl(X)=\bdl(\pi(X))$, where $\pi(X)$ is the projection to the first 
  $n-1$ co-ordinates.
\end{proof}

\begin{parag}{\bf Continuity, additivity and invariance of coarse 
  dimension.}\label{continuity} We now come to two crucial properties of 
  $\bdl$: its additivity and its invariance. This will turn $\bdl$ into a 
  dimension-like quantity with properties very similar to those say of the 
  transcendence degree of the field extension generated by a tuple. To get 
  these properties it is enough to prove that $\bdl$ has the continuity 
  property we will now define. This property essentially amounts to requiring 
  that for each definable set the subset of fibers of a given size under a 
  co-ordinate projection is itself a definable set or is at least well 
  approximated by one. However continuity is not automatic and to get it we 
  will need to enrich our language $\L$ somewhat artificially, by adding a 
  (still countable) family of definable subsets.

  \begin{definition}\label{cont-def}
    We say that $\bdl$ has the continuity property (or is \defnstyle{continuous}) 
    if given $n,m\ge 1$, $\alpha \in \R$, $\epsilon>0$ and a  
    $\emptyset$-definable set $Y \subset K^{n} \times K^m$ there is a 
    $\emptyset$-definable set $W \subset K^m$ such that $$ \{\b \in K^m : 
    \bdl(Y_{\b}) \geq \alpha+\epsilon\} \subset W \subset \{\b \in K^m : 
    \bdl(Y_{\b}) \geq \alpha\},$$ where $Y_{\b}$ is the fiber $\{\x \in K^n; 
    (\x,\b) \in Y\}$.
  \end{definition}

  It is always possible to force the continuity of $\bdl$ by enlarging the 
  language $\L$ to a new language $\L'$, which is still countable and for which 
  $\bdl$ becomes continuous. Indeed for each $q \in \Q$ we may add a predicate 
  to simulate the quantifier $\exists_{\ge \xi^q}$ of having ``at least $\xi^q$ 
  solutions''. Explicitly, if $\xi=\lim_{s \to \U} \xi_s$ is as in the 
  definition of $\bdl$, let $\L_0 := \L$ and define $\L_{i+1}$ by adding to 
  $\L_i$ a new predicate $\psi_{\phi(\x,\y),q}(\y)$ for each formula 
  $\phi(\x,\y) \in \L_i$ and each $q \in \Q$, interpreted in $K_s$ by 
  $$\psi_{\phi(\x,\y),q}(K_s) := \{ \y : |\phi(K_s,\y)| \geq \xi_s^q\},$$ where 
  we have written $\phi(K_s,\y)$ for $\{\x : \phi(\x,\y) \textnormal{ holds in } 
  K_s\}$. So in the ultraproduct $K$ we have $\psi_{\phi(\x,\y),q}(K) = \{\y :   
  |\phi(K,\y)| \geq \xi^q\}$. Then we set $\L' := \cup_{i<\infty} \L_i$.

  It is then clear that $\bdl$ is continuous once we replace $\L$ with $\L'$. 
  Indeed if $\alpha \in \R$ and $\epsilon>0$ we may pick a rational $q \in 
  (\alpha,\alpha+\epsilon)$. Then if $\b \in \psi_{\phi(\x,\y),q}(K)$ then 
  $|\phi(K,\b)| \geq \xi^q$ so $\bdl(\phi(K,\b)) \geq q > \alpha$, while if 
  $\bdl(\phi(K,\b)) \geq \alpha+\epsilon$ then $\bdl(\phi(K,\b)) \geq q$ and $\b 
  \in \psi_{\phi(\x,\y),q}(K)$.

  \begin{remark}\label{C-def-cont} Note that the continuity property 
    automatically extends to definable sets with parameters. Namely if $Y$ is 
    assumed $C$-definable for some $C \subset K$, and $\bdl$ is continuous, we 
    may find a $C$-definable $W$ as in Definition \ref{cont-def}. Indeed there 
    is a finite tuple $\c_0 \in K^\ell$ for some $\ell \ge1$ with co-ordinates 
    in $C$ such that $Y=Y^0_{\c_0}=\{(\x,\y) : (\x,\y,\c_0) \in Y^0\}$ for some 
    $\emptyset$-definable set $Y^0 \in K^{n+m+\ell}$, so by continuity there is 
    $W^0$ a $\emptyset$-definable subset of $K^m$ such that $$ \{(\b,\c) \in 
    K^{m+\ell} : \bdl(Y_{\b,\c}) \leq \alpha\} \subset W^0 \subset \{(\b,\c) \in 
    K^{m+\ell} : \bdl(Y_{\b,\c}) \leq \alpha+\epsilon\}.$$ But now 
    $W:=W^0_{\c_0}$ is the desired $C$-definable set.
  \end{remark}

  Continuity yields the following crucial properties, which are characteristic
  of a dimension function; in particular, they are shared by transcendence
  degree.

  \begin{fact} \label{fact:dlContProps}
    Let $\a,\b \in K^{< \infty}$ and let $C \subset K$ be countable and 
    $\phi(\x,\y)$ a formula in the language $\L$. If $\bdl$ is continuous (for 
    $\L$) then it is
    \begin{enumerate}[(i)]\item \defnstyle{invariant}: if $\tp(\a)=\tp(\b)$, 
        then $\bdl(\phi(K,\a)) =
        \bdl(\phi(K,\b))$,
      \item \defnstyle{additive}:
        \[ \bdl(\a\b/C)=\bdl(\b/C) + \bdl(\a/\b C).\]
    \end{enumerate}
  \end{fact}

  Here as above $\phi(K,\a)$ denotes the definable set $\{\x : \phi(\x,\a) 
  \textnormal{ holds}\}$. We have used the convention $\alpha+\infty = 
  \infty+\alpha = \infty$, and $\a\b$ is a shorthand for $(\a,\b)$, the 
  concatenation of the tuples $\a$ and $\b$. Also we wrote $\b C$ for the union 
  of $C$ and the co-ordinates of $\b$. 

  \begin{proof} When $\a$ and $\b$ have the same type they belong to the same 
    definable sets, so $(i)$ is immediate from the continuity of $\bdl$. The 
    proof of $(ii)$ is given in \cite[Lemma~2.10]{Hr-psfDims}. We give it again 
    here for the reader's convenience. The idea is the following: if $Y$ is a 
    $C$-definable set in $K^n \times K^m$ containing $(\a,\b)$ and such that all 
    fibers $Y_{\b'}=Y \cap \pi_2^{-1}(\b')$ above the points $\b' \in \pi_2(Y)$ 
    (where $\pi_2$ is the co-ordinate projection to $K^m$) have the same size, 
    then clearly $\bdl(Y) = \bdl(\pi_2(Y)) + \bdl(Y_{\b'})$. Now the continuity 
    property of $\bdl$ ensures that we can find a $Y$ with $\bdl(Y)$ close to 
    $\bdl(\a\b/C)$ and with all fibers of almost the same size. This shows 
    additivity.

    We now give more details: by definition of the coarse dimension as an 
    infimum (see $(\ref{defdel})$), given $\epsilon>0$ we may find $C$-definable 
    sets $Y,Y' \subset K^n \times K^m$ such that $\a\b \in Y,Y'$ and 
    $\bdl(\a\b/C) \leq \bdl(Y) \leq \bdl(\a\b/C) +\epsilon$, $\bdl(\a/\b C) \leq 
    \bdl(Y'_{\b}) \leq \bdl(\a/\b C)+\epsilon$ and a $C$-definable set $Z\subset 
    K^m$ with $\b \in Z$ and $\bdl(\b/C) \leq \bdl(Z) \leq \bdl(\b/C) 
    +\epsilon$. Replacing $Y,Y'$ by $Y\cap Y' \cap \pi_2^{-1}(Z)$, we may assume 
    that $Y=Y'$ and $Z=\pi_2(Y)$. Now by continuity of $\bdl$ there is a 
    $C$-definable set $W \ni \b$ such that $|\bdl(Y_{\b'}) - 
    \bdl(Y_{\b})|<\epsilon$ for all $\b' \in W$. We may then further replace $Y$ 
    by $Y \cap \pi_2^{-1}(W)$ and get to a situation where $\bdl(\a\b/C) \leq 
    \bdl(Y) \leq \bdl(\a\b/C) +\epsilon$, $\bdl(\b/C) \leq \bdl(\pi_2(Y)) \leq 
    \bdl(\b/C) +\epsilon$ and all fibers $Y_{\b'}$ for $\b'\in \pi_2(Y)$ have 
    $\bdl(\a/\b C)-\epsilon \leq \bdl(Y_{\b'}) \leq \bdl(\a/\b C)+\epsilon$. We 
    thus conclude that $|\bdl(\a\b/C) - \bdl(\b/C) -\bdl(\a/\b C)| \leq 
    3\epsilon$ as desired.
  \end{proof}

  \begin{remark}
    \newcommand{\Km}{\mathbb{K}}
    We briefly remark in passing for the model-theoretically inclined reader that 
    a more sophisticated setup is also available, which in some ways is more 
    satisfactory than that described above.
    Working directly in a countable ultrapower with only $\aleph_1$-compactness, 
    as we have in this section, has the consequence that we must pick a 
    countable language to work with. In our applications we will have no real 
    control over the definable sets and can expect no tameness, so having to 
    make this choice is something of a distraction. An alternative would be to 
    define $K$ as above but in a language $\L_\opint$ which includes all 
    internal sets as predicates, and then to take a $\kappa$-saturated 
    $\kappa$-strongly homogeneous elementary extension $\Km$, for a cardinal 
    $\kappa$ which is larger than any parameter set we wish to consider.
    There is then a unique way to define $\bdl(\phi(\x,\a))$ for $\phi \in 
    \L_\opint$ and $\a \in \Km^{<\omega}$ such that $\bdl$ is continuous and 
    extends the original definition in the case $\a \in K^{<\omega}$.
    Namely, $\bdl(\phi(\x,\a)) := \sup \{ q \in \Q : \Km \vDash \exists_{\geq 
    \xi^q} \x.\; \phi(\x,\a) \}$, where $\exists_{\geq \xi^q} \x.\; \phi(\x,\y)$ 
    denotes an $\L_\opint$ formula with free variables $\y$ such that $K \vDash 
    \exists_{\geq \xi^q} \x.\; \phi(\x,\b)$ if and only if $| \phi(K,\b) | \geq 
    \xi^q$, for $\b \in K^{<\omega}$.
    (This is parallel to the way one defines dimension on an elementary 
    extension of a Zariski structure.)
    Here, continuity is meant in the sense of Definition~\ref{cont-def} - or 
    equivalently, that the map from the type space to the 2-point 
    compactification $S_{\y}(\emptyset) \rightarrow \R \cup\{-\infty,\infty\}: 
    \tp(\b) \mapsto \bdl(\phi(\x,\b))$ is well-defined and continuous.
    We can then work with elements of $\Km$ in order to analyse the internal 
    subsets of $K$. We will not use this alternative presentation, but some 
    readers may prefer to pretend that we do so throughout.
  \end{remark}

\end{parag}

\begin{parag}{\bf Algebraic independence and transcendence 
  degree.}\label{notn:baseField} At the heart of the combinatorial results of 
  this paper lies the interplay between combinatorics (via the coarse 
  pseudo-finite dimension $\bdl$) and algebraic geometry (via the notion of 
  algebraic dimension, or transcendence degree). To this effect we will fix a  
  \defnstyle{base field} $C_0$ and assume it is \emph{countable and 
  algebraically closed} and contained in $K$. We will then have to consider 
  the subclass of definable sets that are $C_0$-definable using only the 
  language of rings $\L_{ring}$. In the applications $C_0$ will be the 
  algebraic closure of the field of definition of the variety. As is 
  well-known, in an algebraically closed field, the sets that are 
  $C_0$-definable in $\L_{ring}$ coincide with the so-called constructible 
  sets of algebraic geometry defined over $C_0$, namely solutions of finitely 
  many polynomial equations and inequations with coefficients in $C_0$. After 
  enlarging $\L$ if necessary we can make the following

  \begin{assumption}
    We assume that $\L$ contains a constant symbol for each element of $C_0$.
  \end{assumption}

  \begin{notation*}[$0$ superscript] We will use a superscript $0$, e.g.\ 
    $\tp^0$, to indicate that we work in the
    structure $(K;+,\cdot,(c)_{c \in C_0})$ of $K$ as an algebraically closed
    field extension of the base field $C_0$, rather than in the full language
    $\L$. For example for $\a,\b \in K^{< \infty}$ and $C \subset K^{<\infty}$, 
    saying that
    $\tp^0(\a/C)=\tp^0(\b/C)$ means that they satisfy the same
    polynomial equations over the field $C_0(C)$ generated by $C_0$ and the 
    co-ordinates of all tuples belonging to $C$, i.e.\ for $f \in 
    C_0(C)[\tuple{X}]$, $f(\a)=0$
    if and only if $f(\b)=0$.
  \end{notation*}

  \begin{notation*}[algebraic closure $\acl^0$] Similarly for a subset $A 
    \subset K^{< \infty}$ we denote by $\acl^0(A)$ the field-theoretic 
    algebraic closure in $K$ of the subfield $C_0(A)$ generated by $C_0$ and 
    the co-ordinates of the elements of $A$.
  \end{notation*}

  When there is no superscript, we work in the full language $\L$.

  \begin{notation*}[transcendence degree $\dimo$] We write $\dimo$ for the 
    dimension with respect to $\acl^0$, i.e.\ for $A,B \subset K^{< \infty}$ 
    we set:
    $$\dimo(A/B) := \trd(C_0(AB)/C_0(B)),$$ where $\trd$ denotes the transcendence
    degree, and $C_0(B)$ the field extension of $C_0$ generated by $B$ and $AB$ 
    is short for $A \cup B$.
  \end{notation*}

  Note that, just like $\bdl$, \emph{$\dimo$ is additive}: if $\a,\b \in 
  K^{<\infty}$ and $C \subset K$, then \begin{equation}\label{add-dim0}\dimo(\a 
    \b / C) = \dimo(\a / \b C) + \dimo(\b / C),
  \end{equation}
  where, as earlier, $\b C$ is short for the union of $C$ and the co-ordinates 
  of $\b$. 

  Note finally that clearly $\dimo(A/B) = \dimo(A/\acl^0(B))$. 

  \begin{notation*}[independence $\ind^0_C$]\label{ind} If $A,B,C$ are subsets 
    of tuples of $K$, we will say that $A$ is algebraically independent of $B$ 
    over $C$ and write $A \ind^0_C B$ if $$\dimo(A/BC) = \dimo(A/C),$$ i.e.\ 
    if $C_0(A)$ is algebraically independent from $C_0(B)$ over $C_0(C)$.
    This is clearly a symmetric relation, namely $A \ind^0_C B$ if and only if $B 
    \ind^0_C A$.
  \end{notation*}





  \begin{notation*} For $A \subset K^{< \infty}$, we write
    \begin{equation}\label{closalg} \acl^0(A)^{< \infty} := \bigcup_{n\ge 1} 
      (\acl^0(A))^{n} \subset K^{< \infty} ,
    \end{equation} for the set of tuples algebraic over $A$. Note that this is 
    also
    the set of tuples with finite orbit under the group of field automorphisms
    $\Aut(K/C_0(A))$ fixing $C_0(A)$ pointwise.
  \end{notation*}
\end{parag}


\begin{parag}{\bf Coarse dimension of an algebraic tuple.} Let $C \subset K^{< 
  \infty}$ be a countable subset. If a tuple $\a$ belongs to $\acl^0(C)^{< 
  \infty}$, then it is contained in a finite $C$-definable set, namely the 
  Galois orbit of $\a$ over $C_0(C)$. In particular, since $\xi > \R$ we have 
  $\bdl(\a/C) = 0$. So
  \[ \a \in \acl^0(C)^{< \infty} \Rightarrow \bdl(\a/C) = 0 .\]
  We also record here the following generalisation of this observation, which 
  will be used in the proof of Proposition \ref{lem:cgpCoherentLinearity}. For 
  any $\a \in K^{< \infty}$ and countable $C \subset K$:
  \begin{equation}\label{alg-ineq} \bdl(\a/C)=\bdl(\a/\acl^0(C)).
  \end{equation}
  Indeed, first we have $\bdl(\a/C)\geq \bdl(\a/\acl^0(C))$ by $(\ref{mono2})$. 
  For the opposite inequality it is enough to show that if $\b \in \acl^0(C)$, 
  then $\bdl(\a/C) \leq \bdl(\a/\b)$. To see this, note that $\bdl(\b/C)=0$ by 
  the above remark, and thus by additivity $(\ref{fact:dlContProps}.ii)$ 
  $$\bdl(\a/C) \leq \bdl(\a\b/C) = \bdl(\a/\b C)+\bdl(\b/C)= \bdl(\a/\b C) \leq 
  \bdl(\a/\b).$$
\end{parag}


%

\begin{parag}{\bf Locus of a tuple.} If $\a \in K^n$ and $C \subset K$, we
  define the \defnstyle{locus} of $\a$ over $C_0(C)$, denoted by 
  $\loc^0(\a/C)$, to be the smallest Zariski-closed subset $V \subset K^n$ 
  such that $\a \in V$ and $V$ is defined by the vanishing of polynomials with 
  coefficients in $C_0(C)$. We also write $\loc^0(\a)$ for 
  $\loc^0(\a/\emptyset)$.

  Note that by definition $\loc^0(\a/C)$ is irreducible over $C_0(C)$, i.e.\ it 
  cannot be written as a finite union of more than one proper Zariski-closed 
  subset of $K^n$ defined over $C_0(C)$, but it may not be absolutely 
  irreducible (i.e.\ irreducible over $K$). However each absolutely irreducible 
  component is defined over some finite algebraic extension. In particular 
  $\loc^0(\a/\acl^0(C))$\emph{ is an absolutely irreducible component of} 
  $\loc^0(\a/C)$, and 
  \begin{equation}\label{dim-loc}\dimo(\a/C)=\dim(\loc^0(\a/C))=\dim(\loc^0(\a/\acl^0(C))). 
  \end{equation}

\end{parag}

\begin{parag}{\bf Abstract varieties.} \label{absVars}
  Our setup is adapted to working with tuples of elements of $K$, but in our 
  applications we will want to work with points of algebraic groups and of 
  general abstract algebraic varieties. We explain here how we bridge this gap 
  using standard notions from the model theory of algebraically closed fields, 
  as described in \cite{Pillay-ACF} or \cite[7.4]{Marker-MT}.

  We adopt the convention that varieties are always separated, but not 
  necessarily irreducible.

  If $V$ is an algebraic variety over an algebraically closed subfield $C \leq 
  K$, then $V$ admits a cover by finitely many affine open subvarieties over 
  $C$; that is, there are open subvarieties $V_i \subset V$ and (closed) 
  affine subvarieties $U_i \subset \mathbb{A}^{n_i}$ and isomorphisms $f_i : 
  V_i \rightarrow U_i$ over $C$, such that $V = \bigcup_i V_i$.
  Then $V(K)$ can be identified with the quotient of the disjoint union of the 
  $V_i(K)$ by the equivalence relation of representing the same point of 
  $V(K)$.
  Now $ACF_0$, the theory of algebraically closed fields of characteristic 
  zero, admits elimination of imaginaries, which exactly means that such a 
  quotient is in definable bijection over $C$ with a definable (i.e.\ 
  constructible) subset of $K^n$ for some $n$. We refer to 
  \cite[Remark~3.10(iii), Lemma~1.7]{Pillay-ACF} for details of this 
  construction.
  In this way, we embed $V(K)$ as a subset of $K^n$.
  Note that this embedding is not continuous.

  The precise embedding depends on our choice of cover. However, if $W \subset 
  V$ is another subvariety and $f : W \rightarrow U \subset K^m$ is an 
  isomorphism over $C$ with an affine variety, and if $a \in W(K) \subset V(K) 
  \subset K^n$, then the subfield of $K$ generated over $C$ by the 
  co-ordinates of $a$ according to our embedding of $V(K)$ in $K^n$ and those 
  according to $f$ are equal, $C(a) = C(f(a))$. In particular, for $a \in 
  V(K)$ the subfield $C(a) \leq K$ does not depend on our choice of cover.

  For $\tau \in \N$, we say that the \defnstyle{complexity} of a closed 
  subvariety $W \subset V$ is at most $\tau$ if for each $i$ the affine 
  variety $f_i(W\cap V_i) \subset U_i \subset \mathbb{A}^{n_i}$ can be defined 
  as the set of common zeros of a collection of polynomials each of degree at 
  most $\tau$.
  Note that the family of subvarieties of $V$ of complexity at most $\tau$ 
  forms a definable family; that is, there is $m \in \N$ and a constructible 
  set $X \subset V \times K^m$ over $C$ such that every subvariety of $V$ over 
  $K$ of complexity at most $\tau$ is of the form $X(\b) = \{ v : (v,\b) \in X 
  \}$ for some $\b \in K^m$. In fact this is the only property we require of 
  the notion of complexity.


\end{parag}

\begin{parag}{\bf Generic elements.} Let $V$ as before be an algebraic variety 
  over an algebraically closed $C \leq K$. For $a \in V(K)$ and $C' \leq K$ an 
  algebraically closed subfield containing $C$, we define the 
  \defnstyle{locus} of $a$ over $C'$ within $V$, $\locus^V(a/C')$, to be the 
  smallest Zariski-closed subvariety of $V$ defined over $C'$ and containing 
  $a$. If $V(K) \subset K^m$ is affine and defined over $C$, then 
  $\locus^V(a/C)=\loc^0(a/C)$.

  If $V$ is irreducible, a point $a \in V(K)$ of $V$ is \defnstyle{generic} if 
  it is contained in no proper closed subvariety over $C$, i.e.\ 
  $\locus^V(a/C) = V$; equivalently, $\trd(a/C) = \dim(V)$.

  \begin{remark} \label{rmk:corrGeneric}
    If $V \subset \prod_i W_i$ and $V' \subset \prod_i W_i'$ are closed 
    subvarieties where $V,V',$ $W_i,W_i'$ are irreducible varieties over $C_0$, 
    then $V$ and $V'$ are in co-ordinatewise correspondence if and only if they 
    have generics $\a \in V(K)$ and $\a' \in V'(K)$ such that $a_i \in W_i(K)$ 
    and $a'_i \in W'_i(K)$ are generic and for some permutation $\sigma \in 
    \operatorname{Sym}(n)$, we have $\acl^0(a_i) = \acl^0(a_{\sigma i}')$. 
    Indeed, $\loc^{0}(a_i,a_{\sigma i}') \subset W_i\times W_{\sigma i}'$ is 
    then a generically finite algebraic correspondence between $\loc^0(a_i)$ and 
    $\loc^0(a_{\sigma i})$, as required.
  \end{remark}

\end{parag}

\begin{parag}{\bf Canonical base.}\label{parag:canonical} In the proof of our 
  main theorems, it will be crucial to understand the dimensions of certain 
  families of varieties. The right concept for this (which serves a similar 
  purpose as the concept of Hilbert scheme in classical algebraic geometry) is 
  the notion of canonical base. 

  Recall that the \emph{field of definition} of a Zariski-closed subset $V 
  \subset K^n$ is
  the smallest field $k$ over which $V$ is defined; equivalently (since
  $\operatorname{char}(K)=0$), $k$ is such that a field automorphism $\sigma 
  \in \Aut(K)$
  fixes $V$ setwise if and only if it fixes $k$ pointwise.

  Given $\a \in K^{n}$ and $C \subset K^{<\infty}$ let $k\leq K$ be the field 
  of definition of the absolutely irreducible Zariski-closed subset 
  $\loc^0(\a/\acl^0(C))$ of $K^n$. A tuple $\d\in K^{<\infty}$ is said to be a 
  \defnstyle{canonical base} of $\a$ over $C$ if its co-ordinates together 
  with $C_0$ generate the subfield of $K$ generated by $C_0$ and $k$.

  Clearly if $\d \in K^{<\infty}$ is a canonical base of $\a$ over $C$ then it 
  is a canonical base of $\a$ over $\acl^0(C)$ and conversely. Furthermore $\d 
  \in \acl^0(C)$ and since $\loc^0(\a/\acl^0(C))$ is defined over $C_0(\d)$ we 
  have:
  $$\loc^0(\a/\acl^0(C)) = \loc^0(\a/\d)$$ so $\loc^0(\a/\d)$ is (absolutely) 
  irreducible and
  $$\dimo(\a/\d) = \dimo(\a/C) =  \dimo(\a/\acl^0(C)) =  \dimo(\a/C \d),$$
  in other words: $\a \ind^0_{\d} C$. In the proof of Proposition 
  \ref{lem:cgpCoherentLinearity} below we shall require the following fact.


  \begin{lemma} \label{lem:cb}
    Let $\a \in K^{<\infty}$ and $C \subset K^{< \infty}$. Let $\d \in 
    K^{<\infty}$ be a canonical base of $\a$ over $C$ and $V:=\loc^0(\a\d)$. Let 
    $\d_1,\d_2 \in \tp^0(\d)(K)$ and $\a' \in K^{< \infty}$ such that $\a'\d_i 
    \in V$. Then either $\dimo(\a'/\d_1\d_2) < \dimo(\a/\d)$, or $\d_1=\d_2$.
  \end{lemma}
  \begin{proof} Note that $$\dimo(\a'/\d_1\d_2) \le 
    \dimo(\a'/\d_1)=\dimo(\a'\d_1)-\dimo(\d_1)\le \dim V - \dimo(\d) = 
    \dimo(\a/\d).$$ So if $\dimo(\a'/\d_1\d_2) \ge \dimo(\a/\d)$, then the above 
    inequalities are equalities and in particular $\dimo(\a'\d_i)=\dimo(\a\d)$ 
    for each $i$. Since $V$ is irreducible we obtain $V=\loc^0(\a'\d_i)$. Hence 
    there exist $\sigma_i \in \Aut(K/C_0)$ with $\sigma_i(\a)=\a'$ and 
    $\sigma_i(\d)=\d_i$.

    Since $\d$ is a canonical base for $\a$ over $C$, $C_0(\d)$ is the field of 
    definition of $\loc^0(\a/\d)$. Hence $C_0(\d_i)$ is the field of definition 
    of $\loc^0(\a'/\d_i)$, and thus $\d_i$ is a canonical base of $\a'$ over 
    $\d_i$. In particular $\loc^0(\a'/\d_i)=\sigma_i(\loc^0(\a/\d))$ is 
    irreducible. Since $\loc^0(\a'/\d_1\d_2) \subset \loc^0(\a'/\d_i)$ have the 
    same dimension, we conclude that $\loc^0(\a'/\d_1\d_2)=\loc^0(\a'/\d_i)$.

    In particular $\loc^0(\a'/\d_1)= \loc^0(\a'/\d_2)$. Setting $\sigma=\sigma_2 
    \sigma_1^{-1}$, we get $\sigma(\a')=\a'$, $\sigma(\d_1)=\d_2$ and 
    $\sigma(\loc^0(\a'/\d_1))=\loc^0(\a'/\d_2)=\loc^0(\a'/\d_1)$. Hence $\sigma$ 
    fixes $\loc^0(\a'/\d_1)$ setwise. It must fix its field of definition 
    $C_0(\d_1)$ pointwise. Hence $\sigma(\d_1)=\d_1$, and $\d_2=\d_1$ as claimed.
  \end{proof}

\end{parag}

\begin{parag}{\bf Isogenies.}
  We say that commutative algebraic groups $G_1,G_2$ are \defnstyle{isogenous} 
  if there exists an \defnstyle{isogeny} $\theta : G_1 \twoheadrightarrow G_2$; 
  that is, a surjective algebraic group homomorphism with finite kernel. The 
  relation of being isogenous is an equivalence relation.

  We will apply in multiple places the following useful criterion for the 
  existence of an isogeny.

  \begin{fact} \label{fact:corrIsog}
    Let $(G;\times)$ and $(G';+)$ be connected algebraic groups over an 
    algebraically closed field $C_0$ of characteristic zero.
    Suppose the graphs $\Gamma_{\times}$ and $\Gamma_+$ of the group operations 
    are in co-ordinatewise correspondence, and $G'$ is commutative. Then $G$ is 
    also commutative, and is isogenous to $G'$.

    Moreover, if $(g,h) \in G^2(K)$ and $(g',h') \in (G')^2(K)$ are each 
    generic,
    and if $\acl^0(g) = \acl^0(g')$ and $\acl^0(h) = \acl^0(h')$ and 
    $\acl^0(g\times h) = \acl^0(g'+h')$
    (where $\acl^0(x) = C_0(x)^{\alg}$),
    then there are $n \in \N_{>0}$ and an isogeny $\alpha : G \rightarrow G'$ 
    and a point $c \in G'(C_0)$ such that $\alpha g = ng' + c$.
  \end{fact}
  \begin{proof}
    This is a consequence of \cite[Lemme~2.4]{BMP-beauxGroupes}.
    Indeed, that lemma yields, via Remark~\ref{rmk:corrGeneric}, that there is 
    an algebraic subgroup $S \leq G \times G'$ such that the projections to $G$ 
    and $G'$ are surjective and have finite kernels. It follows that $G$ is 
    abelian.
    Indeed, if $g \in G$ then $g^G$ is finite, so the centraliser $C_g$ is a 
    finite index subgroup of $G$ and hence is equal to $G$ since the latter is 
    connected. Alternatively, one may assume by the Lefschetz principle that 
    $G$, $G'$, and $S$ are complex algebraic groups, and observe that $S$ 
    induces an isomorphism of the Lie algebras.

    Now let $n$ be the exponent of the subgroup $\operatorname{coker}(S) := \{ y 
    \in G' : (0,y) \in S \} \leq G'$.
    Then by setting $\alpha(x) := ny$ whenever $(x,y) \in S$, we obtain a 
    well-defined isogeny $\alpha : G \twoheadrightarrow G'$.
    So $G$ is isogenous to $G'$.

    For the ``moreover'' clause, we use that the subgroup $S$ in the above cited 
    lemma is a coset of $V := \locus^{G\times G'}((g,g')/C_0)$. Knowing that $G$ 
    is abelian, we can see this fairly directly as follows.
    Let $S$ be the stabiliser of $\tp^0(g,g')$, namely $S := \{ \gamma \in 
    G\times G' : \dim(V \cap (\gamma+V)) = \dim(V) \}$.
    Then $S$ projects surjectively with finite kernel to $G$ and to $G'$, and it 
    follows from our assumptions that $V$ is a coset of $S$;
    indeed, this can be seen by applying \cite[Theorem~1]{Ziegler-noteGeneric} 
    to $(g,g') + (h,h') = (g \times h,g'+h')$.

    Since $C_0$ is algebraically closed and both $V$ and $S$ are over $C_0$, 
    there is $(c_1,c_2) \in (G \times G')(C_0)$ such that $V = (c_1,c_2)+S$.
    Then since the projection $\pi_1 : S \twoheadrightarrow G$ is surjective, 
    there exists $c' \in G'(C_0)$ such that $(g,g'+c') \in S$, namely any $c'$ 
    such that $(c_1,c_2-c') \in S$.
    Then with $\alpha,n$ as above, we have $\alpha(g) = n(g'+c') = ng' + nc'$,
    so $c := nc'$ is as required.
  \end{proof}
\end{parag}

\subsection{Incidence bounds and Szemer\'edi-Trotter}\label{subsec:incidence}

As in \cite{ES-groups} we will require some incidence bounds \`a la 
Szemer\'edi-Trotter in higher dimension. As is well-known, if $\mathcal{G}$ is 
a bi-partite graph between vertex sets $X_1$ and $X_2$ with the property that 
no two distinct points in $X_2$ have more than $B$ common neighbours, then a 
simple argument via the Cauchy-Schwarz inequality (e.g.\ see \cite[Prop. 
12]{ES-groups}), implies that the number of edges of $\mathcal{G}$ is at most 
$O(|X_1|^{\frac{1}{2}}|X_2| + B|X_1|)$. The theorem of Szemer\'edi-Trotter and 
its generalisations (such as \cite{pach-sharir}, \cite[Theorem 9]{ES-groups} 
or more recently \cite[Theorem~1.2]{FoxPachEtAl}) aim at improving this 
inequality by some power saving in the situation when the vertex sets $X_1$ 
and $X_2$ are points in Euclidean space and the graph $\mathcal{G}$ is given 
by some algebraic relation. For example Elekes-Szab\'o prove the following 
Szemer\'edi-Trotter-type result:

\begin{theorem}[{{\cite[Theorem 9]{ES-groups}}}]\label{st} If $V \subset 
  \C^{n_1}  
  \times \C^{n_2}$ is a complex algebraic subvariety there is 
  $\epsilon_0=\epsilon_0(n_2)>0$ such that the following holds. Let $B 
  \in \N$. Let $X_1,X_2 \subset \C^n$ be finite subsets. Write $V(y):=\{x \in 
  \C^n : (x,y) \in V\}$ for the fiber above $y \in \C^n$. Assume that for any 
  two distinct $y,y' \in X_2$ the intersection $V(y) \cap V(y')$ contains at 
  most $B$ points from $X_1$. Then the number $I$ of incidences $(x,y) \in X_1 
  \times X_2$ with $x \in V(y)$ satisfies:
  $$I \leq O_{B,V,n_1}(|X_1|^{\frac{1}{2}(1+\epsilon_0)} |X_2|^{1-\epsilon_0} 
  +|X_1| + |X_2|\log|X_1|).$$
\end{theorem}

We note that this bound has been slightly improved, with a better $\epsilon_0$ 
(namely any $\epsilon_0<\frac{1}{4n_2-1}$) and no $\log$ factor in 
\cite[Theorem~1.2]{FoxPachEtAl}.

Looking carefully at the proofs of the above theorem we find that the 
dependence in $B$ of the big-$O$ is sublinear, that is $O_{B,V,n_1} \leq B \cdot 
O_{V,n_1}$ (see \cite[Problem~11.4]{Sheffer-Incidence}). This aspect will be 
important for us (we can afford a polynomial dependence).

In what follows we spell out how the above incidence bound reads in the 
formalism of coarse pseudo-finite dimension. With the notation and terminology 
of Section~\ref{subsect:psfDim} (in particular $K$ is an ultraproduct of 
fields of characteristic zero and $\bdl$ is the coarse dimension 
\ref{basic-prop}), we have:

\begin{lemma}[Szemer\'edi-Trotter-type bound] \label{lem:SzT}
  Let $X_1 \subset K^{n_1}$ and $X_2 \subset K^{n_2}$, suppose each $X_i$ is 
  $\bigwedge$-internal, and let $X = (X_1 \times X_2) \cap V$ where $V \subset 
  K^{n_1+n_2}$ is a $K$-Zariski closed subset. Assume that $\bdl(X_1),\bdl(X_2)$ 
  are both finite.
  Set
  \[ \beta:= \sup_{a,b \in X_2 ;\; a \neq b} \bdl(X(a) \cap X(b)),\]
  where $X(y) := \{ x \in X_1 : (x,y) \in X \}$. Then for some $\epsilon_0>0$ 
  depending only on $n_2$, writing $y^+:=\max\{0,y\}$,
  \[ \bdl(X) \leq \beta + \max \left(\frac{1}{2}\bdl(X_1) + \bdl(X_2) - 
  \epsilon_0 (\bdl(X_2) - \frac{1}{2}\bdl(X_1))^{+},\bdl(X_1), \bdl(X_2) 
  \right) .\]
\end{lemma}
\begin{remark} In the same way, the trivial bound mentioned earlier (via 
  Cauchy-Schwarz) yields the same estimate on $\bdl(X)$ as above, but with 
  $\epsilon_0=0$. The original Szemer\'edi-Trotter theorem 
  \cite{szemeredi-trotter} corresponds to the case when $X_1$ is
  the ultraproduct of finite sets of points in $\R^2$ and $X_2$ is the
  ultraproduct of finite sets of lines in $\R^2$, and $V$ is the incidence
  relation $p \in \ell$. In this case $\epsilon_0=\frac{1}{3}$, which is 
  optimal.
\end{remark}
\begin{proof}[Proof of Lemma~\ref{lem:SzT}]
  Suppose first that $X_1$ and $X_2$ are internal sets, i.e.\ $X_i = \prod_{s 
  \rightarrow \U} X_i^{K_s}$ for $i=1,2$, for some $X_i^{K_s} \subset 
  K_s^{n_i}$, and $X^{K_s} = (X_1^{K_s} \times X_2^{K_s}) \cap V(K_s)$. Since 
  $\bdl(X_i)$ is finite, $X_i^{K_s}$ is finite for $\mathcal{U}$-almost every 
  $s$. The assumption $\bdl(X(a) \cap X(b))\leq \beta$ for each $a,b \in X_2$ 
  implies that for each $\epsilon>0$ for $\mathcal{U}$-almost every $s$ we 
  have: $$|X^{K_s}(a) \cap X^{K_s}(b)| \leq B_s,$$ where 
  $B_s:=\xi_s^{\beta+\epsilon}$, and $\xi=\lim_{s \to \mathcal{U}} \xi_s$ is 
  the scaling constant as in \S \ref{basic-prop}. Now Theorem \ref{st} 
  implies:
  \[ |X^{K_s}|
  \leq B_s \cdot O_{V,n_1} 
  \left({|X_1^{K_s}|^{\frac{1}{2}(1+\epsilon_0)}|X_2^{K_s}|^{1-\epsilon_0}
  + |X_1^{K_s}| + |X_2^{K_s}| \log |X_1^{K_s}|}\right) . \]
  Taking logarithms and passing to the ultralimit yields the desired bound.

  Finally the following claim allows us to reduce to the case when $X_i$ are 
  internal sets: 

  \begin{claim}
    For any $\beta' > \beta$,
    there are internal subsets $X_i' \supset X_i$, for $i=1,2$, such that
    for all $a,b \in X_2'$ with $a \neq b$, we have $\bdl(X'(a) \cap X'(b)) <
    \beta'$, where $X' := (X_1' \times X_2') \cap
    V$.
  \end{claim}
  \begin{proof}
    The variety $V$ is defined over a countable (finitely generated) subfield 
    of $K$, which we denote by $k$.
    Since each $X_i$ is $\bigwedge$-internal, we may work in a language $\L$ 
    in which each $X_i$ is $\bigwedge$-definable and $\bdl$ is continuous.
    Note that $X(y)=X_1 \cap V(y)$ for each $y \in X_2$. Since $X_1$ is 
    $\bigwedge$-definable, in view of $(\ref{typedefdel})$, for any $a,b \in 
    X_2$ with $a\neq b$ there is a $\emptyset$-definable subset $X_1^{a,b} 
    \supset X_1$ such that $\bdl(X_1^{a,b} \cap V(a) \cap V(b)) < \beta'$. By 
    continuity of $\bdl$ (see \ref{continuity} and Remark \ref{C-def-cont}), 
    there is a $k$-definable subset of $Z^{a,b}$ of $K^{n_2}\times K^{n_2}$ 
    containing $(a,b)$ such that $\bdl(X_1^{a,b} \cap V(a') \cap V(b')) < 
    \beta'$ for all $(a',b') \in Z^{a,b}$. Hence $(X_2)^2 \setminus \Delta$
    (where $\Delta$ denotes the diagonal) is covered by the family of 
    $k$-definable sets $Z^{a,b}$. This is a countable family, because there 
    are only countably many $k$-definable sets. Combined with the fact that 
    $X_2$ is $\bigwedge$-definable, $\aleph_1$-compactness (see \ref{sat}) now 
    implies that there must be a $\emptyset$-definable set $X'_2$ containing 
    $X_2$ such that $(X'_2)^2 \setminus \Delta$ is contained in finitely many 
    $Z^{a,b}$'s, say $Z^{a_1,b_1},\ldots,Z^{a_m,b_m}$. Let $X_1'$ be the 
    intersection of the corresponding $X_1^{a_i,b_i}$, $i=1,\ldots,m$. Then by 
    monotonicity of coarse dimension $\bdl(X_1' \cap V(a') \cap V(b')) < 
    \beta'$ for all $a',b' \in (X'_2)^2 \setminus \Delta$. So $X'_1$ and 
    $X'_2$ are as desired.
  \end{proof}

\end{proof}

\section{Warm-up: the Elekes-Szab\'o theorem}\label{sec:warmup}

In this section, we show how the proof of the original Elekes-Szab\'o theorem 
translates in the non-standard setup expounded in the previous section. This 
will help us motivate the notions introduced in the following section, where 
we will pass to the general case of arbitrary dimension and arity and work 
towards Theorems \ref{thm:main1} and \ref{thm:main}. We begin with the 
one-dimensional case, i.e.\ we prove Theorem \ref{thm:ES}. A similar result 
was proven by Hrushovski using similar techniques as 
\cite[Proposition~5.21]{Hr-psfDims}. We then proceed to recover 
Elekes-Szab\'o's second theorem, which corresponds to the case of a 
$2d$-dimensional variety in $(\C^d)^3$, and at the same time add two things: 
we establish that the associated algebraic group is in fact commutative (this 
was noted already in \cite{breuillard-wang}), and we also give an explicit gap 
in the power-saving, $\frac{1}{16}$ in fact. Although this is indeed new, we 
include this section mostly for the reader's convenience as a way to introduce 
some of the ideas in a special case. But a reader only interested in the proof 
of Theorems \ref{thm:main1} and \ref{thm:main} may safely skip ahead to 
Section~\ref{sec:proj}.

\begin{parag}{\bf Abelian group configuration theorem.} \label{fact:abGrpConf}

  While Elekes-Szab\'o used their `composition lemma' to establish the existence 
  of the associated algebraic group, we will rely directly on the Group 
  Configuration Theorem. This is a by now classical theorem of model theory due 
  to Zilber and Hrushovski. We first recall its statement in the form we need 
  and then describe a variant, due to Hrushovski, which ensures that the 
  associated group is commutative. In this paragraph $C_0 \leq K$ are arbitrary 
  algebraically closed fields, and we use the notation of \S 
  \ref{notn:baseField}, in particular $K^{<\infty} = \cup_{n>0} K^n$ and 
  $\acl^0(A)$ is the algebraic closure of $C_0(A)$ in $K$.

  \begin{theorem}[Group Configuration Theorem]
    \label{thm:groupConf}
    Suppose $a,b,c,x,y,z \in K^{<\infty}$ are such that in the following
    diagram

    
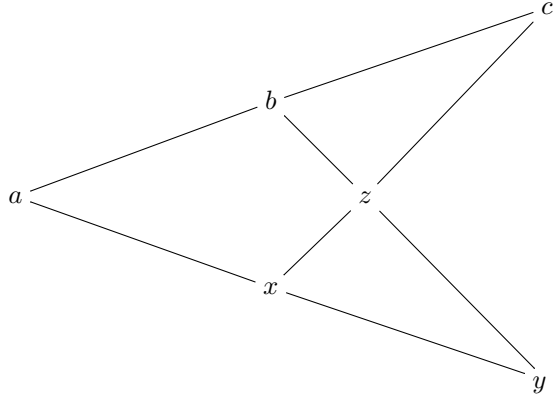
\[ \xymatrix{
    & & &&&& c \ar@{-}'[ddll][dddlll] \ar@{-}'[dlll][ddllllll] \\
    & & & b  & & & \\
    a & & & & z  & & \\
    & & & x && & \\
    && & & & & y\;\; \ar@{-}'[uull][uuulll] \ar@{-}'[ulll][uullllll] } \]

    for any three distinct points $a_1,a_2,a_3$,

    \begin{itemize}\item if $a_1,a_2,a_3$ lie on a common line then
        $a_i \in \acl^0(a_j,a_k)$ whenever $\{i,j,k\}=\{1,2,3\}$,
      \item if $a_1,a_2,a_3$ do not lie on a common line then $a_i \ind^0 a_ja_k$ 
        whenever $\{i,j,k\}=\{1,2,3\}$.
    \end{itemize}

    Then there is a connected algebraic group $(G,\cdot)$ defined
    over $C_0$, and generic elements $a',b',c'\in G(K)$ such that each primed 
    element is
    $\acl^0$-interalgebraic with the corresponding unprimed element, namely 
    $\acl^0(x) = \acl^0(x')$ for each $x \in \{a,b,c\}$, and $c'=b'\cdot a'$.
  \end{theorem}

  \begin{remark}
    Here, $\acl^0(x')$ is to be understood via a coding of elements of $G(K)$ as 
    tuples from $K$, as discussed in \ref{absVars}. But since $x'$ is generic, 
    we may equivalently fix a single arbitrary affine patch over $C_0$ and take 
    co-ordinates there.
  \end{remark}

  For a proof of this theorem, we refer the reader to \cite[Theorem~5.4.5, 
  Remark~5.4.10]{Pillay-GST}. Strictly speaking only a $\bigwedge$-definable 
  (in ACF) group $G'$ satisfying desired conclusions is obtained there, but by 
  \cite[Remark~1.6.21]{Pillay-GST}, $G'$ is in fact definable (in ACF). By the 
  Van den Dries-Hrushovski-Weil theorem \cite[Theorem~4.12]{Pillay-ACF} any 
  such group is definably isomorphic over $C_0$ to an algebraic group $G$ as 
  required.

  \begin{theorem}[Abelian Group Configuration Theorem]
    \label{thm:abGroupConf}
    Suppose $a,b,c,w,x,y,z \in K^{<\infty}$ are such that in the following
    diagram
    
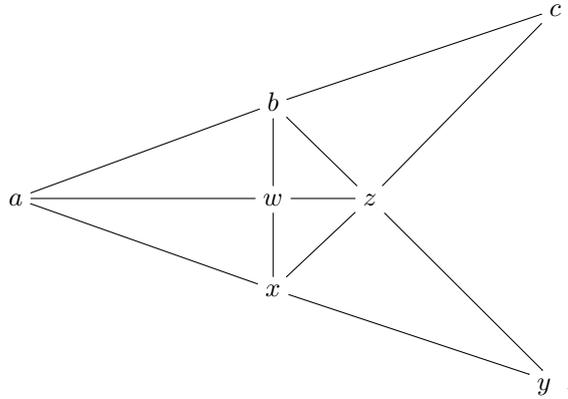
\[ \xymatrix{
    & & &&&& c \ar@{-}'[ddll][dddlll] \ar@{-}'[dlll][ddllllll] \\
    & & & b \ar@{-}'[d][dd] & & & \\
    a & & & w & z \ar@{-}'[l][llll] & & \\
    & & & x && & \\
    && & & & & y\;\; , \ar@{-}'[uull][uuulll] \ar@{-}'[ulll][uullllll] } \]
    for any three distinct points $a_1,a_2,a_3$,

    \begin{itemize}\item if $a_1,a_2,a_3$ lie on a common line then
        $a_i \in \acl^0(a_j,a_k)$ whenever $\{i,j,k\}=\{1,2,3\}$,
      \item if $a_1,a_2,a_3$ do not lie on a common line and $\{a_1,a_2,a_3\} \neq 
        \{w,c,y\}$ then $a_i \ind^0 a_ja_k$ whenever $\{i,j,k\}=\{1,2,3\}$.
    \end{itemize}

    Then there is an connected commutative algebraic group $G$ defined
    over $C_0$, and generics $a',b',c'\in G(K)$ such that each primed element is
    $\acl^0$-interalgebraic with the corresponding unprimed element, and 
    $c'=b'+a'$.

  \end{theorem}

  Note that the hypotheses of Theorem~\ref{thm:groupConf} are satisfied, so we 
  need only show that our additional assumptions yield that the algebraic group 
  $G$ obtained from that theorem is commutative. We refer to 
  \cite[Theorem~C.1]{BHM-CCMA} for a proof of this.



\end{parag}

\begin{parag}{\bf Elekes-Szab\'o - one dimensional case.} In this paragraph we 
  reprove the original Elekes-Szab\'o theorem, namely Theorem \ref{thm:ES}. We 
  start by reformulating it in the non-standard setup of the last section; in 
  particular we keep the notation of Section~\ref{subsect:psfDim}. So $K$ is 
  an ultrapower of the complex field, $\bdl$ is the coarse dimension 
  \ref{basic-prop} which is continuous in a countable language $\L$ containing 
  $\L_{ring}$ and constant symbols for each element of the countable 
  algebraically closed field $C_0$ over which $V$ is defined, and $\dimo$ 
  denotes transcendence degree over $C_0$.

  \begin{theorem}[Reformulation of Theorem \ref{thm:ES}]\label{thm:ES-ref} Let 
    $a_1, a_2, a_3 \in K$ and assume that for all $i\neq j$,
    $dim^0(a_i,a_j ) = dim^0(a_1,a_2,a_3) = 2$, $\bdl(a_i) \leq 1$ and 
    $\bdl(a_1,a_2,a_3) = 2$.
    Then there exists a connected one-dimensional algebraic group $G$ over $C_0$ and 
    $a'_1,a'_2,a'_3 \in G(K)$ with $\acl^0(a_i)=\acl^0(a'_i)$ for $i=1,2,3$ and 
    $a'_3=a'_1+a'_2$.
  \end{theorem}

  \begin{proof}[Proof of Theorem \ref{thm:ES} from Theorem \ref{thm:ES-ref}] 
    Assume $V \subset \C^3$ does not project to a curve on two co-ordinates and 
    has no power-saving. Then we may find a sequence of positive integers 
    $(N_s)_{s \ge 0}$ with $\lim_{s \to \infty} N_s=+\infty$ and finite subsets 
    $X^s_1,X^s_2$ and $X^s_3$ in $\C$ with $|X^s_i| \leq N_s$ for each $i,s$ 
    such that $$|X^s_1 \times X^s_2 \times X^s_3 \cap V| \geq 
    N_s^{2-\epsilon_s}$$ for some $\epsilon_s>0$ with $\lim_{s \to \infty} 
    \epsilon_s=0$. Passing to an ultraproduct $X_i=\prod_{s \to 
    \mathcal{U}}X^s_i$ for some non-principal ultrafilter $\mathcal{U}$ over the 
    integers, we obtain three internal sets $X_i \subset K$, where $K$ is the 
    ultrapower of $\C$, and we define the coarse dimension $\bdl$ as in 
    \ref{basic-prop} with scaling constant $\xi=\lim_{s \to \mathcal{U}} N_s$. 
    Hence $\bdl(X_i) \leq 1$ for each $i$ and $\bdl(X_1 \times X_2 \times X_3 
    \cap V) =2$. 

    Since $V$ is irreducible and does not project to a curve on two co-ordinates, 
    the fibers of co-ordinate projections of $V$ on pairs of co-ordinates have 
    uniformly bounded size. Consequently $|X^s_1 \times X^s_2 \times X^s_3 \cap V| 
    = O(|X^s_i \times X^s_j|)$ for all $s$ and all $i \neq j$. It follows that 
    $2=\bdl(X_1 \times X_2 \times X_3 \cap V) \leq \bdl(X_i) +\bdl(X_j)$, and 
    hence that $\bdl(X_i)=1$ for each $i$. 

    The variety $V$ is defined over some finitely generated subfield of $\C$. Let 
    $C_0$ be its algebraic closure in $\C$. It is a countable subfield. To be able 
    to talk about definable sets we specify a language $\L$ as follows: we start 
    with $\L_{ring}=(K,+,\cdot,0,1)$ the language of rings and enlarge it by 
    adding a constant symbol for each element of $C_0$ as well as a predicate for 
    each $X_i$, $i=1,2,3$, thus in effect forcing $X_i$ to be definable. Finally 
    we enlarge $\L$ as in \S \ref{continuity} so as to make $\bdl$ continuous and 
    hence additive.

    Now by Fact \ref{fact:ideal} we may find a triple $(a_1,a_2,a_3) \in X_1 
    \times X_2 \times X_3 \cap V$ such that $\bdl(a_1,a_2,a_3) = \bdl(X_1 \times 
    X_2 \times X_3 \cap V) =2$. Note that $(a_1,a_2,a_3)$ is generic in $V$, 
    i.e.\ it is not contained in any proper algebraic subvariety over the base 
    field $C_0$, because $|X^s_1 \times X^s_2 \times X^s_3 \cap W| =O_W(N_s)$ for 
    every one-dimensional subvariety $W \subsetneq V$ over $C_0$, and so $\bdl(X_1 
    \times X_2 \times X_3 \cap W) \leq 1$. Consequently 
    $\dimo(a_1,a_2,a_3)=\dimo(a_i,a_j)=2$ for all $i\neq j$. So we are in the 
    situation of Theorem \ref{thm:ES-ref}. Then $\loc^0(a'_1,a'_2,a'_3)$ is the 
    graph $\Gamma_G(\C)$ of the group operation of $G$, and we conclude that $V$ 
    has the required description via the correspondence 
    $\loc^0((a_1,a_2,a_3),(a'_1,a'_2,a'_3)) \subset V \times \Gamma_G(\C)$, which 
    is defined over $C_0$ and projects to the correspondences given by the 
    (irreducible) curves $\loc^0(a_i,a'_i) \subset \C \times G(\C)$.
  \end{proof}

  We now pass to the proof of Theorem \ref{thm:ES-ref}. We need to verify that 
  the hypotheses of the group configuration are met. For this we crucially need 
  the following lemma, which can be interpreted as saying that a $2$-parameter 
  family of plane curves with no power-saving must in fact be one-dimensional. 
  This is where the Szemer\'edi-Trotter bound comes into play.

  \begin{lemma}\label{1base} Let $x_1,\ldots,x_4 \in K$ be such that $\bdl(x_i)= 
    1$ and $\bdl(x_1,\ldots,x_4)=\dimo(x_1,\ldots,x_4)$. Assume that 
    $\dimo(x_1,x_2/x_3,x_4)=1$. Then there is $x_5 \in \acl^0(x_3,x_4)$ with 
    $\bdl(x_5)=\dimo(x_5)=1$ such that $\dimo(x_1,x_2/x_5)=1$.
  \end{lemma}

  \begin{proof} We postpone the proof of this lemma to the next subsection, 
    where a stronger quantitative version of it will be proven as Lemma 
    \ref{1base-m}. It is also a special case of Proposition 
    \ref{lem:cgpCoherentLinearity}.
  \end{proof}

  \begin{proof}[Proof of Theorem \ref{thm:ES-ref}] First note that the 
    assumptions imply that $\bdl(a_i)=1$ for each $i$. Indeed for any three 
    distinct $i,j,k$ we have $a_i \in \acl^0(a_j,a_k)$. Hence by 
    $(\ref{alg-ineq})$ we have $\bdl(a_i/a_j,a_k)=0$. And by additivity of 
    $\bdl$ (see Fact \ref{fact:dlContProps}) we get
    \begin{equation}\label{delcal} \bdl(a_i,a_j,a_k) = \bdl(a_i/a_j,a_k) + 
      \bdl(a_j,a_k) = \bdl(a_j,a_k) \leq \bdl(a_j)+\bdl(a_k).
    \end{equation} This forces $\bdl(a_j)$ and $\bdl(a_k)$ to be equal to $1$, 
    since both are $\leq 1$.

    Let $X=\tp(a_2,a_3/a_1)(K)$ be the set of realisations of the type of the pair 
    $(a_2,a_3)$ over $a_1$, namely the intersection of all definable sets over 
    $C:=\{a_1\}$ containing $(a_2,a_3)$. By additivity of $\bdl$ we have 
    $\bdl(X)=\bdl(a_1,a_2,a_3)-\bdl(a_1)=2-1=1$, by assumption. According to Fact 
    \ref{fact:ideal} we can find $(a_4,a_5) \in X$ such that 
    \begin{equation}\label{bdlX}\bdl(a_4,a_5/a_1,a_2,a_3) = 
      \bdl(X)=1.
    \end{equation}

    We will show that there is $a_6 \in K$ such that $a_1,\ldots,a_6$ satisfy the 
    hypotheses of the group configuration theorem as in the following diagram:

    
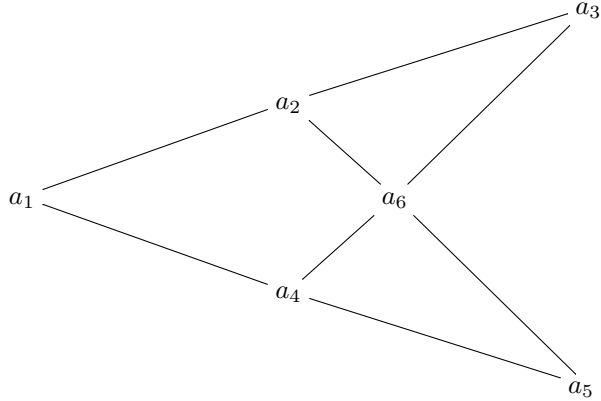
\[ \xymatrix{
    & & &&&& a_3 \ar@{-}'[ddll][dddlll] \ar@{-}'[dlll][ddllllll] \\
    & & & a_2 & & & \\
    a_1 & & & & a_6 & & \\
    & & & a_4 && & \\
    && & & & & a_5\;\; \ar@{-}'[uull][uuulll] \ar@{-}'[ulll][uullllll] } \]

    Since $(a_4,a_5)$ and $(a_2,a_3)$ have the same type over $a_1$, they have the 
    same type over the empty set, and in particular they belong to the same 
    algebraic subsets of $K^2$ defined over $C_0$. So $\dimo(a_4,a_5,a_1)=2$, 
    $\dimo(a_4,a_5)=\dimo(a_4,a_1)=\dimo(a_5,a_1)=2$. 

    Moreover the Zariski dimension of the whole system is $3$, i.e.\ 
    $\dimo(a_1,\ldots,a_5)=3$. Indeed it is at most $3$ given that $a_5 \in 
    \acl^0(a_1,a_4)$ and $a_3 \in \acl^0(a_1,a_2)$, but it cannot be less, for 
    otherwise $\dimo(a_4,a_5/a_1,a_2,a_3)=0$ forcing $\bdl(a_4,a_5/a_1,a_2,a_3)=0$ 
    by $(\ref{alg-ineq})$, a contradiction to $(\ref{bdlX})$.
    $$ \emph{Claim:} \hspace{.5cm} \dimo(a_3,a_4)=2, \dimo(a_2,a_5)=2 
    \textnormal{ and }\dimo(a_2,a_5/a_3,a_4)=1.$$
    Indeed if $\dimo(a_3,a_4)<2$, then $a_4 \in \acl^0(a_3)$, and thus 
    $$\dimo(a_1,\ldots,a_5)= \dimo(a_1,a_2,a_3,a_4)= \dimo(a_1,a_2,a_3)=2,$$ where 
    we have used that $a_5 \in \acl^0(a_1,a_4)$. In a similar way 
    $\dimo(a_2,a_5)=2$. Now by additivity of $\dimo$ we finally get 
    $\dimo(a_2,a_5/a_3,a_4)=1$, proving the claim.

    Further note that by additivity and $(\ref{bdlX})$ we have 
    $$\bdl(a_2,a_3,a_4,a_5)=\bdl(a_4,a_5/a_2,a_3)+\bdl(a_2,a_3)=1+2=3=\dimo(a_2,a_3,a_4,a_5).$$ 
    So Lemma \ref{1base} applies and gives $a_6 \in \acl^0(a_3,a_4)$ such that 
    $\dimo(a_2,a_5/a_6)=1$ and $\dimo(a_6)=1$. It then follows easily by 
    additivity of $\dimo$ that $\dimo(a_6,a_2)=\dimo(a_5,a_6)=2$ and $a_6 \in 
    \acl^0(a_2,a_5)$. This shows that $a_1,\ldots, a_6$ satisfy the hypotheses of 
    the group configuration theorem. We are done.
  \end{proof}

\end{parag}

\begin{parag}{\bf Coarse general position.} \label{cgp-parag}
  A significant new difficulty arises when dealing with the higher dimensional 
  situation, i.e.\ when $m=\dim W_i>1$ say in Theorem \ref{thm:main}. We will 
  have to assume that the finite sets $X_i \subset W_i$ do not have too large an 
  intersection with proper subvarieties. There are various ways to quantify this 
  assumption, for instance Elekes-Szab\'o's notion of \emph{general position} 
  requires that the intersections have bounded size with a bound depending only 
  on the complexity of the subvariety. We will adopt here the weaker assumption 
  of \emph{coarse general position}. As explained in Section \ref{gpNecessity} 
  below, some assumption of this kind is \emph{necessary} for the result to 
  hold.

  Recall from Definition \ref{defn:taucgp} that for $\tau \in \N$, a finite 
  subset $X$ of a complex algebraic variety $W$ is said to be in 
  \defnstyle{coarse $(C,\tau)$-general position} (or $(C,\tau)$-$cgp$ for short) 
  with respect to $W$ if $|W' \cap X| \leq |X|^{\frac 1\tau}$ for any proper 
  irreducible complex subvariety $W' \subsetneq W$ of complexity at most $C \in 
  \N$. In the non-standard setup of Section \ref{sec:setup}, where we have 
  specified a language $\L$ and defined the coarse dimension $\bdl$, it will be 
  convenient to define a notion of coarse general position for tuples $\a \in 
  K^{<\infty}$. We will say that $\a \in K^{<\infty}$ is in coarse general 
  position or is \defnstyle{cgp} for short if for every $\b \in K^{<\infty}$ 
  such that $\a$ is not independent from $\b$, that is such that $\dimo(\a/\b) < 
  \dimo(\a)$, we have: $$\bdl(\a/\b)=0.$$

  The two notions are closely related as follows. Suppose $W \subset \C^n$ is a 
  variety and $X = \prod_{s \rightarrow \U} X_s \subset W(K)$ is an internal 
  set. Assume that $X$ is definable without parameters in the countable language 
  $\L$ of Section \ref{sec:setup} for which the coarse dimension $\bdl$ is 
  continuous (see \ref{continuity}).

  \begin{lemma}\label{cgplemfin} Suppose that $0 < \bdl(X) <   \infty$ and that 
    for any $\tau \in \N$, there is $C \ge \tau$ such that for $\U$-many $s$, 
    $X_s$ is $(C,\tau)$-$cgp$ in $W(\C)$. Then any tuple $\a \in K^n$ lying in 
    $X$ is $cgp$.
  \end{lemma}

  \begin{proof}Let $\b \in K^{<\infty}$ such that $\dimo(\a/\b) < \dimo(\a)$. 
    Then setting $W'=\loc^0(\a/\acl^0(\b))$ we get an absolutely irreducible 
    subvariety of $W$, which is proper, since by $(\ref{dim-loc})$ one has 
    $\dim(W')=\dimo(\a/\b) < \dim W$, and contains $\a$. Let $c$ be the 
    complexity of $W'$. Then for every $\tau>c$, $\U$-many $X_s$ are 
    $(c,\tau)$-$cgp$ in $W(\C)$, and this implies that $\bdl(\a/\acl^0(\b)) \leq 
    \bdl(X \cap W'(K)) =0$. Hence $\bdl(\a/\b)=0$ by $(\ref{alg-ineq})$. 
  \end{proof}

  \begin{remark}\label{cgp-invariance} The property of being $cgp$ for a tuple 
    $\a \in K^m$ depends only on the type $\tp(\a)$ of $\a$. Indeed, suppose 
    $\a'\in K^m$ has the same type, and $\b \in K^n$ is such that $\a' \nind^0 
    \b$. Then there is $\b' \in K^n$ such that $\tp(\a,\b') = \tp(\a',\b)$, by 
    $\aleph_1$-compactness of $K$. So $\a \nind^0 \b'$, and so 
    $\bdl(\a/\b')=0$. But then by invariance of $\bdl$ (see 
    \ref{fact:dlContProps}ii.), we have $\bdl(\a'/\b)=\bdl(\a/\b)=0$.
  \end{remark}

\end{parag}

\begin{parag}{\bf Higher dimensional case: Elekes-Szab\'o with a gap and 
  commutativity.} We now move on to the second theorem proved by 
  Elekes-Szab\'o in \cite{ES-groups}, which is the extension of Theorem 
  \ref{thm:ES} to higher dimensional varieties. We give a proof following the 
  strategy used above in the one-dimensional case. As a payoff we will also 
  get an explicit bound, $\frac{1}{16}$, on the power-saving and we will 
  establish that the group involved must be commutative. This feature (rather 
  the nilpotency) had been hinted at already by Elekes-Szab\'o (see their 
  Example 28 in \cite{ES-groups}), but was first established by H. Wang and 
  the second named author \cite{breuillard-wang} via a different argument 
  using the classification of approximate groups from \cite{bgt}.

  As before we consider three irreducible complex varieties $W_1,W_2,W_3$ of 
  dimension $d$.
  We say that a subvariety $V \subset \prod_i W_i$ admits a power-saving 
  $\eta>0$ if there exists $\tau \in \N$ such that
  $$|V \cap X_1 \times X_2 \times X_3| \leq O_{V,\tau}(N^{\dim(V) - \eta}).$$
  for every $N \in \N$ and all finite subsets $X_i \subset W_i$ with $|X_i|\leq 
  N^d$ and each $X_i$ in coarse $\tau$-general position in $W_i$.

  \begin{theorem} \label{cor:ES}
    Suppose $V \subset W_1\times W_2 \times W_3$ are irreducible complex 
    varieties, and $\dim(W_i)=d$ and $\dim(\pi_{ij}(V)) = 2d = \dim(V)$ for all 
    $i\neq j \in \{1,2,3\}$. Then either $V$ admits a power-saving $\frac{1}{16}$, 
    or $V$ is in co-ordinatewise correspondence with the graph $\Gamma_+ \subset 
    G^3$ of the group operation of a commutative algebraic group $G$.
  \end{theorem}

  \begin{remark} \label{rmk:ESGap2}
    Note that we obtain a power saving which is independent of $d$. In fact the 
    method gives a power-saving $\eta$ for any $\eta<\frac{d}{16d-1}$, which is 
    slightly better. For $d=1$, a power-saving of $\frac16$ was obtained by Wang 
    \cite{Wang-gap} and independently by Raz, Sharir, and de Zeeuw \cite{RSZ-ESR}. 
    The method given below also gives $\frac{1}{6}$ when $d=1$, see Remark 
    \ref{best-exp}.
  \end{remark}

  The remainder of this section is devoted to the proof of Theorem \ref{cor:ES}. 
  As in the one-dimensional case, we first reformulate the result in the 
  framework of coarse dimension using the notation of Section \ref{sec:setup}.

  \begin{theorem}[Reformulation of Theorem \ref{cor:ES}] \label{thm:coherentES}
    Suppose $\a_1,\a_2,\a_3 \in K^d$ with $\dimo(\a_i,\a_j ) = 
    \dimo(\a_1,\a_2,\a_3) = 2d$, $\bdl(\a_i) \leq d$ and $\bdl(\a_1,\a_2,\a_3) 
    \in [2d - \frac{1}{16},2d]$. Assume that each $\a_i$ is $cgp$ in the sense 
    of \ref{cgp-parag}.
    Then there is an $d$-dimensional commutative algebraic group $G$ over $C_0$,
    and $\a'_1,\a'_2,\a'_3 \in G(K)$
    such that $\dimo(\a'_i)=d$, $\acl^0(\a_i)=\acl^0(\a'_i)$ and 
    $\a'_1+\a'_2=\a'_3$.
  \end{theorem}

  \begin{proof}[Reduction of Theorem \ref{cor:ES} to Theorem 
    \ref{thm:coherentES}] This is essentially the same argument as in the 
    one-dimensional case, so we shall be brief. Let $\eta=\frac{1}{16}$. Arguing 
    by contradiction and carefully negating quantifiers we obtain an increasing 
    sequence of integers $(N_s)_{s \ge 0}$ and a sequence of finite sets $X_i^s 
    \subset W_i(\C)$ in coarse $s$-general position with $|X_i^s| \leq N_s^d$ 
    but $|X^s_1 \times X^s_2 \times X^s_3 \cap V| \geq N_s^{2d-\eta}$. Passing 
    to an ultralimit we obtain three internal sets $X_i \subset W_i(K)$, which 
    we add as predicates to our language (thus turning them into definable 
    sets). Clearly $\bdl(X_i)\leq d$ and $\bdl(X_1 \times X_2 \times X_3 \cap V) 
    \geq 2d-\eta$. By Fact \ref{fact:ideal} we find three tuples $\a_i \in X_i$ 
    such that $\bdl(\a_1,\a_2,\a_3) = \bdl(X_1 \times X_2 \times X_3 \cap V)$. 
    By Lemma \ref{cgplemfin} each $\a_i$ is $cgp$. Moreover $(\a_1,\a_2,\a_3)$ 
    is generic in $V$, for otherwise it would be contained in a subvariety $W 
    \subsetneq V$ over $C_0$ forcing $\bdl(\a_1,\a_2,\a_3) \leq \bdl(X_1 \times 
    X_2 \times X_3 \cap W) \leq \dim W\leq2d-1$ (see Lemma 
    \ref{lem:power-saving}). Therefore $\dimo(\a_i,\a_j ) = 
    \dimo(\a_1,\a_2,\a_3) = 2d$ and the assumptions of Theorem 
    \ref{thm:coherentES} are met.
  \end{proof}

  Analogously to the one-dimensional case, the Szemer\'edi-Trotter bounds of 
  Lemma \ref{lem:SzT} will be used to prove the following crucial step, which 
  shows that the $2$-parameter family of varieties $\loc^0(\a/\b)$ is in fact a 
  $1$-parameter family.

  \begin{lemma}\label{1base-m} Let $\eta \in [0,\frac{d}{8d-1})$. Let 
    $\x_1,\ldots,\x_4 \in K^d$, and set $\a=(\x_1,\x_2)$ and $\b=(\x_3,\x_4)$. 
    Assume each $\x_i$ is $cgp$ and $\bdl(\x_i)\leq \dimo(\x_i)=d$. Assume 
    further $\bdl(\a,\b)\geq \dimo(\a,\b) - \eta$ and $\dimo(\a)=\dimo(\b)=2d$ 
    and $\dimo(\a/\b)=d$. Then there is $\x_5 \in \acl^0(\b) ^{<\infty} \cap 
    \acl^0(\a)^{<\infty}$ with $\dimo(\x_5)=\dimo(\a/\x_5)=d$.
  \end{lemma}

  \begin{proof} First note that $\bdl(\b) \geq \dimo(\b) - \eta$. To see this 
    observe that by $cgp$ either $\bdl(\x_2/\b\x_1) = 0$ or $\bdl(\x_1/\b) = 0$. 
    Indeed otherwise $\dimo(\x_2/\b\x_1)=\dimo(\x_2)=d$ and 
    $\dimo(\x_1/\b)=\dimo(\x_1)=d$, contradicting $\dimo(\x_1,\x_2/\b)=d$. So we 
    have $\bdl(\a/\b) = \bdl(\x_2/\b\x_1) + \bdl(\x_1/\b) \leq d = 
    \dimo(\a/\b)$. And so $\bdl(\b) = \bdl(\a,\b) - \bdl(\a/\b)\ge 3d-\eta -d 
    =\dimo(\b) - \eta$.

    Let $\d \in K^{<\infty}$ be a canonical base for $\a$ over $\b$ (see \S 
    \ref{parag:canonical}). By definition $\d \in \acl^0(\b)^{< \infty}$. We will 
    show that this $\x_5:=\d$ satisfies the desired conclusion. By definition we 
    have $\dimo(\a/\d)=\dimo(\a/\b)=d$. Hence 
    $\dimo(\d)=\dimo(\a,\d)-\dimo(\a/\d)\ge \dimo(\a)-d=d$. 

    It only remains to show the upper bound $\dimo(\d) \leq d$. Indeed this will 
    also imply that $\d \in \acl^0(\a)^{<\infty}$, because then $\dimo(\d)=d$ and 
    $\dimo(\a\d)=\dimo(\a/\d)+\dimo(\d)=2d$, so $\dimo(\a\d)=\dimo(\a)$ and thus 
    $\dimo(\d/\a)=0$.

    So suppose that $\dimo(\d) > d$. Then $\dimo(\b/\d)=\dimo(\b) -\dimo(\d)<d$. 
    Writing $$\dimo(\b/\d) =\dimo(\x_3/\d) +\dimo(\x_4/\x_3,\d),$$
    each term is $<d$. Since each $\x_i$ is $cgp$, this forces $\bdl(\b/\d)=0$. 
    Consequently $\bdl(\d)= \bdl(\b\d)$. But $\bdl(\d/\b)=0$ by 
    $(\ref{alg-ineq})$. So $\bdl(\d)=\bdl(\b)$.

    However $\bdl(\b) \ge \dimo(\b)-\eta=2d-\eta$. Using the Szemer\'edi-Trotter 
    bound of Lemma \ref{lem:SzT} we will obtain an upper bound on $\bdl(\d)$ 
    contradicting this lower bound.

    By the primitive element theorem, the field extension $C_0(\d)$ is generated 
    over $C_0$ by a tuple of length $\dimo(\d)+1$.
    So we may assume $\d \in K^{\dimo(\d)+1}$.
    Assume first that $\dimo(\d) < 2d$, so $\max(2d,\dimo(\d)+1) = 2d$;
    we handle the case $\dimo(\d)=2d$ separately below.

    Now let $V=\loc^0(\a\d)$, $X:=(X_1 \times X_2) \cap V$, $X_1=\tp(\a)(K)$ and 
    $X_2=\tp(\d)(K)$. Note that $X_1 \subset K^{2d}$ and $X_2 \subset 
    K^{\dimo(\d)+1}$. First we check that $\bdl(X(\d_1) \cap X(\d_2))=0$ when 
    $\d_1 \neq \d_2$ belong to $X_2$, so that $\beta=0$ in Lemma \ref{lem:SzT}.

    Recall that $X(\d_1):=\{\a' \in X_1 : \a'\d_1 \in V\}$. By Fact 
    \ref{fact:ideal} we may find $\a'=\a'_1\a'_2 \in X(\d_1)\cap X(\d_2)$ such 
    that $\bdl(\a'/\d_1\d_2)=\bdl(X(\d_1)\cap X(\d_2))$. Then $\a'\d_i \in V$ for 
    $i=1,2$. We are thus in the setting of Lemma \ref{lem:cb}. Since $\d_1 \neq 
    \d_2$ we conclude that $\dimo(\a'/\d_1\d_2)<d$. Hence 
    $\dimo(\a'_i/\d_1\d_2)<d$, and since both $\a_i'$ are $cgp$ (because $cgp$ is 
    type invariant, see Remark \ref{cgp-invariance}), we conclude that 
    $\bdl(\a'_i/\d_1\d_2)=0$ for each $i$. Thus by additivity $\bdl(X(\d_1)\cap 
    X(\d_2))=\bdl(\a'/\d_1\d_2)=0$ as claimed.

    The Szemer\'edi-Trotter bound of Lemma \ref{lem:SzT} then gives the bound
    \[ \bdl(X) \leq \max \left(\frac{1}{2}\bdl(X_1) + \bdl(X_2) - \epsilon_0 
    \big(\bdl(X_2) - \frac{1}{2}\bdl(X_1)\big)^{+},\bdl(X_1), \bdl(X_2) \right) 
    ;\]
    using the incidence bound of \cite[Theorem~1.2]{FoxPachEtAl} mentioned 
    after Theorem~\ref{st}, we obtain this bound for
    all $\epsilon_0 \in [0,\frac{1}{8d-1})$. If the maximum on the right hand 
    side is $\bdl(X_1)$, then $\bdl(X_2) \leq \frac{1}{2}\bdl(X_1)\leq d$, 
    contradicting our lower bound $\bdl(X_2)=\bdl(\d)\geq 2d - \eta$ above. If 
    the maximum is $\bdl(X_2)$, then $\bdl(X)=\bdl(\a\d)\leq \bdl(\d)$, so 
    $\bdl(\a/\d)=0$ and hence $\bdl(\a/\b)=0$ by $(\ref{alg-ineq})$. But 
    $\bdl(\a/\b)=\bdl(\a\b)-\bdl(\b) \geq d-\eta>0$. 

    So we must conclude that 

    $$\bdl(\a/\b) + \bdl(\d) \leq \bdl(X) \leq \frac{1}{2}\bdl(\a) + \bdl(\d) - 
    \epsilon_0 \big(\bdl(\d) - \frac{1}{2}\bdl(\a)\big).$$

    In other words, since $\bdl(\d)=\bdl(\b)$:

    $$\frac{1}{\epsilon_0}\bdl(\a\b) \leq \frac{1}{2} 
    \bdl(\a)(\frac{1}{\epsilon_0} +1) + \bdl(\b)( \frac{1}{\epsilon_0} -1).$$

    Finally since $\bdl(\a), \bdl(\b) \leq 2d$ and $\bdl(\a\b) \geq 3d - \eta$ by 
    assumption, we conclude that 

    $$\eta \geq \frac{d}{8d-1}$$
    a contradiction to our assumption.

    We assumed above that $\dimo(\d) < 2d$, and so we conclude from this 
    contradiction that $\dimo(\d) = d$ or $\dimo(\d) = 2d$. it remains to rule out 
    the latter case. If we were willing to weaken our bound to 
    $\frac{d}{4(2d+1)-1}$, the argument above would suffice. To obtain the sharper 
    bound, we use an additional argument inspired by \cite[Theorem~4.3]{RSZ-ESR}.

    So assume $d^0(\d) = 2d$, i.e.\ $\acl^0(\d)=\acl^0(\b)$.

    Let $W := \loc^0(\a\b)$
    and $V := \loc^0(\a\d)$.
    Write $V_{\y}$ for $\{ \x : (\x,\y) \in V \}$, and similarly for $W_{\y}$.
    So $V_{\d} = \loc^0(\a/\d)$ is the irreducible component of $W_{\b} = 
    \loc^0(\a/\b)$ containing $\a$.
    Now $\tp^0(\d/\b)(K)$ is finite, so say $\tp^0(\d/\b)(K) = \{\d_1,\ldots 
    ,\d_k\}$.
    Then $V_{\d_i}$ are precisely the irreducible components of $W_{\b}$, and 
    $W_{\b} = \bigcup_i V_{\d_i}$.
    Indeed, any $V_{\d_i}$ is an automorphic image over $C_0(\b)$ of $V_{\d}$ and 
    so is a component, and conversely $\bigcup_i V_{\d_i}$ is 
    automorphism-invariant over $C_0(\b)$ and hence is defined over $C_0(\b)$,
    so $W_{\b} \subset \bigcup_i V_{\d_i}$.

    \begin{claim} \label{clm:finDegree}
      There are only finitely many $\b' \in \tp^0(\b)(K)$ with
      $\dim(W_{\b} \cap W_{\b'}) = d$.
    \end{claim}
    \begin{proof}
      Fix field automorphisms $\tau_{ij} \in \Aut(K/C_0)$
      such that $\tau_{ij}(\d_i) = \d_j$. Suppose $\b' \in \tp^0(\b)(K)$ and 
      $\operatorname{dim}(W_{\b} \cap W_{\b'}) = d$.
      Then $W_{\b}$ and $W_{\b'}$ share a component,
      so say $V_{\d_j} \subset W_{\b'}$. Now let $\sigma \in \Aut(K/C_0)$ be such 
      that $\b' = \sigma(\b)$.
      Since the $\sigma(V_{\d_i})$ are the irreducible components of $W_{\b'}$,
      for some $i$ we have $\sigma(V_{\d_i}) = V_{\d_j}$.

      Let $\sigma' := \tau_{ij}^{-1} \sigma$.
      Then $\sigma'(V_{\d_i}) = \tau_{ij}^{-1} V_{\d_j} = V_{\tau_{ij}^{-1} \d_j} 
      = V_{\d_i}$,
      so by the definition of canonical base,
      we have $\sigma'(\d_i) = \d_i$. Then since $\b \in \acl^0(\d_i)$,
      there are only finitely many possibilities for $\sigma'(\b)$,
      and so there are only finitely many possibilities for
      $\b' = \sigma(\b) = \tau_{ij}(\sigma'(\b))$.
    \end{proof}

    Now let $G$ be the graph with vertex set $\tp(\b)(K)$ and with an edge between 
    $\b'$ and $\b''$ if and only if $\dim(W_{\b'} \cap W_{\b''}) = d$. By 
    Claim~\ref{clm:finDegree}, $G$ has constant finite degree.

    \begin{claim}\label{clm:bigAnticlique}
      If $G=(A,E)$ is a graph where the vertex set $A$ is $\bigwedge$-internal and 
      the edge relation $E$ is internal, and if $G$ has finite maximal degree $k$, 
      then there is a $\bigwedge$-internal anticlique $A'$ with 
      $\bdl(A')=\bdl(A)$.

    \end{claim}
    \begin{proof}
      If $A$ is internal, then $G$ is the ultraproduct of finite graphs $G_i = 
      (A_i,E_i)$ of maximal degree $k$. Then $G_i$ has chromatic number at most 
      $k+1$, and so has an anticlique of size at least $ \frac{|A_i|}{k+1}$.
      The ultraproduct of such anticliques is then an internal anticlique $A' 
      \subset A$ as required.

      In general, our $\bigwedge$-internal $A$ is, by $\aleph_1$-compactness of 
      the ultraproduct, contained in an internal $A_0$ such that $(A_0,E)$ has 
      maximal degree at most $k$, because the property of having maximal degree at 
      most $k$ can be expressed as the inconsistency of a partial $(k+1)$-type. So 
      then the same holds for all internal $A_1$ with $A \subset A_1 \subset A_0$, 
      and hence the claim follows from the internal case.
    \end{proof}

    Now let $X_2 \subset \tp(\b)(K)$ be an anticlique as in 
    Claim~\ref{clm:bigAnticlique} for the graph defined above, and $X_1 := 
    \tp(\a)(K)$ and $X := (X_1 \times X_2) \cap W$. If $\a' \in X(\b_1) \cap 
    X(\b_2)$ then $\dimo(\a'/\b_1\b_2) \leq \dim(W_{\b_1}\cap W_{\b_2}) < d$ since 
    $X_2$ is an anticlique, and so $\bdl(\a')=0$ by $cgp$. So we contradict the 
    Szemer\'edi-Trotter bound exactly as in the case $d<\dimo(\d)<2d$ above.

    This contradiction shows that $\dimo(\d)=d$ and ends the proof.
  \end{proof}


  \begin{proof}[Proof of Theorem \ref{thm:coherentES}]Here again the strategy is 
    the same as in the $1$-dimensional case, so we shall be brief. Let 
    $\eta=\frac{1}{16}$.

    As before set $X=\tp(\a_2,\a_3/\a_1)(K)$ and note that $\bdl(\a_1,\a_2,\a_3) = 
    \bdl(X)+\bdl(\a_1)$ by additivity of $\bdl$. It follows that $\bdl(X) \geq 
    d-\eta$.

    By Fact \ref{fact:ideal} we may find $(\a_4,\a_5) \in X$ with 
    $\bdl(a_4,a_5/a_1,a_2,a_3) = \bdl(X)$. Note that $\a_4,\a_5$ are both $cgp$ 
    (see Remark \ref{cgp-invariance}). We will show that there are $\a_6$ and 
    $\a_7$ in $K^d$ such that $(\a_1,\ldots,\a_7)$ satisfy the hypotheses of the 
    abelian group configuration theorem as in the following diagram:
    
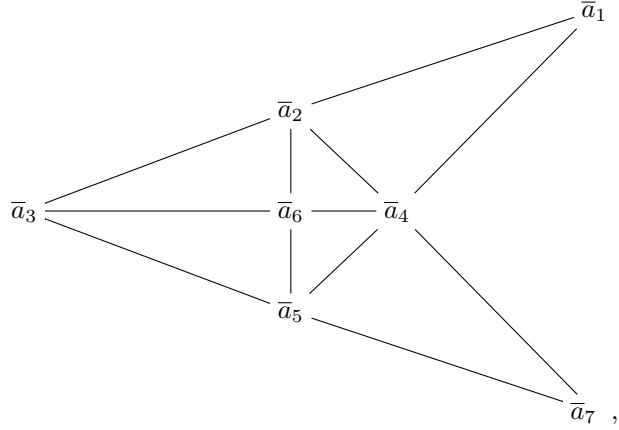
\[ \xymatrix{
    & & &&&& \a_1 \ar@{-}'[ddll][dddlll] \ar@{-}'[dlll][ddllllll] \\
    & & & \a_2 \ar@{-}'[d][dd] & & & \\
    \a_3 & & & \a_6 & \a_4 \ar@{-}'[l][llll] & & \\
    & & & \a_5 && & \\
    && & & & & \a_7\;\; \ar@{-}'[uull][uuulll] \ar@{-}'[ulll][uullllll] } \]
    As earlier we have that $\dimo(\a_4,\a_5,\a_1)=2d$, 
    $\dimo(\a_4,\a_5)=\dimo(\a_4,\a_1)=\dimo(\a_5,\a_1)=2d$. And we also have 
    $\dimo(\a_1,\ldots,\a_5)=3d$. Indeed otherwise 
    $\dimo(\a_4,\a_5/\a_1,\a_2,\a_3)<d$, which implies that 
    $\dimo(\a_i/\a_1,\a_2,\a_3)< d$ for $i=4,5$ and thus 
    $\bdl(\a_i/\a_1,\a_2,\a_3)=0$ since $\a_4,\a_5$ are $cgp$. But this 
    contradicts $\bdl(\a_4,\a_5/\a_1,\a_2,\a_3) = \bdl(X)>0$.

    Again as in the $1$-dimensional case, using only the additivity of $\dimo$ we 
    conclude that $\dimo(\a_3,\a_4)=\dimo(\a_2,\a_5)=2d$ and 
    $\dimo(\a_2,\a_5/\a_3,\a_4)=d$, and also that 
    $\dimo(\a_2,\a_4)=\dimo(\a_3,\a_5)=2d$ and $\dimo(\a_3,\a_5/\a_2,\a_4)=d$.

    Moreover $\bdl(\a_2,\a_3)=\bdl(\a_1,\a_2,\a_3)$ by additivity since $\a_1$ is 
    $cgp$. Similarly by additivity
    $$\bdl(\a_2,\a_3,\a_4,\a_5)=\bdl(\a_1,\ldots,\a_5)=\bdl(X) + 
    \bdl(\a_1,\a_2,\a_3),$$
    hence this is $\ge 3d- 2\eta$. Since $2\eta \leq \frac18 < \frac{d}{8d-1}$, we 
    are thus in a position to apply Lemma \ref{1base-m} to both 
    $\a_2,\a_5,\a_3,\a_4$ and to $\a_2,\a_4, \a_3,\a_5$. This yields $\a_6$ and 
    $\a_7$ as desired and concludes the proof of the theorem.

  \end{proof}

  \begin{remark}[Quality of the power-saving]\label{best-exp} The quality of the 
    power-saving depends crucially on the quality of $\epsilon_0$ in the 
    Szemeredi-Trotter type bound of Lemma \ref{lem:SzT}. We immediately lose a 
    factor $2$ because this bound is usually proven for real algebraic varities, 
    while we consider complex varieties and somewhat carelessly view them as 
    real varieties of twice the dimension. It is plausible that the bound 
    $\epsilon_0=\frac{1}{2n-1}$ holds in Theorem \ref{st}. In fact this is known 
    when $n=1$ (see \cite{zhal-szabo-sheffer} or \cite[Thm 4.3]{RSZ-ESR}), and 
    consequently Lemma \ref{1base-m} holds for all $\eta \in [0,\frac13)$ when 
    $d=1$ and thus yields a power-saving $\eta$ for any $\eta<\frac16$ in the 
    $1$-dimensional Elekes-Szab\'o theorem. This recovers the bound obtained in 
    \cite{Wang-gap} and \cite{RSZ-ESR}. The latter work however gave more 
    precise information on the multiplicative constant and the dependence on the 
    degree of the variety $V$, which is an aspect we do not investigate in our 
    paper (it would require working with the Hrushovski-Wagner fine 
    pseudo-finite dimension while we restrict attention to the coarse dimension 
    $\bdl$).
  \end{remark}

  The following corollary indicates the robustness of the commutativity of the 
  group in Theorem~\ref{cor:ES}.

  \begin{corollary} \label{specialAbelian}
    Suppose $(G;\cdot)$ is a connected complex algebraic group.
    Suppose the graph $\Gamma$ of multiplication admits no power-saving.
    Then $G$ is commutative.
  \end{corollary}
  \begin{proof}
    By Theorem~\ref{cor:ES}, $\Gamma$ is in co-ordinatewise correspondence with 
    the graph $\Gamma_+ \subset G'^3$ of the group operation of a commutative 
    connected algebraic group $G'$.
    So this is an immediate consequence of Fact~\ref{fact:corrIsog}.
  \end{proof}

  \begin{remark} Another proof of this corollary was noted in 
    \cite{breuillard-wang}. It can be derived as a consequence of the 
    Balog-Szemeredi-Gowers-Tao theorem combined with \cite[Theorem 2.5]{bgt}, 
    one of the main results of \cite{bgt} which was proven there for linear 
    algebraic groups, but can be extended to all algebraic groups.
  \end{remark}

  In the following sections we will handle the general case of a cartesian 
  product of an arbitrary number, say $n$, of subvarieties. As in the 
  reformulations of Elekes-Szab\'o's statements expounded above it is easy to 
  see that a subvariety without power-saving leads to a tuple 
  $(\a_1,\ldots,\a_n)$ such that each $\a_i$ is $cgp$, belongs to $K^d$ and has 
  $\bdl(\a_i) \leq d$, such that 
  $$\bdl(\a_1,\ldots,\a_n)=\dimo(\a_1,\ldots,\a_n).$$
  In Sections \ref{sec:proj} and \ref{sec:varieties} we will forget for a moment 
  the original combinatorial problem and focus entirely on the study of these 
  tuples. Then in Section \ref{sec:asymptotic} we will return to combinatorics 
  and give a proof of Theorem \ref{thm:main}.

\end{parag}

\section{Necessity of general position}
\label{gpNecessity}

We give an example showing that Theorem~\ref{cor:ES} fails dramatically if we 
weaken too far the coarse general position assumption in the definition of 
power-saving. Indeed varieties which are not even in correspondence with a 
group operation can then have no power-saving even when the finite sets are 
assumed to be say in weak general position, namely assuming $\bdl$ always 
drops when there is an algebraic dependence.


Define an operation $* : \C^2 \times \C^2 \rightarrow \C^2$ by
\[ (a_1,b_1) * (a_2,b_2) = (a_1+a_2+b_1^2b_2^2,b_1+b_2) ,\]
and let $V := \Gamma_*$ be its graph $\{ (x,y,x*y) | x,y \in \C^2 \}$.

Set $X_N := \{0,\ldots ,N^4-1\} \times \{0,\ldots ,N-1\} \subset \C^2$.
Then $|X_N^3 \cap \Gamma_*| \geq \Omega(|X_N|^2)$;
indeed, if $a_i < \frac 13 N^4$ and $b_i < \frac 12 N$ for $i=1,2$,
then $(a_1,b_1)*(a_2,b_2) \in X_N$.

Now $|X_N \cap (\{0\}\times\C)| = N = |X_N|^{\frac15}$, so $X_N$ is not 
$6$-$cgp$. Nevertheless $X_N$ is still in weak general position, namely $|X_N 
\cap W| =O_W(N^4)$ for every algebraic curve $W \leq \C^2$ and $N^4\leq 
|X_N|^{1-\frac{1}{5}}$. But if we were to remove the $cgp$ assumption in the 
definition of power-saving, then $(X_N)_N$ witnesses that $\Gamma_*$ would 
admit no power-saving.

However, one can show that $\Gamma_*$ is not in co-ordinatewise correspondence 
with the graph of the group operation of any complex algebraic group 
$(G;\cdot)$.
We sketch a proof of this.

Suppose such a group $(G;\cdot)$ and a correspondence exist, defined over a 
finitely generated field $A \leq \C$. Then if we take independent generics 
$x_1,x_2,y_1,y_2 \in \C^2$ over $A$ and set $z_{ij} := x_i * y_j$, then 
$z_{22}$ lies in the algebraic closure $\acl(z_{11},z_{12},z_{21})$ of 
$A(z_{11},z_{12},z_{21})$.
This follows from the equation in the algebraic group $G$
\[ x_2'\cdot y_2' = (x_2'\cdot y_1')(x_1'\cdot y_1')^{-1}(x_1'\cdot y_2');\]
cf. \cite[6.2]{HZ} and \cite[Theorem~41]{Tao-expanding} where a converse to 
this is proven in the 1-dimensional case.

But if one takes $z_{11},z_{12},z_{21},x_2$ independent generics and 
calculates in order $y_1,x_1,y_2,z_{22}$ using the definition of $*$, then
$x_1,x_2,y_1,y_2$ are also independent generics,
and writing $z_{11} = (z_{11}',z_{11}'')$ and so on, one obtains
\begin{align*} z'_{22} &= z_{21}'+z_{12}'-z_{11}'- x_2''^2(z_{21}''-x_2'')^2 
  +(z_{11}''-z_{21}''+x_2'')^2(z_{21}''-x_2'')^2\\
  & -(z_{11}''-z_{21}''+x_2'')^2(z_{12}''-z_{11}''+z_{21}''-x_2'')^2,\\
  z''_{22}&= z_{21}''+z_{12}''-z_{11}'' ,
\end{align*}
and e.g.\ $x_2''z_{11}''z_{21}''z_{12}''$ has a non-zero coefficient,
and so $z_{22}$ is not independent from $x_2$;
but $z_{11},z_{12},z_{21}$ is independent from $x_2$ by assumption,
so $z_{22} \notin \acl(z_{11},z_{12},z_{21})$.

%

\section[Projective geometries]{Projective geometries arising from varieties 
without power-saving}\label{sec:proj}

As hinted in Section~\ref{sec:warmup}, the proof of our main results will rely 
on the study of $cgp$ tuples from $K^m$ whose algebraic dimension $\dimo$ 
coincides with their pseudo-finite coarse dimension $\bdl$. In this section we 
study the geometry underlying these tuples, and prove Theorem 
\ref{cor:coherentGeom} below, which establishes that the associated geometry 
is modular and hence satisfies the Veblen axiom making it, via the 
Veblen-Young co-ordinatisation theorem, a sum of projective geometries over 
division rings.

\subsection{Geometries and modularity}\label{subsect:geoms}
We begin with recalling some classical terminology and basic results regarding 
abstract projective geometry (see \cite{Artin-geometricAlgebra, Cameron-pps}) 
and the general notions of pregeometries, geometries and modularity.

\begin{definition*}
  A \defnstyle{closure structure} is a set $P$ with a map $\cl : \powerset(P) 
  \rightarrow \powerset(P)$ such that for $A,B \subset P$ we have $A \subset 
  \cl(A)$, $A \subset B \Rightarrow \cl(A) \subset \cl(B)$, and $\cl(\cl(A)) = 
  \cl(A)$.
  A subset of $P$ is \defnstyle{closed} if it is in the image of $\cl$.
  A closure structure $(P,\cl)$ is a \defnstyle{pregeometry} if the following 
  two properties also hold:
  \begin{itemize}\item Exchange: $b \in \cl(A \cup \{c\}) \setminus \cl(A) 
      \Rightarrow c \in \cl(A \cup \{b\})$;
    \item Finite character: $\cl(A) = \bigcup_{A_0 \subset A,\; A_0 \textrm{ 
      finite}} \cl(A_0)$.
  \end{itemize}

  Finite pregeometries are also known as \defnstyle{matroids}.

  Let $(P,\cl)$ be a pregeometry. For $A,B \subset P$, a \defnstyle{basis} for 
  $A$ over $B$ is a subset $A'\subset A$ of minimal size such that $\cl(A' 
  \cup B) = \cl(A \cup B)$. Any two bases have the same cardinality, which is 
  denoted by $\dim(A/B)$ and called the \defnstyle{dimension} of $A$ over $B$. 
  When $B$ is empty, this is the dimension of $A$, which we denote as usual by 
  $\dim(A)$.


  Subsets $A,B \subset P$ are \defnstyle{independent} over a subset $C \subset 
  P$, written $A \ind_C B$, if $\dim(A'/B\cup C)=\dim(A'/C)$ for any finite 
  $A' \subset A$.


  A pregeometry $(P,\cl)$ is a \defnstyle{geometry} if $\cl(\emptyset 
  )=\emptyset $ and $\cl(\{a\})=\{a\}$ for $a \in P$.
  Every pregeometry gives rise to a geometry by projectivisation: the 
  \defnstyle{projectivisation} of a closure structure $(P,\cl)$ is the closure 
  structure $\P(P,\cl)$ with points $\{ \cl(\{a\}) : a \notin \cl(\emptyset ) 
  \}$ and the induced closure.
  If $(P,\cl)$ is a pregeometry, then $\P(P,\cl)$ is the \defnstyle{associated 
  geometry}.

  A geometry $(P,\cl)$ is said to be \defnstyle{modular} if for distinct 
  $a_1,a_2 \in P$ and $B \subset P$,
  and if $a_2 \in \cl(\{a_1\}\cup B) \setminus \cl(B)$,
  then there exists $d \in \cl(B)$ such that $d \in \cl(\{a_1,a_2\})$.

  A pregeometry is modular if its associated geometry is modular.
  Equivalently (\cite[Lemma~8.1.13]{Marker-MT}), for any closed sets $A,B$,
  $$\dim(A\cup B) = \dim(A) + \dim(B) - \dim(A\cap B).$$

  Points $a_1,a_2$ of a geometry $(P,\cl)$ are \defnstyle{non-orthogonal} if 
  there exists $B \subset P$ such that $a_2 \in \cl(\{a_1\}\cup B)\setminus 
  \cl(B)$.

  A \defnstyle{subgeometry} of a geometry $(P,\cl)$ is the restriction 
  $(Y,\cl\negmedspace\restriction_Y)$ to a subset $Y \subset P$, where 
  $\cl\negmedspace\restriction_Y(A): = \cl(A) \cap Y$.

  The \defnstyle{sum} of geometries $(P_i,\cl_i)$ is the non-interacting 
  geometry on the disjoint union, namely the geometry $(\Disjointunion_i P_i, 
  \cl)$ where $\cl(\Disjointunion_i A_i) := \Disjointunion_i \cl_i(A_i)$.
\end{definition*}

The proofs of all claims made in the above definitions are straightforward and 
classical. We refer the reader to \cite[Appendix~C.1]{TentZiegler} for them 
and further details.

\begin{example}[Projective spaces over division rings]
  If $V$ is a vector space over a division ring $F$,
  then $V$ equipped with $F$-linear span forms a pregeometry $(V,\left<{\cdot 
  }\right>_F)$ of dimension $\dim(V)$,
  and the associated geometry $\P(V) := \P(V, \left<{\cdot}\right>_F)$ is the 
  projective space of $V$; it also has dimension $\dim(V)$, and it is a 
  modular geometry.
\end{example}
\begin{example}[Algebraic closure]
  An algebraically closed field $K$ equipped with field-theoretic algebraic 
  closure over the prime field forms a pregeometry $(K,\acl)$. The dimension 
  is the transcendence degree over the prime field. If $\dim(K) \geq 3$ then 
  the associated geometry is not modular, as can be seen by considering a 
  generic solution to $b=c_1a+c_2$, see \cite[Appendix~C.1]{TentZiegler}.
\end{example}

\begin{example}[Algebraic closure on tuples]\label{alg-tuples}
  If $C_0 \leq K$ are algebraically closed fields, the set of all tuples 
  $K^{<\infty}$ equipped with the algebraic closure $\acl^0$ over $C_0$ forms a 
  closure structure, where the closure of a subset $A \subset K^{<\infty}$ is 
  $\acl^0(A)^{<\infty}$ as defined in $(\ref{closalg})$. But it is in general 
  not a pregeometry.
\end{example}

In the sequel we will only consider closure operators of the types described 
in the above examples.


Let $(P,\cl)$ be a modular geometry.
Then $a,b \in P$ are non-orthogonal
if and only if there exists $c \in P \setminus \{a\}$ such that $a \in 
\cl(\{b,c\})$.
In other words, $a=b$ or $|\cl(\{a,b\})|>2$. It is easy to see from modularity 
that this is an equivalence relation.

\begin{example}[Abstract projective space] An \defnstyle{abstract projective 
  space} is a pair $(P,L)$ of sets, where $P$ is the set of points and $L$ 
  the set of lines, a unique line passes through every two distinct points, 
  every line has at least three points and \emph{the Veblen axiom} holds: 
  given four distinct points $a,b,c,d$, if the lines $ab$ and $dc$ intersect, 
  then so do $ad$ and $bc$. Any such abstract projective space gives rise to a modular 
  geometry on $P$ in which the closure of a subset is the union of all lines passing 
  through two points in the subset. Conversely any modular geometry in which 
  every pair of points is non-orthogonal gives rise to an abstract projective 
  space with the same set of points and with the set of lines being the set of 
  closures of pairs of distinct points.
\end{example}


We now recall the classical Veblen-Young co-ordinatisation theorem of 
projective geometry, which characterises modular geometries.

\begin{fact} \label{fact:modStructure}
  If $(P,\cl)$ is a modular geometry,
  and every two points $a,b \in P$ are non-orthogonal,
  and $\dim(P) \geq 4$,
  then $P$ is isomorphic to a projective space $\P(V)$,
  where $V$ is a vector space over a division ring.

  More generally,
  if $(P,\cl)$ is a modular geometry,
  then non-orthogonality is an equivalence relation,
  and $(P,\cl)$ is the sum of the subgeometries on its non-orthogonality 
  classes,
  each of which either has dimension $\leq 2$, or is a projective space over a 
  division ring, or is a non-Desarguesian projective plane.
\end{fact}
\begin{proof}
  This is a consequence of the classical Veblen-Young co-ordinatisation 
  theorem for projective geometries. Veblen's axiom is a direct consequence of 
  modularity.
  We refer to \cite[Theorem~3.6]{Cameron-pps} for a statement which directly 
  implies the stated result and for an overview of the proof,
  and to \cite[Chapter~II]{Artin-geometricAlgebra} for a detailed proof of the 
  co-ordinatisation theorem for Desarguesian projective planes.
\end{proof}

In our applications the geometries will be modular and infinite dimensional. 
So by the above they will be sums of projective geometries over division 
rings.

\subsection{Coarse general position, coherence and modularity}

We recall the notion of coarse general position for tuples introduced in \S 
\ref{cgp-parag}. We keep the notation and setup of Section \ref{sec:setup}.

\begin{definition} \label{defn:cgpType}
  A tuple $\c \in K^{<\infty}$ is said to be in \defnstyle{coarse general 
  position}
  (or \defnstyle{is cgp})
  if for any $B \subset K^{<\infty}$,
  if $\c \nind^0 B$ then $\bdl(\c/B)=0$.
\end{definition}

Recall that $K^{<\infty}$ is the set of all tuples of elements from $K$ and $c 
\nind^0 B$ means that $\dimo(\c/B)<\dimo(\c)$, where $\dimo(\c/B)$ denotes as 
earlier the transcendence degree of the tuple $\c$ over the field $C_0(B)$ 
field generated by all co-ordinates of elements from $B$, where $C_0\leq K$ is 
the base field defined in \ref{notn:baseField}. The coarse dimension $\bdl$ 
was defined in $(\ref{defdel})$.




\begin{definition}
  A subset $P \subset K^{<\infty}$ is said to be \defnstyle{cgp-coherent} if 
  every $\a \in P$ is $cgp$ and $\bdl(\a_1,\ldots,\a_n) = 
  \dimo(\a_1,\ldots,\a_n)$ for all choices of $\a_1,\ldots,\a_n \in P$.
\end{definition}

In this paper, we abbreviate `cgp-coherent' to just `\defnstyle{coherent}'.

We will also say that a tuple of tuples from $K^{<\infty}$ is coherent when 
the set of its elements is coherent.

\begin{remark}The term ``coherent'' is borrowed from 
  \cite[Section~5]{Hr-psfDims}, where it is used in a parallel context to 
  refer to the same idea that a pseudo-finite dimension notion is in accord 
  with transcendence degree.
\end{remark}



We are now ready to state the main result of this section.

\begin{theorem} \label{cor:coherentGeom}
  Suppose $P \subset K^{<\infty}$ is coherent. Then 
  $(P;\acl^0\negmedspace\restriction_{P})$ is a pregeometry. Moreover $P$ 
  extends to a coherent $P' \subset K^{<\infty}$ such that the geometry 
  $\G_{P'} := \P(P';\acl^0\negmedspace\restriction_{P'})$ is a sum of 
  $1$-dimensional geometries and infinite dimensional projective geometries 
  over division rings.
\end{theorem}

Here the closure operator is simply the restriction to $P'$ of the algebraic 
closure $\acl^0$ as in Example \ref{alg-tuples}, namely if $A \subset P'$, 
$\acl^0\negmedspace\restriction_{P'}(A)$ is the set of tuples in $P'$ whose 
co-ordinates are algebraic over the subfield of $K$ generated by $C_0$ and the 
set of all co-ordinates of all tuples from $A$.

The proof of Theorem \ref{cor:coherentGeom} will proceed in two steps. First 
we will show that if $P \subset K^{<\infty}$ is coherent, then its coherent 
algebraic closure $\ccl(P): = \{ \x \in \acl^0(P)^{<\infty} : \x \textrm{ is 
cgp and }   \bdl(\x) = \dimo(\x) \}$ is also coherent. And second we will 
prove that if $P=\ccl(P)$ is coherent, then 
$(P;\acl^0\negmedspace\restriction_P)$ is a modular pregeometry. The latter 
step will use the incidence bounds \`a la Szemer\'edi-Trotter recalled in 
Section \ref{subsec:incidence}. Theorem \ref{cor:coherentGeom} will then 
follow by applying the Veblen-Young theorem recalled in Fact 
\ref{fact:modStructure} above to the projectivisation of 
$(P;\acl^0\negmedspace\restriction_P)$.

The rest of this section is devoted to the proof of Theorem 
\ref{cor:coherentGeom}.

\begin{parag}{\bf Properties of coherent sets.}

  \begin{proposition}\label{pregeo}
    If $P \subset K^{\infty}$ is coherent then 
    $(P;\acl^0\negmedspace\restriction_P)$ is a pregeometry.
  \end{proposition}
  \begin{proof}
    We verify exchange, the other properties being immediate.
    Suppose $\b \in \acl^0(A \cup \{\c\}) \setminus \acl^0(A)$ for $A \subset P$ 
    and $\b,\c \in P$. So $\b \nind^0_{A} \c$. By symmetry of $\ind^0$ (see \S 
    \ref{ind}), we get $\c \nind^0_{A} \b$. Now the next lemma forces $\c \in 
    \acl^0(A \cup \{\b\})$.
  \end{proof}

  \begin{lemma}\label{lem:cohFullEmb} Suppose $P \subset K^{<\infty}$ is 
    coherent. Let $\c \in P$ and $A,B \subset P$. Then either $\c \ind^0_A B$, 
    or $\c \in \acl^0(A \cup B)$.
  \end{lemma}

  \begin{proof}
    Suppose $\c \nind^0_A B$.
    Then in particular $\c \nind^0 (B \cup A)$,
    and so $\c \nind^0 \b\a$ for some tuples $\b \in B^{<\infty}$ and $\a \in 
    A^{<\infty}$.
    Since $\c$ is $cgp$ this implies $\bdl(\c/\b\a)=0$. Since $P$ is coherent, 
    $\dimo(\c\b\a)=\bdl(\c\b\a)$ and $\dimo(\b\a)=\bdl(\b\a)$ so by additivity 
    of $\dimo$ and $\bdl$ we get $\dimo(\c/\b\a)=0$. Hence $\c \in \acl^0(A 
    \cup B)$.
  \end{proof}

  The next lemma will be used to form coherent sets.

  \begin{lemma}\label{lem1} Let $\a_1,\ldots,\a_n \in K^{<\infty}$. Assume that 
    each $\a_i$ is $cgp$ and $\bdl(\a_i) \leq \dimo(\a_i)$. Then for every $C 
    \subset K^{<\infty}$ we have:
    \begin{equation}\label{upb}\bdl(\a_1,\ldots,\a_n/C) \leq 
      \dimo(\a_1,\ldots,\a_n/C).
    \end{equation}
    Moreover $\bdl(\a_1,\ldots,\a_n)= \dimo(\a_1,\ldots,\a_n)$ if and only if 
    $\{\a_1,\ldots,\a_n\}$ is coherent.
  \end{lemma}

  \begin{proof} The proof is by induction on $n$. Suppose first $n=1$. We have a 
    $cgp$ $\a_1 \in K^{<\infty}$ such that $\bdl(\a_1) \leq \dimo(\a_1)$ and we 
    need to show that $\bdl(\a_1/C) \leq \dimo(\a_1/C).$ If $\a_1 \ind^0 C$, 
    then $\dimo(\a_1/C)=\dimo(\a_1)$, so the desired inequality follows 
    immediately. On the other hand, if $\a_1 \nind^0 C$, then by $cgp$ 
    $\bdl(\a_1/C)=0$, so the desired inequality is then obvious.

    Suppose $(\ref{upb})$ holds for $n-1$ tuples and any $C$. Let $\x=\a_1 \ldots 
    \a_{n-1}$. $$\bdl(\x\a_n/C) = \bdl(\x/C \cup\{\a_n\}) + \bdl(\a_n/C) \leq 
    \dimo(\x/C \cup\{\a_n\}) + \dimo(\a_n/C)= \dimo(\x\a_n/C),$$
    where we applied the induction hypothesis and the case $n=1$.

    Finally we turn to the last claim of the lemma. Suppose 
    $\bdl(\a_1\ldots\a_n)=\dimo(\a_1\ldots\a_n)$. We need to show that 
    $\bdl(\x)=\dimo(\x)$ for all concatenated tuples $\x$ made of subtuples from 
    $\{\a_1,\ldots,\a_n\}$. Note that for every tuple of $\a_i$'s the quantities 
    $\bdl$ and $\dimo$ depend only on the subset of $\a_i$'s appearing in the 
    tuple (see Fact \ref{set-dep}), so up to relabelling co-ordinates we may 
    assume that $\x=\a_1\ldots \a_i$ for $i\in [1,n]$. Let 
    $\y:=\a_{i+1}\ldots\a_n$. Then by assumption $\bdl(\x\y)=\dimo(\x\y)$. By 
    $(\ref{upb})$ we have $\bdl(\y/\x) \leq \dimo(\y/\x)$ and $\bdl(\x) \leq 
    \dimo(\x)$. Hence by additivity we conclude that the last two inequalities are 
    equalities. This ends the proof.
  \end{proof}

  Finally we record one last observation, which will be useful in the next 
  paragraph.

  \begin{lemma}\label{lem2} If $P \subset K^{<\infty}$ is coherent and $\x \in 
    \acl^0(P)^{<\infty}$. Then $\bdl(\x) \geq \dimo(\x)$.
  \end{lemma}

  \begin{proof} Pick $\a_1,\ldots,\a_n \in P$ such that $\x \in 
    \acl^0(\{\a_1,\ldots,\a_n\})^{<\infty}$ and concatenate the $\a_i$'s in 
    $\a:=\a_1\ldots\a_n$. Then $\dimo(\a)=\dimo(\a\x)$. By additivity 
    $\dimo(\a\x)=\dimo(\x)+\dimo(\a/\x)$ and $\bdl(\a\x)=\bdl(\x)+\bdl(\a/\x)$. 
    By coherence of $P$ we have $\bdl(\a)=\dimo(\a)$. But $\bdl(\a/\x) \leq 
    \dimo(\a/\x)$ by Lemma \ref{lem1}. So $$\bdl(\x) \ge \bdl(\a\x) - 
    \dimo(\a/\x) \ge \bdl(\a) - \dimo(\a/\x) = \dimo(\a) - \dimo(\a/\x) = 
    \dimo(\x).$$
  \end{proof}

\end{parag}

\begin{parag}{\bf The Veblen axiom and incidence bounds.} In this paragraph we 
  exploit the Szemer\'edi-Trotter-type bounds described in Subsection 
  \ref{subsec:incidence} in order to show that the pregeometry 
  $(P;\acl^0\negmedspace\restriction_P)$ satisfies the Veblen axiom of 
  projective geometry.

  \begin{proposition}\label{lem:cgpCoherentLinearity} Assume $P\subset 
    K^{<\infty}$ is coherent, let $\a_1,\a_2 \in P$ and $B \subset P$. Assume 
    that $\a_1,\a_2 \notin \acl^0(B)^{<\infty}$, but $\a_2 \in\acl^0(\{\a_1\} 
    \cup B)^{<\infty}\setminus \acl^0(\a_1)^{<\infty}$. Then there is $\d \in 
    \acl^0(B)^{<\infty}$ such that $\d \in \acl^0(\a_1,\a_2)^{<\infty}$ and 
    $\bdl(\d)=\dimo(\d)=\dimo(\a_1)=\dimo(\a_2)$.
  \end{proposition}

  \begin{proof} For brevity set $\a:=\a_1\a_2$. Without loss of generality we 
    may assume $B$ is finite and $B=\{\a_3,\ldots,\a_n\}$. We set 
    $\b=(\a_3,\ldots,\a_n)$. First we check that $\a_1 \in 
    \acl^0(\a_2,\b)^{<\infty}$, $\a_i \ind^0 \b$ and $\a_1 \ind^0 \a_2$, and 
    that if $k:=\dimo(\a_1)$, then $\dimo(\a_2)=k$ and $\dim(\a)=2k$. The first 
    property follows from the exchange property of pregeometries and from 
    Proposition \ref{pregeo}. Lemma \ref{lem:cohFullEmb} tells us that $\a_i 
    \ind^0 \b$, since $\a_i \notin \acl^0(\b)^{<\infty}$. For the same reason 
    $\a_1 \ind^0 \a_2$. Then $\dimo(\a\b)=\dimo(\a_1\b)=\dimo(\a_2\b)$ is equal 
    to both $\dimo(\a_1)+\dimo(\b)$ and $\dimo(\a_2)+\dimo(\b)$. Hence 
    $\dimo(\a_2)=k$ and $\dimo(\a)=2k$. This also shows that $\dimo(\a/\b)=k$.

    Let $\d \in K^{<\infty}$ be a canonical base for $\a$ over $\b$ (see \S 
    \ref{parag:canonical}). By definition $\d \in \acl^0(\b)^{< \infty}$. We will 
    show that this $\d$ satisfies the desired conclusion. By definition we have 
    $\dimo(\a/\d)=\dimo(\a/\b)=k$. Hence $\dimo(\d)=\dimo(\a\d)-\dimo(\a/\d)\ge 
    \dimo(\a)-k=k$. By Lemma \ref{lem2} $\bdl(\d) \geq \dimo(\d)$. Hence we are 
    left to show the upper bound $\bdl(\d) \leq k$.

    To this end let $V$ be the locus of the tuple $\a\d$, i.e.\ $V=\loc^0(\a,\d)$, 
    let $X_1\subset K^{<\infty}$ be the type of $\a$, i.e.\ $X_1=\tp(\a)(K)$, let 
    $X_2=\tp(\d)(K)$ and finally let $X=(X_1 \times X_2) \cap V$. We wish to apply 
    the Szemer\'edi-Trotter bound of Lemma \ref{lem:SzT} to this data. 

    For this we first show that $\bdl(X(\d_1)\cap X(\d_2))=0$ for all $\d_1, \d_2 
    \in X_2$ with $\d_1 \neq \d_2$, so that $\beta=0$ in this lemma. Recall that 
    $X(\d_1):=\{\a' \in X_1 : \a'\d_1 \in V\}$. By Fact \ref{fact:ideal} we may 
    find $\a'=\a'_1\a'_2 \in X(\d_1)\cap X(\d_2)$ such that 
    $\bdl(\a'/\d_1\d_2)=\bdl(X(\d_1)\cap X(\d_2))$. Then $\a'\d_i \in V$ for 
    $i=1,2$. We are thus in the setting of Lemma \ref{lem:cb}. Since $\d_1 \neq 
    \d_2$ we conclude that $\dimo(\a'/\d_1\d_2)<k$. Hence 
    $\dimo(\a'_i/\d_1\d_2)<k$, and since both $\a_i'$ are $cgp$ (because $cgp$ is 
    type invariant, see Remark \ref{cgp-invariance}), we conclude that 
    $\bdl(\a'_i/\d_1\d_2)=0$ for each $i$. Thus by additivity $\bdl(X(\d_1)\cap 
    X(\d_2))=\bdl(\a'/\d_1\d_2)=0$ as claimed.

    Next note that $\bdl(X_2)=\bdl(\d)$, $\bdl(X_1)=\bdl(\a)=\dimo(\a)=2k$ by 
    coherence, and $\bdl(X) \geq \bdl(\a\d)=\bdl(\a/\d)+\bdl(\d)$. But 
    $\bdl(\a/\d) \geq \bdl(\a/\acl^0(\b)^{<\infty})=\bdl(\a/\b)$ by 
    $(\ref{alg-ineq})$. By coherence $\bdl(\a/\b)=\dimo(\a/\b)=k$. We conclude
    $$\bdl(X) \geq k + \bdl(\d)= \frac{1}{2}\bdl(X_1) + \bdl(X_2).$$
    Now comparing this to the Szemer\'edi-Trotter bound of Lemma \ref{lem:SzT} we 
    obtain $ \bdl(X_2) \leq \frac{1}{2}\bdl(X_1)$. In other words $\bdl(\d) \leq 
    k$. This ends the proof.
  \end{proof}

\end{parag}

\begin{parag}{\bf Proof of Theorem \ref{cor:coherentGeom}.}
  Here we show Theorem \ref{cor:coherentGeom}. 
  Lemma~\ref{lem:cgpCoherentLinearity} will help us find a modular geometry 
  explaining algebraic dependence on a coherent set. This is the engine behind 
  our main results. The idea comes from \cite[Subsection~5.17]{Hr-psfDims}, and 
  the context is essentially that of \cite[Remark~5.26]{Hr-psfDims}.

  \begin{definition}
    For $P \subset K^{<\infty}$, define the \defnstyle{coherent algebraic 
    closure} by
    \[ \ccl(P) := \{ \a \in \acl^0(P)^{<\infty} : \a \textrm{ is cgp and }
    \bdl(\a) = \dimo(\a) \} .\]
  \end{definition}

  \begin{lemma} \label{lem:cgpCoherentClosure}
    If $P$ is coherent, then so is $\ccl(P)$.
  \end{lemma}

  \begin{proof}
    We need to show that $\bdl(\a_1\ldots\a_n)=\dimo(\a_1\ldots\a_n)$ for any 
    $\a_i$'s from $\ccl(P)$. We proceed by induction on $n$. This holds when 
    $n=1$ by the definition of $\ccl(P)$. Set $\x:=\a_1\ldots\a_{n-1}$. By 
    induction hypothesis and Lemma \ref{lem1} $\{\a_1,\ldots,\a_{n-1}\}$ is 
    coherent and $\bdl(\x/\a_n) \leq \dimo(\x/\a_n)$. So by additivity 
    $$\bdl(\x\a_n) =\bdl(\x/\a_n)+\bdl(\a_n) \leq 
    \dimo(\x/\a_n)+\dimo(\a_n)=\dimo(\x\a_n).$$ But Lemma \ref{lem2} implies 
    that $\bdl(\x\a_n) \geq \dimo(\x\a_n)$. This ends the proof.
  \end{proof}

  Clearly $\ccl(\ccl(P)) = \ccl(P)$.
  We say $P \subset K^{<\infty}$ is \defnstyle{coherently algebraically closed} 
  if $P$
  is coherent and $P=\ccl(P)$.

  \begin{proposition} \label{thm:cgpCoherentModularity}
    Suppose $P \subset K^{<\infty}$ is coherently algebraically closed.
    Then the pregeometry $(P;\acl^0\negmedspace\restriction_P)$ is modular.
  \end{proposition}

  \begin{proof}
    We must show that if $B \subset P$ and $\a_1,\a_2 \in P\setminus 
    \acl^0(B)^{<\infty}$ are such that $\a_1 \in \acl^0(B \cup 
    \{\a_2\})^{<\infty}$, then $\a_1 \in \acl^0(\{\d,\a_2\})^{<\infty}$ for some 
    $\d \in P \cap \acl^0(B)^{<\infty}$. We may assume without loss of 
    generality that $B$ is finite, say $B=\{\a_3,\ldots,\a_n\}$. This is the 
    situation of Proposition \ref{lem:cgpCoherentLinearity} from which we 
    conclude that there is an integer $k$ such that 
    $k=\dimo(\a_1\a_2/\d)=\dimo(\a_2)=\dimo(\a_1)=\dimo(\d)$ for some $\d \in 
    \acl^0(\{\a_1,\a_2\})^{<\infty}$.

    We are left to show that $\d \in P$, and since we already know that 
    $\bdl(\d)=\dimo(\d)$ and $\d \in \acl^0(P)$, we are only left to check that 
    $\d$ is $cgp$. 

    To this end assume that $\d$ is not $\acl^0$ independent from $E$, for some $E 
    \subset K^{<\infty}$. We need to show that $\bdl(\d/E)=0$. By Fact 
    \ref{fact:ideal} we may pick $\a_1',\a_2' \in K^{< \infty}$ such that 
    $\tp(\a_1',\a_2'/\d) = \tp(\a_1,\a_2/\d)$ and $\bdl(\a_1',\a_2'/E\d) = 
    \bdl(\a_1',\a_2'/\d)$. For brevity write $\a:=\a_1\a_2$ and $\a'=\a_1'\a_2'$. 
    By additivity of $\bdl$ we may write:
    $$\bdl(\d/E)=\bdl(\d/\a'E)+\bdl(\a'/E) - \bdl(\a'/\d E).$$
    Let us examine the three terms on the right hand side.

    The first term $\bdl(\d/\a'E)$ is zero since $\d \in \acl^0(\a')^{<\infty}$, 
    because $\d \in \acl^0(\a)^{<\infty}$.

    The second term is equal to $\bdl(\a_1'/\a_2' E) + \bdl(\a_2'/ E)$. We claim 
    that this is at most $k$. Note that $\a_1',\a_2'$ are $cgp$ because 
    $\a_1',\a_2'$ are $cgp$ (see Remark \ref{cgp-invariance}). By $(\ref{upb})$ it 
    is enough to show that one of these terms is zero. Hence by the $cgp$ we only 
    need to check that either $\a_1' \nind^0 \a_2' E$ or $\a_2' \nind^0 E$. In 
    other words that $\a_1'\a_2' \nind^0 E$. This is indeed the case for otherwise 
    $\d$ would be independent from $E$, because $\d \in \acl^0(\a')^{<\infty}$.

    Finally let us turn to the third term. Since $\d \in \acl^0(B)^{<\infty}$ by 
    Fact \ref{alg-ineq} we have $$ \bdl(\a'/\d) =  \bdl(\a/\d) \geq 
    \bdl(\a/\acl^0(B)) = \bdl(\a/B).$$ On the other hand since $\a_1,\a_2,B$ lie 
    in $P$ and $P$ is coherent we have $\bdl(\a/B)= \dimo(\a/B)$. By Proposition 
    \ref{lem:cgpCoherentLinearity} this is $k$. Hence $\bdl(\a'/\d E) 
    =\bdl(\a'/\d) \geq k$. This concludes the proof.
  \end{proof}

  \begin{proof}[Proof of Theorem \ref{cor:coherentGeom}] By Proposition 
    \ref{pregeo} $(P;\acl^0\negmedspace\restriction_{P})$ is a pregeometry. By 
    Lemma \ref{lem:cgpCoherentClosure}, enlarging $P$ to $\ccl(P)$ if necessary, 
    we may assume that $P$ is coherently algebraically closed. By 
    Proposition~\ref{thm:cgpCoherentModularity} the associated geometry $\G_P$ 
    is modular. Hence the non-orthogonality relation is an equivalence relation. 

    Up to enlarging $P$ further if necessary, we can assume that each 
    non-orthogonality class in $\G_P$ of dimension $>1$ has infinite dimension. 
    This follows from the Lemma \ref{dim-increase} below applied iteratively 
    countably many times. In each finite dimensional non-orthogonality class we 
    pick a point $a$ and increase its dimension without altering the other 
    classes, until all classes are infinite dimensional. Now we conclude by the 
    Veblen-Young Theorem as recalled in Fact~\ref{fact:modStructure}.
  \end{proof}

  \begin{lemma}[dimension increase]\label{dim-increase} If $P$ is coherently 
    algebraically closed and $a,b,c \in P$ are distinct in $\G_P$ and collinear 
    in the sense that $c \in \acl^0({a,b})$, then there is $a' \in K^{<\infty}$ 
    non-orthogonal to $a$ such that $P':=\ccl(P \cup \{a'\})$ is coherent, $a' 
    \ind^0 P$, and every $x \in P'$ is either in $P$ or non-orthogonal to $a$. 
  \end{lemma}

  \begin{proof} Note first if $a,b \in P$ are non-orthogonal, then 
    $\dimo(a)=\dimo(b)$. This is part of the conclusion of Proposition 
    \ref{lem:cgpCoherentLinearity}, or also follows easily from Lemma 
    \ref{lem:cohFullEmb}.

    Now by Fact \ref{fact:ideal} we can pick $a',c' \in K^{<\infty}$ with 
    $\tp(a'c'/b)=\tp(ac/b)$ and $\bdl(a'c'/P)=\bdl(ac/b)$. Since $P$ is coherent 
    $\bdl(a/b)=\dimo(a/b)=\dimo(a)=\bdl(a)$, while since $a,a'$ have the same type 
    $\bdl(a)=\bdl(a')$ coincides with $\dimo(a')=\dimo(a)$. Also $a'$ and $c'$ are 
    $cgp$ (see Remark \ref{cgp-invariance}). Since $\tp(a'c'/b)=\tp(ac/b)$ we have 
    $c' \in \acl^0(P \cup \{a'\})$ and hence $\bdl(c'/Pa')=0$. Similarly, 
    $\bdl(c/ab)=0$. By additivity we conclude 
    $\bdl(a'/P)=\bdl(a'c'/P)=\bdl(ac/b)=\bdl(a/b)=\dimo(a)$. Hence also $a' \ind^0 
    P$ since $a'$ is $cgp$. We conclude $\dimo(a'/P)=\bdl(a'/P)$. It follows from 
    Lemma~\ref{lem1} and additivity that $P \cup \{a'\}$ is also coherent.

    By Lemma \ref{lem:cgpCoherentClosure} $P'=\ccl(P \cup \{a'\})$ is also 
    coherent. Moreover $b,a',c'$ are collinear so $a'$ is non-orthogonal to $a, b$ 
    and $c$. Finally if $x \in P'$, then $x \in \acl^0(P \cup \{a'\})$, so if $x 
    \notin P \cup \acl^0(a')$ by modularity there is $y \in P$ such that $x,a',y$ 
    are collinear, hence $a'$ and $x$ are non-orthogonal.
  \end{proof}


  \begin{remark}
    The results of this section go through with the same proofs when $ACF_0$ 
    is replaced by an arbitrary finite U-rank theory in which 
    Lemma~\ref{lem:SzT} holds.
  \end{remark}

\end{parag}

\section{Varieties with coherent generics}\label{sec:varieties}
We now show in Proposition \ref{prop:coherentSpecial} below that the locus of 
a coherent tuple is a special variety.
This will follow from Theorem~\ref{cor:coherentGeom} and a characterisation of 
the projective geometries which can arise from $\acl^0$. We shall give such a 
characterisation in Appendix~\ref{appx:EHeq}, generalising a result of 
Evans-Hrushovski.

\begin{proposition} \label{prop:coherentSpecial}
  Suppose $a_1,\ldots ,a_n \in K^{<\infty}$ are such that $\a=(a_1,\ldots 
  ,a_n)$ is coherent.
  Then $\loc^0(\a)$ is a special subvariety of $\prod_i \loc^0(a_i)$.
\end{proposition}
\begin{remark}
  Technically, we defined ``special'' only for complex varieties, but 
  $\loc^0(\a)$ is a variety over $C_0$ and $C_0$ need not come with an 
  embedding into $\C$. In our main applications in 
  Section~\ref{sec:asymptotic} below, $C_0$ will come with such an embedding; 
  more generally, we may take an arbitrary such embedding, or just define 
  ``special'' for varieties over an algebraically closed field $C_0$ by exact 
  analogy to the definition for varieties over $\C$ in the introduction.
\end{remark}


\begin{proof}
  In this proof we make use of some of the definitions from 
  Appendix~\ref{appx:EHeq}, applied to the pair of algebraically closed fields 
  $C_0 \leq K$. In particular, given $x \in K^{<\infty} \setminus 
  \acl^0(\emptyset )$ we set $\widetilde x := \acl^0(x)$, and we let $\G_K $ 
  be the projectivisation of the closure structure $(K^{<\infty},\acl^0)$ 
  defined in Example \ref{alg-tuples} above, namely
  $\G_K := \P(K^{<\infty};\acl^0) = \{ \widetilde x : x \in K^{<\infty} 
  \setminus \acl^0(\emptyset ) \}$.

  We may assume no $a_i$ lies in $\acl^0(\emptyset )$. Indeed, if say $a_1 \in 
  \acl^0(\emptyset )$, then $\loc^0(\a) = \{ a_1 \} \times \loc^0(a_2,\ldots 
  ,a_n)$, and $\{ a_1 \}$ is special (with the trivial group, which is a 
  special subgroup of itself), and so it suffices to show that 
  $\loc^0(a_2,\ldots ,a_n)$ is special.

  By Theorem~\ref{cor:coherentGeom}, $\{a_1,\ldots,a_n\}$ extends to a 
  coherent set $P$ such that $\G_P = \{ \widetilde p : p \in P \} \subset 
  \G_K$ splits as a sum of 1-dimensional and infinite dimensional projective 
  geometries over division rings. This induces a corresponding splitting of 
  $\a$ into subtuples of $a_i$'s, and the locus of $\a$ is the product of 
  their loci. So it suffices to show that each such locus is special. So we 
  may assume $\widetilde a_1,\ldots ,\widetilde a_n$ are all contained in a 
  single summand.

  We conclude by showing that $\loc^0(\a)$ is in co-ordinatewise 
  correspondence with a special subgroup, by finding a commutative algebraic 
  group $G$ over $C_0$ and generics $h_i \in G(K)$ with $\widetilde a_i = 
  \widetilde {h}_i$, such that $\loc^0(\h)$ is a special subgroup of $G^n$. By 
  Remark~\ref{rmk:corrGeneric}, this will suffice.

  If $\widetilde {a_i} = \widetilde {a_j}$ for $i \neq j$, we can take $h_j := 
  h_i$.
  So assume there are no such interalgebraicities. Let $\G_{\a} := \{ 
  \widetilde a_1, \ldots, \widetilde a_n \} \subset \G_P \subset \G_K$.

  If $\dim(\G_{\a})=1$ (the ``trivial'' case), then $\a=a_1$ and we may take 
  $G := \mathbb{G}_a^{\dimo(a_1)}$ and a point $h_1 \in G(K) = K^{\dimo(a_1)}$ 
  with $\widetilde {a_1} = \widetilde {h_1}$.

  Else, $\G_{\a}$ embeds in a projective geometry over a division ring,
  where moreover by Lemma~\ref{lem:cohFullEmb} the latter geometry is fully 
  embedded in $\G_K$ in the sense of Definition~\ref{defn:fullyEmbedded}.
  So by Proposition~\ref{prop:EHeq},
  there is an abelian algebraic group $G$ over $C_0$ with $\dim(G) = 
  \dimo(a_i)$,
  and a division subring $F$ of $\End^0_{C_0}(G)$,
  and $\h = (h_1,\ldots ,h_n) \in G(K)^n$ with $\widetilde {h_i} = \widetilde 
  {a_i}$,
  such that (in particular) $\dim_F(\left<{\h/G(C_0)}\right>_F) = 
  \dim(\G_{\a})$.

  Hence $A\cdot(\h/G(C_0))=0$ for some $A \in \Mat_n(F)$ of rank $n - 
  \dim(\G_{\a})$.
  By clearing denominators, we may assume that $A$ has entries from 
  $\End_{C_0}(G) \cap F$.

  Let $\c := A\cdot\h \in G(C_0)^n$. Since $C_0$ is algebraically closed, 
  $A\cdot\x = \c$ has a solution $\h_0 \in G(C_0)^n$. Replacing $\h$ by 
  $\h-\h_0$, which does not affect $\widetilde {h_i}$, we may assume $\h \in 
  \ker(A)$.
  Write $\ker(A)^o$ for the connected component of $\ker(A)$.
  By further replacing $\h$ by $e\cdot\h$ where $e \in \N$ is the exponent of 
  the finite group $\ker(A)/\ker(A)^o$, we may assume $\h \in \ker(A)^o$.
  Now it is not hard to see, e.g.\ by considering Gaussian elimination,
  that $\dim(\ker(A)) = \dim(G)(n-\operatorname{rank}(A))$.
  So $\dim(\ker(A)^o) = \dim(G)(n-\operatorname{rank}(A)) = \dimo(\a) = 
  \dimo(\h)$. Therefore $\loc^0(\h) = \ker(A)^o$ is a special subgroup of 
  $G^n$ as required.
\end{proof}

\section{Asymptotic consequences}\label{sec:asymptotic}

In this section we first unpeel the ultraproduct construction to show how 
coherent tuples correspond to varieties without powersaving. Then, combining 
this with Proposition~\ref{prop:coherentSpecial} above and some further 
argument, we prove the combinatorial theorems stated in the introduction.

Let $W_i$, $i=1,\ldots ,n$, be irreducible complex algebraic varieties each of 
dimension $d$, and let $V \subset \prod_i W_i$ be an irreducible complex 
closed subvariety. We first recall the following simple observation, already 
mentioned in the introduction.


\begin{lemma}[the trivial bound] \label{lem:power-saving}
  Let $\tau>d$. There is $C \in \N$ depending only on $\tau$ and $V$ such that
  if $X_i \subset W_i$ is in $(C,\tau)$-coarse general position in $W_i$ (see 
  Def. \ref{defn:taucgp}) and if $|X_i|\leq N^d$,
  then $|V \cap \prod_{i\leq n} X_i| \leq O_V(N^{\dim(V)})$.
  Furthermore, if $V$ admits no power-saving then $\dim(V)$ is an integral 
  multiple of $d$.
\end{lemma}
\begin{proof}
  We prove this by induction on $n$ and $\dim(V)$, with $C$ and the 
  multiplicative constant in $O_V$ depending only on the complexity of $V$ and 
  the $W_i$'s.
  For $n=1$ it is clear.
  For $n>1$, consider the projection $\pi : V \rightarrow \prod_{i<n} W_i$.
  Let $Y$ be the Zariski closure of the image $\pi(V)$.
  Then $Y$ is irreducible since $V$ is.
  By \cite[Lemma~3.7]{bgt},
  there is a proper closed subvariety $Z \subsetneq Y$ such that for $y \in 
  \pi(V) \setminus Z$ the fibre $\pi^{-1}(y)$ has dimension $D := \dim(V) - 
  \dim(Y)$
  and both $Y$ and $Z$ as well as the fibers $\pi^{-1}(y)$ have complexity 
  bounded by a constant depending only on the complexity of $V$ and the 
  $W_i$'s. We may assume that $\tau$ is larger than this constant.

  Now $V' := \pi^{-1}(Z)$ is a proper closed subvariety of $V$,
  so by the inductive hypothesis applied to its irreducible components,
  $|V' \cap \prod_{i\leq n} X_i| \leq O_V(N^{\dim(V)-1})$.

  If $D=0$, the fibres over $\pi(V) \setminus Z$ have size uniformly bounded 
  by some $c \in \N$, and so
  \[|(V \setminus V') \cap \prod_{i\leq n} X_i| \leq c|Y \cap \prod_{i<n} X_i| 
  \]
  and we conclude by the inductive hypothesis and $\dim(V) = \dim(Y)$.

  If $D=d$, we conclude by the trivial estimate
  \[ |(V \setminus V') \cap \prod_{i\leq n} X_i| \leq N^d|Y \cap \prod_{i<n} 
  X_i| \]
  and $\dim(V) = d + \dim(Y)$. 

  If $0<D<d$, by $\tau$-coarse general position of $X_n$, and the inductive 
  hypothesis
  \[ |(V \setminus V') \cap \prod_{i\leq n} X_i| \leq O(N^{\frac{d}{\tau}} |Y 
  \cap \prod_{i<n} X_i|) \leq O(N^{\frac{d}{\tau} + \dim(V)-D});\]
  so we see that for $\tau>d$ the desired bound holds,
  and moreover that $V$ admits a power-saving.
  Finally if the projection of $V$ to $W_i$ for some $i$ is not dominant, then 
  has no power saving, then so does $Y$ and $V$ has dominant projections on all 
  $W_i$'s with $i<n$.
\end{proof}

Let $C_0 \subset \C$ be a countable algebraically closed field over which $V$ 
and the
$W_i$'s are defined. Consider as in \textsection\ref{subsect:psfDim} a 
sequence $K_s$ of $\L$-structures on the complex field $\C$ and a scaling 
constant $\xi \in \stR$ so as to form the coarse pseudo-finite dimension 
$\bdl$ defined on subsets of tuples with co-ordinates in $K:=\prod_{s 
\rightarrow \U} K_s$ for some non-principal ultrafilter $\U$. Here as before 
$\L$ is a countable language expanding the language of rings on $(K,+,\cdot)$ 
and containing a constant symbol for each $c \in C_0$, and closed under 
cardinality quantifiers so as to make $\bdl$ invariant and continuous (cf. \S 
\ref{continuity}). 

For an irreducible algebraic variety $W$ defined over $C_0$, we will say that 
an internal set $X \subset W(K)$ is \defnstyle{cgp in $W$} if $0 < \bdl(X) < 
\infty$ and for any proper closed subvariety $W' \subsetneq W$ over $K$, we 
have $\bdl(X \cap W'(K)) = 0$.

\begin{lemma}\label{cgplem}
  Suppose $X \subset W(K)$ is an internal set which is cgp in $W$. Then any $a 
  \in X$ is cgp.
\end{lemma}
\begin{proof}
  Suppose $B \subset K^{<\infty}$ and $a \nind^0 B$. Then $W' := 
  \locus^W(a/B)$ is a proper subvariety of $W$, and so $\bdl(a/B) \leq \bdl(X 
  \cap W') = 0$.
\end{proof}

We also introduce one last piece of terminology:

\begin{definition}\label{def-dcgp} We will say that an element $a \in W(K)$ is 
  \defnstyle{dcgp in $W$} (for \emph{definably in coarse general position}) if 
  it is contained in a subset $X \subset W(K)$, which is definable without 
  parameters in $\L$ and is $cgp$ in $W$.
\end{definition}
It is immediate from Lemma \ref{cgplem} that every $dcgp$ tuple $\a$ is $cgp$. 
Recall that $W_1,\ldots,W_n$ are irreducible complex algebraic varieties of 
dimension $d$.


\begin{lemma} \label{lem:powersavingCoherent} Let $V \subset \prod_i W_i$ be 
  an irreducible complex closed subvariety. The following are equivalent:
  \begin{enumerate}
    \item The subvariety $V$ admits no power-saving,
    \item (existence of a coherent generic) for some language $\L$ as above there 
      are some $\a_i \in W_i(K)$ such that $\a=(\a_1,\ldots,\a_n) \in V(K)$ is 
      coherent and generic in $V$ with $\a_i$ $dcgp$ in $W_i$ for each $i$.
  \end{enumerate}
\end{lemma}

\begin{proof}
  If $V$ admits no power-saving, then for any $s \in \N$, setting $\tau := 1+s$ 
  and $\epsilon := \frac1\tau$,
  there exists $N_s\geq s$
  and $|X_{i,s}| \leq N_s^d$ such that $|V \cap \prod_i X_{i,s}| > N_s^{\dim 
  V-\epsilon}$
  and $|X_{i,s} \cap W'_i| \le |X_{i,s}|^\epsilon$ for any $W'_i$ a proper 
  closed subvariety of $W_i$ of complexity $\leq \tau$.

  After enlarging $\L$ if necessary, we may assume that $\prod_{s \rightarrow 
  \U} X_{i,s} =: X_i$ are definable without parameters in $\L$. Set the scaling 
  constant $\xi := \lim_{s \rightarrow \U} N_s$.
  Then by the above estimates and Lemma~\ref{lem:power-saving},
  $\bdl(V \cap \prod_i X_i) = \dim V$.
  So by Fact~\ref{fact:ideal}, say $\a=(\a_1,\ldots,\a_n) \in V \cap 
  \prod_{i=1}^n X_i$ with $\bdl(\a)=\dim V$. By construction each $X_i$ is 
  $cgp$ in $W_i$, so $\a_i$ is $dcgp$ and hence $cgp$ (Lemma \ref{cgplem}). In 
  particular $\bdl(\a_i)\leq \dimo(\a_i)$, for $\bdl(\a_i) \leq d$ since $\a_i 
  \in X_i$ and either $\dimo(\a_i)=d$ or $\a_i$ is contained in a proper 
  subvariety of $W_i$ defined over $C_0$, which forces $\bdl(\a_i)=0$, since 
  $X_i$ is $cgp$ in $W_i$. Also $\a$ is generic in $V$, i.e.\ $\dimo(\a)=\dim 
  V$, for otherwise $\a \in V'$ for some proper irreducible subvariety $V'$ of 
  $V$ defined over $C_0$ and hence by the trivial bound of 
  Lemma~\ref{lem:power-saving} we would have $\bdl(\a)\leq \bdl(V' \cap 
  \prod_{i=1}^n X_i) \leq \dim V -1$. It then follows from Lemma \ref{lem1} 
  that $\a$ is coherent.

  Suppose conversely that, for some $K_s$ and $\xi \in \stR$, we have a tuple 
  $\a \in V(K)$ which is coherent generic and for each $i$ we have $\a_i \in 
  X_i$, an $\L$-definable without parameters and $cgp$ subset of $W_i(K)$. To 
  say that $\a$ is coherent means that $\{\a_1,\ldots,\a_n\}$ is a coherent 
  set. In particular $\a_i$ is $cgp$ and $\bdl(\a_i)=\dimo(\a_i)$. Since $\a$ 
  is generic in $V$, its projection $\a_i$ is generic in the co-ordinate 
  projection $\pi_i(V) \subset W_i$. We may assume that this projection is 
  dominant, for otherwise by $cgp$ we would have $\bdl(\a_i)=0$ and hence 
  $\dimo(\a_i)=0$, which would mean that the projection $\pi_i(V)$ is a point, 
  and we may replace $V$ with the fibre of this projection and omit $W_i$. So 
  we have $\bdl(\a_i)=\dimo(\a_i)=d$. 

  Now let $\epsilon>0$ and $\tau \in \N$. Pick $Y_i \subset W_i(K)$ definable 
  over $\emptyset$ with $\a_i \in Y_i$ and $d \leq \bdl(Y_i) < d+\epsilon$. 
  Replacing $Y_i$ by $Y_i \cap X_i$ we may assume $Y_i$ is $cgp$ in $W_i$. Then 
  $\bdl(V \cap \prod_i Y_i) \geq \bdl(\a) = \dim V$. Let $Y_{i,s} := Y_i^{K_s}$ 
  be the interpretation in $K_s$.
  Then for $\U$-many $s$, for all $i \in \{1,\ldots ,n\}$, $Y_{i,s}$ is 
  $\tau$-$cgp$ in $W_i$, $1/\epsilon < |Y_{i,s}| < \infty$, and
  $|V \cap \prod_i Y_{i,s}| \geq |Y_{i,s}|^\frac{\dim V}{d+2\epsilon}$. Hence 
  $V$ admits no power-saving.

\end{proof}

\subsection{Sharpness}
\label{subsect:sharp}
In this subsection, we show the converse to 
Proposition~\ref{prop:coherentSpecial} and prove that every special subvariety 
has no power saving. For this we will need to construct certain well chosen 
cartesian products of finite sets, which are adapted to the special 
subvariety. 

The construction we are about to describe consists in building certain long 
multi-dimensional arithmetic progressions on few algebraically independent 
elements. The difficulty is that in order to belong to a given special 
subgroup, these progressions will need to satisfy some almost invariance under 
the division ring $F$ of $End^0(G)$ used to define the special subgroup. For 
this purpose it will be convenient to introduce the notion of constrainedly 
filtered ring, as follows.

\begin{definition}
  A \defnstyle{constrained filtration} of a ring $\O$ is a chain $\O_n \subset 
  \O$ of finite subsets, $n \in \N$, such that
  \begin{enumerate}
    \item[(CF0)] $\bigcup_{n \in \N} \O_n = \O$, and $\forall n \in \N.\; \O_n 
      \subset \O_{n+1}$; \item[(CF1)] $\exists k \in \N.\; \forall n \in \N.\; 
      \O_n+\O_n \subset \O_{n+k}$;
    \item[(CF2)] $\forall a \in \O.\; \exists k \in \N.\; \forall n \in \N.\; 
      a\O_n \subset \O_{n+k}$;
    \item[(CF3)] $\forall \epsilon>0.\; \frac{|\O_{n+1}|}{|\O_n|} \leq 
      O_{\epsilon}(|\O_n|^\epsilon)$.
  \end{enumerate}
  If a constrained filtration exists, we say $\O$ is \defnstyle{constrainedly 
  filtered}.
\end{definition}

\begin{example}\label{eg:Zconstrained}
  $\Z$ is constrainedly filtered, since $([-2^n,2^n])_n$ is a constrained 
  filtration.
\end{example}

Constrained filtrations are somewhat similar to Bourgain systems.

%

\begin{lemma} \label{lem:constFilt}
  Suppose $\O$ is a constrainedly filtered ring.
  \begin{enumerate}[(i)]\item The polynomial ring $\O[X]$ is constrainedly 
      filtered.
    \item If $\O$ is an integral domain and $a \in \O$, then the subring 
      $\O[a^{-1}]$ of the fraction field of $\O$ is constrainedly filtered.
    \item If $\O' \supset \O$ is an extension ring in which $\O$ is central and 
      which is free and finitely generated as a $\O$-module,
      then $\O'$ is constrainedly filtered.
  \end{enumerate}
\end{lemma}
\begin{proof}
  Say $(\O_n)_n$ is a constrained filtration of $\O$.
  \begin{enumerate}[(i)]\item
      Let $\O_n' := \sum_{i<n} \O_n X^i$.
      Then (CF0)-(CF2) are easily verified.
      For (CF3), note that $|\O_n'| = |\O_n|^n$,
      and so for $\epsilon>0$, for $n >> 0$,
      $|\O_{n+1}'|/|\O_n'| = |\O_n|(|\O_{n+1}|/|\O_n|)^{n+1}
      \leq O_\epsilon(|\O_n|^{1+(1+n)\epsilon})
      \leq O_\epsilon(|\O_n'|^{\frac 1n + \frac{1+n}n \epsilon})
      \leq O_\epsilon(|\O_n'|^{2\epsilon})$.
    \item
      Say $k \in \N$ is such that for all $n \in \N$ we have $a\O_n \subset 
      \O_{n+k}$ and $\O_n + \O_n \subset \O_{n+k}$.
      Let $\O'_n := a^{-n}\O_{2kn}$.
      Then
      $\O'_n + \O'_n
      = a^{-n}(\O_{2kn}+\O_{2kn})
      \subset a^{-n}(\O_{2kn+k})
      = a^{-(n+1)}(a\O_{2kn+k})
      \subset a^{-(n+1)}(\O_{2kn+k+k})
      = \O'_{n+1}$, so (CF1) holds.

      (CF0) and (CF2) are immediate,
      and (CF3) holds since $|\O'_n| = |\O_{2kn}|$.

    \item
      Say $\O' = \bigoplus_{i=1}^d a_i \O$.
      Then let $\O'_n := \bigoplus_{i=1}^d a_i \O_n$.

      Then (CF0), (CF1), and (CF3) are clear. For (CF2),
      let $c_{ij}^t \in \O$ be such that $a_ia_j = \sum_t c_{ij}^t a_t$;
      then given $\beta = \sum a_ib_i \in \O'$ with $b_i \in \O$,
      let $k := \max_{i,j,t} (k_{b_i} + k_{c_{ij}^t})$
      where $\alpha \O_n \subset \O_{n+k_\alpha}$ ($\forall n$),
      and say $\O_n+\O_n \subset \O_{n+l}$ ($\forall n$).
      Then $\beta\O'_n
      = \sum_j \beta a_j \O_n
      = \sum_{i,j} a_ib_ia_j \O_n
      = \sum_{i,j} b_ia_ia_j \O_n
      = \sum_{i,j,t} b_i c_{ij}^t a_t \O_n
      \subset \sum_{i,j,t} a_t \O_{n+k}
      \subset \O'_{n+k+d^2l}.$
  \end{enumerate}
\end{proof}

%
%

\begin{lemma} \label{lem:finDimConstFilt}
  Suppose $D$ is a finite-dimensional algebra over a characteristic 0 field 
  $L$, and $\O \subset D$ is a finitely generated subring. Then there exists a 
  constrainedly filtered subring $\O' \subset D$ extending $\O$.
\end{lemma}
\begin{proof} Let $(e_k)_{1\leq k \leq d}$ be an $L$-basis of $D$ and 
  $(f_j)_j$ generators of $\O$. Without loss of generality we may change $L$ 
  into the subfield generated by the co-ordinates $f_{j}^k$ of the $f_j$'s and 
  the co-ordinates $c_{ij}^k$ of the products $e_ie_j$'s. So we may assume 
  that $L$ is finitely generated. Let $\z=(z_1,\ldots,z_n)$ be a transcendence 
  basis for $L$ over $\Q$. Then $[L:\Q(\z)]$ is finite, so $D$ is again 
  finite-dimensional over $\Q(\z)$, and without loss of generality, we may 
  assume that $L=\Q(\z)$. There is a polynomial $g \in \Z[\z]$ such that all 
  $f_j^k$ and $c_{ij}^k$ belong to $\Z[\z,\frac{1}{g}]$. By 
  Example~\ref{eg:Zconstrained}, $\Z$ is constrainedly filtered. Then by 
  Lemma~\ref{lem:constFilt}(i), so is $\Z[\z]$, and by item $(ii)$ so is 
  $R:=\Z[\z,\frac{1}{g}]$, and by item $(iii)$ so is $\O' := \sum_{k=1}^d R 
  e_k \supset \O$.



\end{proof}

\begin{fact} \label{fact:KA}
  A division subring of a matrix algebra over a division ring has finite 
  dimension over its centre.
\end{fact}
\begin{proof}
  This is a special case of the Kaplansky-Amitsur theorem 
  \cite[p17]{Jacobson-PIAlgs}, which shows that any primitive algebra 
  satisfying a proper polynomial identity is finite dimensional over its 
  centre. Indeed, any division ring is a primitive algebra, and any matrix 
  algebra $M_n(\Delta)$ over a division ring $\Delta$ satisfies a polynomial 
  identity (e.g. by the Amitsur-Levitzki theorem \cite[p21]{Jacobson-PIAlgs}).
\end{proof}

In particular, combining this fact with the previous lemma, we see that if $F$ 
is a division subring of a matrix algebra, then every finite subset of $F$ is 
contained in a constrainedly filtered subring of $F$. We will use this 
observation in the next proposition. Although this is sufficient for our 
purposes, we do not know it to be the optimal result along these lines - in 
fact, for all we know, it could be that every finitely generated subring of 
$M_n(\C)$ is constrainedly filtered.

\begin{proposition} \label{prop:converse}
  Suppose $V \subset \prod_i W_i$ is special.
  Then, for appropriate choices of $C_0 \leq \C$ and structures $K_s$ with 
  universe $\C$ and scaling constant $\xi \in \stR$,
  the variety $V$ has a coherent generic $\a \in V(K)$ such that each $a_i$ is 
  $dcgp$ in $W_i$.
\end{proposition}
\begin{proof}
  The conclusion is preserved by taking products and under correspondences, so 
  we may assume $W_i = G$ where $G$ is a $d$-dimensional commutative connected 
  algebraic group defined over a countable algebraically closed subfield $C_0 
  \subset \C$,
  and $V = H \leq G^n$ is a special subgroup.
  By Remark~\ref{remk:lieSpecial} and permuting co-ordinates, we may assume 
  that the Lie subalgebra
  $\Lie(H) \leq \Lie(G)^n$ is the graph of an $F$-linear map $\theta = 
  (\theta_1,\ldots ,\theta_m) : \Lie(G)^m \rightarrow \Lie(G)^{n-m}$ with 
  $\theta_i(\x) = \sum_{j=1}^d\alpha_{ij}x_j$, where $\alpha_{ij} \in F$ and 
  $F$ is a division subring of $\End^0(G):=\End(G) \otimes_{\Z} \Q$. We may 
  assume that $F$ is generated by the $\alpha_{ij}$.
  Extending $C_0$ if necessary, we may assume $F \subset \End^0_{C_0}(G)$, 
  i.e.\ that every element of $F$ is a scalar multiple of an algebraic 
  endomorphism of $G$ which is defined over $C_0$.

  Now $F$ acts faithfully by $\C$-linear maps on $\Lie(G) \cong \C^d$, so by 
  Fact~\ref{fact:KA}, $F$ has finite dimension over its centre.
  So by Lemma~\ref{lem:finDimConstFilt}, there is a constrainedly filtered 
  subring $\O \subset F$ such that $\alpha_{ij} \in \O$ ($\forall i,j$).


  Say $(\O_n)_n$ is a constrained filtration for $\O$. Let $\exp_G : \Lie(G) 
  \twoheadrightarrow G(\C)$ be the Lie exponential map, which is a surjective 
  $\End(G)$-homomorphism.

  For every positive integer $s$, let $\gammatup_s = 
  (\gamma_{s,i})_{i=1,\ldots ,s} \in G(\C)^s$ be algebraically generic over 
  $C_0$, i.e.\ $\dim^0(\gammatup_s) = sd$,
  and let $\gammatup'_s \in \Lie(G)^s$ be arbitrary such that 
  $\exp_G(\gamma'_{s,i}) = \gamma_{s,i}$.
  Note then that $\gammatup'_s$ is $F$-linearly independent modulo 
  $\ker(\exp_G)$.

  For $k \in \N_{\ge 0}$ we set $X_{k,s} := \sum_{i=1}^s\O_{s-k}\gamma'_{s,i} 
  \subset \Lie(G)$ if $s\geq k$ and $X_{k,s} := \emptyset $ if $s<k$.
  Let $X_k := \prod_{s \rightarrow \U} X_{k,s} \subset \Lie(G)(K) := 
  (\Lie(G))^\U$. Set the scaling constant $\xi := |X_0|^{\frac{1}{d}}$, so 
  that $\bdl(X_0)=d$.

  \begin{claim}
    Let $Z := \bigcap_{k \ge 0} X_k \subset \Lie(G)(K)$. Then $Z$ is an 
    $\O$-submodule with $\bdl(Z) = d$.
  \end{claim}
  \begin{proof}
    By (CF1),
    there is $k_0$ such that for all $k$ and $s$ we have
    $X_{k,s}+X_{k,s} \subset X_{k-k_0,s}$. It follows that $Z+Z \subset Z$. 
    Similarly, (CF2) implies $aZ \subset Z$ for all $a \in \O$. Finally, by 
    (CF3), for any $k,s$ and $\epsilon \in \R_{>0}$, we have 
    $\frac{|X_{k,s}|}{|X_{k+1,s}|} = \frac{|\O_{s-k}|^s}{|\O_{s-k-1}|^s} \leq 
    O_\epsilon(|\O_{s-k-1}|^{s\epsilon}) = O_\epsilon(|X_{k+1,s}|^\epsilon)$.
    Hence $\bdl(X_k) \leq (1+\epsilon)\bdl(X_{k+1})$ for any $\epsilon>0$,
    so $\bdl(X_k) \leq \bdl(X_{k+1})$.
    Clearly $\bdl(X_{k+1}) \leq \bdl(X_k)$,
    so $\bdl(X_{k+1}) = \bdl(X_k)$. So by induction, $\bdl(X_k) = d$ for all 
    $k$, so $\bdl(Z) = \inf_k \bdl(X_k) = d$.
  \end{proof}

  Now since $Z$ is an $\O$-submodule, the co-ordinate projection to $\Lie(G)^m$ 
  induces a bijection from $\Lie(H)
  \cap Z^n$ to $Z^m$, so $\bdl(\Lie(H) \cap Z^n) = md$. Moreover $\exp_G$ is 
  injective on each $X_{k,s}$ and hence on each $X_k$ and on $Z$.

  \begin{claim}
    $\exp_G(X_0)$ is $cgp$ in $G$.
  \end{claim}
  \begin{proof}
    Suppose $W \subsetneq G$ is a proper closed subvariety over $K$.
    Say $C_0(\b) \subset K$ is a finitely generated extension of $C_0$ such 
    that $W$ is over $C_0(\b)$. Then $W = J_{\b}$ for some $J_{\x}$ a 
    constructible family defined over $C_0$ of proper closed subvarieties of 
    $G$. If $\b = \lim_{s \rightarrow \U} \b_s$ and $W_s := J_{\b_s}$ then 
    $W(K) = \prod_{s \rightarrow \U} W_s(\C)$.

    It holds for $\U$-many $s$ that $\dim^0(\b_s) \leq \dim^0(\b) =: k$.
    We claim that for such $s$, we have $\operatorname{rk}_F ( 
    \exp_G^{-1}(W_s(\C)) \cap X_{0,s} ) \leq k$.

    Indeed, suppose $\g' = (g_i)_{i=0}^k$ is $F$-linearly independent with 
    $g_i' \in \exp_G^{-1}(W_s(\C)) \cap X_{0,s}$. By $F$-linear algebra, some 
    $k+1$-subtuple $\gammatup''$ of the generators $\gammatup'_s$ of $X_{0,s}$ 
    is in the $F$-linear span of $\g'$. Let $g_i := \exp_G(g_i')$. Then 
    $\dim^0(\g) \geq \dim^0(\exp_G(\gammatup'')) = (k+1)d$,
    so $\dim^0(\g) = (k+1)d$.

    Then $\dim^0(\g/\b_s) \leq \dim(W_s^{k+1}) \leq (k+1)(d-1) = \dim^0(\g) - 
    (k+1) < \dim^0(\g) - \dim^0(\b_s)$,
    so $\dim^0(\b_s) < \dim^0(\g) - \dim^0(\g/\b_s) = \dim^0(\b_s) - 
    \dim^0(\b_s/\g) \leq \dim^0(\b_s)$, which is a contradiction.

    So $| W_s(\C) \cap \exp_G(X_{0,s}) | = | \exp_G^{-1}(W_s(\C)) \cap X_{0,s} 
    | \leq | \O_s |^k = (|\O_s|^s)^{\frac ks} = |X_{0,s}|^{\frac ks}$.
    So for any $\epsilon \in \R_{>0}$, considering sufficiently large $s$, we 
    deduce $\bdl(W(K) \cap \exp_G(X_0)) \leq \epsilon\bdl(X_0)$.
    Hence $\bdl(W(K) \cap \exp_G(X_0)) = 0$, as required.
  \end{proof}


  By Fact \ref{fact:ideal} we can pick $\a \in H \cap \exp_G(Z)^n$ with 
  $\bdl(\a) = \bdl(H \cap \exp_G(Z)^n)$. By injectivity of $\exp_G$ on $Z$ 
  this is $\geq \bdl (\Lie(H) \cap Z^n) = md$.
  Note that $\bdl(a_i) \leq \bdl(Z)=d$ and that, by the above Claim $a_i \in 
  \exp_G(X_0)$ is $dcgp$ in $G$ (see Def. \ref{def-dcgp}) and hence $cgp$ (see 
  Lemma \ref{cgplem}). So by Lemma~\ref{lem1} $\a$ is coherent.
\end{proof}

\begin{remark} \label{remk:profiniteCovers}
  The only essential role played by Lie theory in the above proof is to 
  establish that $F$ embeds in a matrix algebra; for the rest of the proof, 
  $\exp_G : \Lie(G) \rightarrow G$ is used only to pick out choices of systems 
  of division points of elements of $G$, and this can instead be done directly 
  by replacing $\exp_G$ with $\rho : \widehat G \rightarrow G$ where $\widehat 
  G$ is the ``profinite cover'' $\widehat G := \invlim ([n] : G \rightarrow 
  G)$ consisting of ``division systems'' $(x_n)_n$ satisfying $[n]x_{nm} = 
  x_m$, and $\rho$ is the first co-ordinate map of the inverse limit, 
  $\rho((x_n)_n) := x_1$. Then $\End^0(G)$ acts on $\widehat G$ by $\frac\eta 
  m (x_n)_n := (\eta x_{nm})_n$.
\end{remark}

\begin{remark} \label{remk:converseGp}
  In the case that $G$ is a semiabelian variety, we can do slightly better and 
  obtain approximate $\O$-modules which are in general position in the sense 
  of Elekes-Szab\'o rather than merely in coarse general position. More 
  precisely, say an internal subset $X \subset G(K)$ is in general position if 
  it has finite intersection with any proper closed subvariety $W \subsetneq 
  G$ over $K$. Then proceed as in the proof of Proposition~\ref{prop:converse} 
  but taking $\gammatup_s := \gamma \in G(\C)$ to be a singleton which is in 
  no proper algebraic subgroup of $G$. Let $\Gamma := \O\gamma \leq G(\C)$ be 
  the $\O$-submodule generated by $\gamma$, which is a finitely generated 
  subgroup of $G(\C)$. As shown in 
  \cite[Theorem~4.7]{Scanlon-automaticUniformity},
  as a consequence of the truth of the Mordell-Lang conjecture,
  if $V_b$ is a constructible family of subvarieties of $G$ then there is a 
  uniform bound on the sizes of finite intersections $V_b \cap \Gamma$.
  Hence $\exp_G(X_k)$ is in general position in $G$.

  However, this approach clearly fails for $G=\mathbb{G}_a^2$, by considering 
  intersections with linear subvarieties. Pach 
  \cite[Theorem~2]{Pach-midpoints} gives an example of an internal subset $X 
  \subset K^2$ with $\bdl(X)=1=\bdl(X+X)$ and where the intersection with any 
  linear subvariety has size at most 2; however, quadratic subvarieties 
  witness that this $X$ is not in general position. We do not know whether it 
  is possible to find such an $X$ which is in general position. This prompts 
  the following question.
\end{remark}

\begin{question}
  Is there a sequence of finite sets $A_n \subset \C^2$ such that $|A_n|\to 
  +\infty$,  $|A_n + A_n| \leq |A_n|^{1 + 1/n}$ and   with the property that, 
  for each degree $d$, $$\sup_{n, \mathcal{C}, \deg \mathcal{C} \leq d} |A_n 
  \cap \mathcal{C}|$$ is finite, where $\mathcal{C}$ runs through algebraic 
  curves $\mathcal{C} \subset \C^2$.
\end{question}

\subsection{Proofs of the main results}
\label{subsect:proofs}

We first observe that Theorem \ref{thm:main1} is a special case of Theorem 
\ref{thm:main}, as follows immediately from the following lemma.

\begin{lemma}\label{lem:1dimSpecial}
  Let $G$ be a 1-dimensional connected complex algebraic group, let $n \geq 
  1$, and let $H \leq G^n$ be a connected algebraic subgroup. Then $H$ 
  is a special subgroup of $G^n$.
\end{lemma}
\begin{proof}
  First, suppose $G$ is the additive group $(\C;+)$.
  Then $F := \End^0(G) = \C$ is a division ring, and $H$ is a vector subgroup, 
  i.e.\ $H = \ker(A)$ for some $A \in \Mat_n(\C)$, as required.

  Otherwise, $G$ is the multiplicative group or an elliptic curve. In either 
  case, $F := \End^0(G)$ is again a division ring, being either $\Q$ or a 
  quadratic imaginary field extension of $\Q$.
  We refer to \cite[Lemma~4.1(i)]{coversfRM} for the fact that $H \leq G^n$ is 
  the connected component of $\ker(A)$ for some $A \in \Mat_n(\End(G))$.
\end{proof}

\begin{proof}[Proof of Theorem~\ref{thm:main}]
  Combine Lemma~\ref{lem:powersavingCoherent} with 
  Proposition~\ref{prop:coherentSpecial} and Proposition~\ref{prop:converse}.
\end{proof}

We end this section with a proof of Corollary \ref{gen-sum-product} from the 
Introduction. It is a special case of the following more precise result:

\begin{corollary}
  Suppose $(G_1;\cdot_1),(G_2;\cdot_2)$ are non-isogenous connected complex 
  algebraic groups of the same dimension,
  and $\Gamma \subset G_1 \times G_2$ is a generically finite algebraic 
  correspondence.
  Then there are $\tau,\epsilon,c>0$ such that if $A_i \subset G_i$ are finite 
  sets such that $\Gamma \cap (A_1 \times A_2)$ is the graph of a bijection 
  between $A_1$ and $A_2$,
  and if $A_i \subset G_i$ and $A_i \cdot_i A_i \subset G_i$ are $\tau$-$cgp$ 
  for $i=1,2$,
  then
  \[ \max( |A_1 \cdot_1 A_1|, |A_2 \cdot_2 A_2| ) \geq c|A_1|^{1+\epsilon} .\]
\end{corollary}
\begin{remark}
  The $cgp$ condition holds trivially for any $A$ when $\dim(G_i)=1$.

  %
\end{remark}

\begin{proof}
  Let $V = \{ (x_1,y_1,x_1\cdot_1y_1,x_2,y_2,x_2\cdot_2y_2) : 
  (x_1,x_2),(y_1,y_2) \in \Gamma \}$.
  Suppose $V$ is special.
  Then $V$ is in co-ordinatewise correspondence with a special subgroup $H 
  \leq G^6$, say. As in Remark~\ref{rmk:mainCodim1}, the projection of $H$ to 
  the first three co-ordinates is in co-ordinatewise correspondence with the 
  graph $\{ (x,y,x+y) \}$ of the group operation of $G$. Hence the graph of 
  the group operation of $G_1$ is in co-ordinatewise correspondence with that 
  of $G$.
  By Fact~\ref{fact:corrIsog}, $G_1$ is commutative and isogenous to $G$.
  Similarly, considering the projection to the last three co-ordinates, $G_2$ 
  is commutative and isogenous to $G$. Since isogeny is an equivalence 
  relation, this contradicts the assumption that $G_1$ and $G_2$ are not 
  isogenous.

  So by Theorem~\ref{thm:main}, $V$ admits a power-saving, say by $\epsilon'$.
  So for sufficiently large $\tau$, if $A_i$ are as in the statement, then
  setting $X := \{ (a_1,b_1, a_1\cdot_1 b_1, a_2,b_2, a_2\cdot_2 b_2) : 
  a_i,b_i \in A_i; (a_1,a_2),(b_1,b_2) \in \Gamma \} \subset V$, we have 
  $|A_1|^2 = |X| \leq O(\max( |A_1 \cdot_1 A_1|, |A_2 \cdot_2 A_2| 
  )^{2-\epsilon'})$.
  So $\epsilon := \frac{\epsilon'}{2-\epsilon'}$ is as required.
\end{proof}

\begin{question}
  Consider the following weakening of coarse general position. Say $a \in 
  K^{<\infty}$ with $\bdl(a)=\dimo(a)$ is in \emph{Larsen-Pink general 
  position} if for any $B \subset K^{<\infty}$, we have $\bdl(a/B) \leq 
  \dimo(a/B)$. This hypothesis suffices to give the trivial upper bound of 
  Lemma~\ref{lem:power-saving}, and it is not satisfied by the counterexample 
  of Section~\ref{gpNecessity}, nor by similar constructions based on 
  nilpotent groups.
  
  Does Theorem~\ref{thm:main} go through unchanged if coarse general position 
  is relaxed to Larsen-Pink general position?
  
  The techniques of this paper are insufficient to prove this generalisation, 
  as the corresponding weakened notion of coherence does not directly yield a 
  pregeometry.

  Since type-definable subgroups of simple algebraic groups do satisfy this 
  Larsen-Pink condition (see \cite[2.15]{Hr-psfDims}), a positive answer 
  should suffice to recover the characteristic 0 case of the main theorem of 
  \cite{bgt}, Theorem~5.5.
\end{question}

\section{Coherence in subgroups}\label{sec:subgroups}

In this section we observe a strengthening of our results in the special case 
of a $\bigwedge$-definable pseudo-finite subgroup and derive Theorem 
\ref{thm:coherentApproxSubgroup} from the introduction. We then briefly 
discuss connections to Diophantine problems and Manin-Mumford.


\begin{theorem} \label{thm:coherentSubgroup} We keep the notational setup of 
  Section \ref{sec:setup}.
  Let $G$ be a commutative algebraic group over $C_0$ and $\Gamma \leq G(K)$ 
  be a $\bigwedge$-definable (over $\emptyset $) subgroup of $G(K)$ contained 
  in a $cgp$ definable (over $\emptyset$) subset of $G$ (see 
  Definition~\ref{def-dcgp}). Assume $\bdl(\Gamma) = \dim(G)$. Then the locus 
  $\locus^{G^n}(\gammatup/C_0)$ of any coherent tuple $\gammatup \in \Gamma^n$ 
  is a coset of an algebraic subgroup.
\end{theorem}
\begin{remark}
  The commutative case is the only relevant case: by 
  Corollary~\ref{specialAbelian}, if $G$ is a connected algebraic group with 
  such a subgroup $\Gamma$ then $G$ is commutative.
\end{remark}
\begin{proof}
  By Lemma~\ref{cgplem} any $\alpha \in \Gamma$ is $cgp$. In particular if 
  $\bdl(\alpha)>0$, then $\alpha$ is generic in $G$ and $\bdl(\alpha) \leq 
  \bdl(\Gamma)=\dim(G)=\dimo(\alpha)$. So for all $\alpha \in \Gamma$ we have 
  $\bdl(\alpha) \leq \dimo(\alpha)$.

  By Fact~\ref{fact:ideal}, we may find $\alphatup \in \Gamma^n$ with 
  $\bdl(\alphatup/\gammatup) = \bdl(\Gamma^n) = n\dim(G)$.
  Then
  $\bdl(\gammatup,\alphatup,\gammatup+\alphatup) = \bdl(\gammatup,\alphatup) = 
  \bdl(\gammatup) + n\dim(G) = \dimo(\gammatup) + n\dim(G) \geq 
  \dimo(\gammatup,\alphatup) = 
  \dimo(\gammatup,\alphatup,\gammatup+\alphatup)$,
  and $\gamma_i, \alpha_i, \gamma_i+\alpha_i \in \Gamma$,
  so $(\gammatup,\alphatup,\gammatup+\alphatup)$ is coherent by 
  Lemma~\ref{lem1}.

  Since the product of cosets is a coset, we may assume that 
  $\locus^{G^n}(\gammatup/C_0)$ is not a product of loci of subtuples.
  Then as in Proposition~\ref{prop:coherentSpecial},
  there is a commutative algebraic group $G'$ over $C_0$
  and a tuple $(\gammatup',\alphatup',\psitup')$ which is generic in a 
  connected subgroup $H \leq G'^{3n}$ and which is co-ordinatewise 
  $\acl^0$-interalgebraic with $(\gammatup,\alphatup,\gammatup+\alphatup)$.
  Furthermore, for each $i$ we have $\theta_i^1\gamma_i' + \theta_i^2\alpha_i' 
  + \theta_i^3\psi_i' \in G'(C_0)$ for some $\theta_i^j \in \End_{C_0}(G')$ 
  invertible in $\End^0_{C_0}(G')$, and so we may assume without loss that 
  $\psi_i' = \gamma_i'+\alpha_i'$.

  Then by Fact~\ref{fact:corrIsog}, there are $m \in \N$ and isogenies $\eta_i 
  : G' \rightarrow G$ and $k_i \in G(C_0)$ such that $\eta_i(\gamma'_i) = 
  m\gamma_i + k_i$ for $i=1,\ldots ,n$.
  Hence $\locus^{G^n}(m\gammatup/C_0) = 
  (\prod_i\eta_i)(\locus^{G^n}(\gammatup')) - \k = (\prod_i\eta_i)(\pi_1(H)) - 
  \k$ where $\k=(k_1,\ldots ,k_n)$ and $\pi_1 : (\x,\y,\z) \mapsto \x$.
  So $\locus^{G^n}(m\gammatup/C_0)$ is a coset of an algebraic subgroup of 
  $G^n$,
  and hence so is $\locus^{G^n}(\gammatup/C_0)$, as required.
\end{proof}

\begin{proof}[Proof of Theorem~\ref{thm:coherentApproxSubgroup}]
  First, note that there exists a function $f : \N \rightarrow \N$ such that 
  any translate of a subvariety of $G^n$ of complexity $\leq \tau$ has 
  complexity $\leq f(\tau)$.
  This follows from \cite[Lemma~3.4]{bgt}, and can also be seen as a 
  consequence of the fact that the family of translates of subvarieties in a 
  constructible family is constructible.
  Increasing $f$ if necessary, we may assume also that $f$ is strictly 
  increasing and $2^{-\tau} + \frac1{f(\tau)} \leq \frac1\tau$ for any $\tau 
  \in \N$.

  By \cite[Proposition~2.26]{TaoVu}, there are $C_1,C_2>0$ such that if $|A+A| 
  \leq |A|^{1+\eps}$ then $H := A-A$ is an $|A|^{C_1\eps}$-approximate 
  subgroup and $|H| \leq |A|^{1+C_2\eps}$.

  It therefore suffices to prove the following revised statement:
  there are $N,\tau,\eps,\eta>0$ depending only on $G$ and the complexity of 
  $V$ such that if $H \subset G$ is a $\tau$-$cgp$ $|H|^\eps$-approximate 
  subgroup and $|H| \geq N$, then $|H^n \cap V| < |H|^{\frac{\dim(V)}{\dim(G)} 
  - \eta}$.
  Indeed, given $\tau' \in \N$, if $(N,\tau,\eps,\eta)$ are as required in the 
  revised statement for $V$ of complexity at most $f(\tau')$, then 
  $(N,\tau,\eps':=\frac{\eps}C_1,\eta':=\frac{\dim(V)}{\dim(G)} - 
  (\frac{\dim(V)}{\dim(G)} - \eta)(1 + C_2\eps'))$
  are as required in the original statement for $V$ of complexity at most 
  $\tau'$, after increasing $C_1$ if necessary to ensure $\eta'>0$.
  Indeed, given $V$ of complexity at most $\tau'$, suppose there is $A \subset 
  G$ with $|A| \geq N$ and $|A^n \cap V| \geq |A|^{\frac{\dim(V)}{\dim(G)} - 
  \eta'}$ and $|A+A| \leq |A|^{1+\eps'}$,
  and with $H := A-A$ being $\tau$-$cgp$.
  Then $H$ is an $|A|^{C_1\eps'} = |A|^{\eps}$-approximate subgroup, and $|H| 
  \geq |A| \geq N$.
  Let $x \in A$, and let $V' := V-(x,x,\ldots ,x)$, which has complexity at 
  most $f(\tau')$.
  Then $|H^n \cap V'| \geq |A^n \cap V|
  \geq |A|^{\frac{\dim(V)}{\dim(G)} - \eta'}
  \geq |H|^{(\frac{\dim(V)}{\dim(G)} - \eta')(1+C_2\eps')^{-1}}
  \geq |H|^{\frac{\dim(V)}{\dim(G)} - \eta}$.
  This contradicts the revised statement.

  Now suppose the revised statement fails. Then there is a family $(V_b)_b$ of 
  bounded complexity subvarieties of $G^n$, such that for each $s \in \N$ 
  there is an $f(s)$-$cgp$ subset $H_s \subset G$ such that $H_s$ is an 
  $|H_s|^{4^{-s}}$-approximate subgroup, and a parameter $b_s$, such that 
  $|H_s| \geq s$ and $|H_s^n \cap V_{b_s}| \geq 
  |H_s|^{\frac{\dim(V_{b_s})}{\dim(G)} - \frac1s}$ and $V_{b_s}$ is not a 
  coset.

  Let $\U$ be a non-principal ultrafilter on $\N$,
  let $\Gamma_i := \prod_{s \rightarrow \U} \sum_{j=1}^{2^{s-i}}H_s$ for $i 
  \in \N$,
  and let $\Gamma := \bigcap_{i\geq 0} \Gamma_i$.
  Set the language $\L$ to be such that each $\Gamma_i$ is definable.
  Then $\Gamma$ is a $\bigwedge$-definable subgroup,
  since $\Gamma_{i+1} + \Gamma_{i+1} \subset \Gamma_i$.
  Now $\bdl(\Gamma_i) = \bdl(\prod_{s \rightarrow \U} H_s)$ for any $i$,
  because $|\sum_{j=1}^{2^{s-i}}H_s| \leq |H_s|^{1+2^{s-i}4^{-s}} \leq 
  |H_s|^{1+2^{-s}}$.
  So $\bdl(\Gamma) = \bdl(\prod_{s \rightarrow \U} H_s)$,
  so setting our scaling parameter $\xi$ appropriately,
  we may ensure $\bdl(\Gamma)=\dim(G)$.

  We claim that $\Gamma_0$ is $cgp$ in $G$. Indeed, $\sum_{j=1}^{2^s}H_s$ is 
  contained in the union of $|H_s|^{2^s 4^{-s}} = |H_s|^{2^{-s}}$ translates 
  of $H_s$,
  so if $W \subsetneq G$ has complexity $\leq s$ then, using the lower bound 
  we assumed on $f$, we have
  $|W \cap \sum_{j=1}^{2^s}H_s| \leq |H_s|^{2^{-s}} |H_s|^{\frac1{f(s)}} \leq 
  |H_s|^{\frac1s}$.
  So $\sum_{j=1}^{2^s}H_s$ is $s$-$cgp$ in $G$.

  Let $b := \lim_{s \rightarrow \U} b_s$.
  Then $\bdl(\Gamma^n \cap V_b) = \dim(V_b)$.
  Set $C_0$ such that $G$ and $V_b$ are over $C_0$.
  By Fact \ref{fact:ideal} we can pick $\gammatup=(\gamma_1,\ldots,\gamma_n) 
  \in \Gamma^n\cap V_b$ with $\bdl(\gammatup) = \bdl(\Gamma^n \cap V_b) $. 
  Since $\Gamma_0$ is $cgp$ in $G$, Lemma \ref{cgplem} implies that all 
  $\gamma_i$'s are $cgp$ and $\bdl(\gamma_i)\leq \dimo(\gamma_i)$. Then 
  $\gammatup$ is coherent by Lemma~\ref{lem1}. So by 
  Theorem~\ref{thm:coherentSubgroup} we have that $V_b = 
  \locus^{G^n}(\gammatup/C_0)$ is a coset.
  But then so is $V_{b_m}$ for $\U$-many $m$, contradicting our assumption.
\end{proof}

\begin{example}[Connections with Manin-Mumford and Mordell-Lang]\label{MM}
  Let $G$ be a complex elliptic curve.
  Write $G[\infty] := \bigcup_{r \in \N} G[r]$ for the torsion subgroup.
  Suppose $V \subset G^n$ is an irreducible closed complex subvariety such 
  that $V(\C) \cap G[\infty]$ is Zariski dense in $V$. 

  We know, by the Manin-Mumford conjecture proven by Raynaud, that $V$ is a 
  coset of an algebraic subgroup. Some co-ordinate projection to $G^{\dim(V)}$ 
  yields an isogeny, so $V = H+\alpha$ where $H=\eta(G^{\dim(V)})$ is a 
  subgroup and $\eta$ is an isogeny, and $\alpha \in G[\infty]^n$.
  Setting $c := |\ker(\eta)|^{-1}$
  it follows that for $r \geq N:=\operatorname{ord}(\alpha)$ we have $|V(\C) 
  \cap G[r!]^n| \geq c\cdot|G^{\dim(V)}[r!]|$, and so for $r \in \N$ we have 
  the lower bound $|V(\C) \cap G[r!]^n| \geq \Omega(|G[r!]|^{\dim(V)})$.

  Suppose conversely that we only know this consequence of Manin-Mumford on 
  the asymptotics of the number of torsion points in $V$, or even just that 
  for every $\epsilon > 0$, for arbitrarily large $r \in \N$ we have $|V(\C) 
  \cap G[r!]^n| \geq |G[r!]|^{{\dim(V)}-\epsilon}$.
  Then it follows that $V$ is a coset of an algebraic subgroup.
  Indeed, $G[r!]$ is a subgroup and is trivially $\tau$-$cgp$ for any $\tau$ 
  since $\dim(G)=1$, so this is an immediate consequence of 
  Theorem~\ref{thm:coherentApproxSubgroup}.

  We can generalise this argument by replacing $G[\infty]$ with a finite rank 
  subgroup, as in the Mordell-Lang conjecture. Indeed, let $\Gamma \leq G(\C)$ 
  be a finite rank $\End(G)$-submodule.
  Say $\Gamma$ is contained in the divisible hull of the subgroup generated by 
  $\gamma_1,\ldots ,\gamma_k$.
  Let $\Gamma_r := \{ x \in \Gamma : (r!)x \in \sum_i [-r,\ldots ,r]\gamma_i 
  \}$.
  Then $\Gamma_r$ is finite and $|\Gamma_r+\Gamma_r|\leq 2^k|\Gamma_r|$, so as 
  above we obtain that if $V \subset G^n$ is an irreducible closed subvariety, 
  then $V$ is a coset of a subgroup if and only if for all $\epsilon>0$, for 
  arbitrarily large $r \in \N$, we have $|V(\C) \cap \Gamma_r^n| \geq 
  |\Gamma_r|^{\dim(V)-\epsilon}$.
\end{example}



\appendix
\section{Projective geometries fully embedded in algebraic geometry}
\label{appx:EHeq}

\cite{EH-projACF} characterises the projective subgeometries of the geometry 
of algebraic closure in an algebraically closed field $K$ over an 
algebraically closed subfield $C_0$. The points of such a geometry are 
$C_0$-interalgebraicity classes of elements of $K$.

In this essentially self-contained appendix, we consider the more general 
situation of a projective geometry induced from field-theoretic algebraic 
dependence whose points are $C_0$-interalgebraicity classes of finite tuples 
from $K$ (or, equivalently, of $K$-rational points of arbitrary varieties over 
$K$).

The arguments are generalisations of those used in \cite{EH-projACF}. We use 
Hrushovski's abelian group configuration theorem to find an abelian algebraic 
group, then apply a version of the fundamental theorem of projective geometry 
to identify the co-ordinatising skew field of the geometry as a skew field of 
quasi-isogenies of the group. Identifying the isogenies involved requires a 
little more care in the higher-dimensional case, as there may be non-trivial 
endomorphisms which are not isogenies, and these cannot appear in the 
co-ordinatising skew field.

We allow ourselves to simplify some of the algebra by restricting ourselves to 
the characteristic 0 case, whereas \cite{EH-projACF} works in arbitrary 
characteristic.

Let $K$ be an algebraically closed field of characteristic 0. 
Let $C_0 \leq K$ be an algebraically closed subfield, and let $\cl : 
\powerset(K^{<\infty}) \rightarrow \powerset(K^{<\infty})$ be field-theoretic 
algebraic closure over $C_0$ as defined in Example \ref{alg-tuples}. In other 
words, for a subset $B \subset K^{<\infty}:= \bigcup_{n \geq 1} K^n$ we let 
$\cl(B)$ be the set of tuples from the field-theoretic algebraic closure 
$C_0(B)^{\alg}$ of the subfield $C_0(B) \leq K$ generated by $C_0$ and the 
co-ordinates of all tuples from $B$. This closure operator\footnote{In model 
theory this is usually denoted by $\acleq$ and is defined on subsets of 
$\Keq$, which we identify here with $K^{<\omega}$ via elimination of 
imaginaries.
} was denoted $\acl^0(B)^{<\infty}$ in Section \ref{sec:proj}. So $a \in 
\cl(B)$ if and only if $a$ has finite orbit under $\Aut(K/C_0(B))$, if and 
only if $a \in (C_0(B)^{\alg})^{<\omega}$.

If $V$ is an algebraic variety over $C_0$ and $a \in V(K)$ is a $K$-rational
point, we may consider $a$ as a tuple in $K^{<\infty}$ as explained in \S 
\ref{absVars}.

Let $\G_K := \P(K^{<\infty};\cl)$ be the projectivisation of the closure 
structure $(K^{<\infty};\cl)$, as defined in \textsection\ref{subsect:geoms}; 
i.e.\ $\G_K = \{ \cl(\{a\}) : a \in K^{< \infty} \setminus \cl(\emptyset ) \}$ 
with the closure induced from (and still denoted by) $\cl$.
For $x \in K^{< \infty}$ and $C \subset K^{< \infty}$, define $\widetilde x := 
\cl(\{x\})$
and $\widetilde C := \{ \widetilde c : c \in C \}$.

As already noted $\G_K$ is not a geometry in general (it does not satisfy the 
exchange property), but here we are interested in geometries that embed in 
$\G_K$. We say that a geometry $P$ is \emph{connected} if any two points $a,b$ 
are non-orthogonal, i.e. if there exists $C \subset P$ such that $a \in 
\cl(b,C) \setminus \cl(C)$.

\begin{lemma} \label{lem:fullyEmbedded}
  Let $P \subset \G_K$ and suppose that the restriction $(P,\cl)$ of $\cl$ to 
  $P$ forms a connected geometry (embedded in $\G_K$). Then the following are 
  equivalent:
  \begin{enumerate}[(i)]\item for any $\widetilde x \in P$ and $\widetilde C 
      \subset P$,
      $\widetilde x \in \cl(\widetilde C) \Leftrightarrow x \nind_{C_0} C$
    \item[(ii)]
      There exists $k \in \N$ such that for any finite subset $\widetilde A 
      \subset P$,
      \[ k \cdot \dim_{\cl}(\widetilde A) = \dimo(A),\]
      where recall $\dimo(A):= \trd(C_0(A)/C_0)$.
  \end{enumerate}

\end{lemma}
\begin{proof} Note that (i) is equivalent to say that $\dimo(x/C)=0$ if and 
  only if $\dimo(x/C)<\dimo(x)$.
  That (i) implies (ii) follows from additivity of dimensions, setting $k := 
  \dimo(a)$ for any $a \in K^{<\infty}$ such that $\widetilde a \in P$, once 
  we see that this does not depend on the choice of $a$. For another such $b$ 
  with
  $\widetilde b \in P$ we show that $\dimo(a)=\dimo(b)$. If $\widetilde b = 
  \widetilde a$ then $b$ and $a$ are interalgebraic over $C_0$, so this holds. 
  Else by non-orthogonality there is $\widetilde C$ such that $\widetilde a 
  \in \cl(\widetilde C \widetilde b)$ and $\widetilde b \in \cl(\widetilde C 
  \widetilde a)$ but $\widetilde a,\widetilde b \notin \cl(\widetilde C)$.
  Then by (i), $\dimo(a) = \dimo(a/C) = \dimo(b/C) = \dimo(b)$.
  The converse is easy and since it is not needed in the sequel, we leave it to 
  the reader.
\end{proof}

\begin{definition} \label{defn:fullyEmbedded}
  We say a connected geometry $(P,\cl) \subset \G_K$ is 
  \defnstyle{($k$-dimensionally) fully embedded in $\G_K$} if the equivalent 
  conditions of the above lemma hold.
\end{definition}

If $G$ is a connected abelian algebraic group over $C_0$,
let $E_G := \End_{C_0}(G)$ be the ring of algebraic endomorphisms defined over 
$C_0$,
and let $E^0_G := \End^0_{C_0}(G) := \Q \otimes_\Z E_G$.
Any $\eta \in E^0_G$ can be written as $q\eta'$ for some $q \in \Q$ and $\eta' 
\in E_G$.
Since $\operatorname{char}(K)=0$ and $G$ is connected, $G(K)$ is divisible, 
and the $n$-torsion is finite for all $n$ and hence contained in $G(C_0)$.
So $V := G(K)/G(C_0)$ is naturally a left $E^0_G$-module.
If $F \leq E^0_G$ is a division subring, we view $V$ as an $F$-vector space 
and let $\P_F(G) := \P(V)$ be its projectivisation, and let $\eta_F^G : 
\P_F(G) \rightarrow \G_K$ be the map induced by $g \mapsto \widetilde {g}$ for 
$g \in G(K)$.
Note that $\eta_F^G$ is not injective.

\begin{example}
  Let $G$ and $F$ be as above.
  Let $g_i \in G(K)$ be independent generics over $C_0$ for $i$ in a (possibly 
  infinite) index set $I$.
  Let $V := \left<{(g_i/G(C_0))_{i \in I}}\right>_F \leq G(K)/G(C_0)$.
  Then $\eta_F^G$ maps the $|I|$-dimensional projective geometry 
  $\mathbb{P}_F(V) \subset \P_F(G)$ injectively into $\G_K$, and the image 
  $\eta_F^G(\mathbb{P}_F(V))$ is $\dim(G)$-dimensionally fully embedded in 
  $\G_K$.

  For example, in the case $G=\mathbb{G}_m$, if $a_0,\ldots ,a_n \in K$ with 
  $\trd(\a/C_0)=n+1$, then they generate in $K^*/C_0^*$ the $\Q$-subspace $\{ 
  \a^{\q} / C_0^* : \q \in \Q^{n+1} \} = \{ \prod_i a_i^{q_i} / C_0^* : 
  q_0,\ldots ,q_n \in \Q \}$; the algebraic dependencies over $C_0$ within 
  this set are precisely those arising from $\Q$-linear dependencies on the 
  exponents, and so this yields an embedding of $\P^n(\Q)$ in $\G_K$.
\end{example}

The following proposition, which is the main result of this appendix, says 
that any fully embedded projective geometry (of sufficiently large dimension) 
is of this form.

\begin{proposition} \label{prop:EHeq}
  Let $(P,\cl) \subset \G_K$ be a $k$-dimensionally fully embedded geometry, 
  and suppose $P$ is isomorphic to the geometry of a projective space over a 
  division ring $F$, and $\dim(P) \geq 3$.

  Then there is an abelian algebraic group $G$ over $C_0$ of dimension $k$,
  and an embedding of $F$ as a subring $F \leq E^0_G$,
  and a closed subgeometry $P'$ of $\P_F(G)$ on which $\eta_F^G$ is injective, 
  such that $P = \eta_F^G(P')$.


  Furthermore, $G$ is unique up to isogeny.
\end{proposition}

The remainder of this appendix constitutes a proof of 
Proposition~\ref{prop:EHeq}.
The strategy of the proof is to find the commutative algebraic group $G$ via 
the abelian group configuration theorem, and then to exhibit a natural 
injective collineation from $P$ to $\P_F(G)$. The fundamental theorem of 
projective geometry, in its version over division rings, then allows to claim 
that this collineation is a projective embedding. However, since we must also 
identify $F$ within $E^0_G$, we will in fact use a more general form of the 
fundamental theorem.

\begin{remark}
  In fact the proof applies directly to $C_0 \prec K$ models of an arbitrary 
  theory of finite Morley rank, with definable groups and endomorphisms in 
  place of algebraic groups and endomorphisms, as long as any connected 
  definable abelian group is divisible (equivalently, has finite $n$-torsion 
  for all $n$).
\end{remark}

\begin{remark}
  Unlike in \cite{EH-projACF}, our techniques do not directly apply in the 
  case of $P$ a projective plane (i.e.\ a 3-dimensional connected modular 
  geometry), and so do not rule out non-Desarguesian projective planes. 
  However, David Evans has pointed out to us that the arguments of 
  \cite{Lindstrom-desarguesian} go through to show that any projective plane 
  appearing as a subgeometry of $\cl$ is Desarguesian (and hence is a 
  projective plane over a division ring).
\end{remark}

\begin{lemma} \label{lem:quadrCollapse}
  Suppose $G$ is an abelian algebraic group over $C_0$,
  and let $g,h \in G(K)$ and $b \in P$.
  Suppose $\widetilde {g},\widetilde {h},b \in P$ are independent,
  and $\widetilde {g+h} \in P$,
  and $d \in \cl(\widetilde {g},b) \setminus \{\widetilde {g},b\}$.

  Then there is $h' \in G(K)$ such that $\widetilde {h'} = b$ and $\widetilde 
  {g+h'}=d$.
\end{lemma}
\begin{proof}
  By modularity, say
  $c \in \cl(\widetilde {h},b) \cap \cl(\widetilde {g+h},d)$.
  Then we proceed as in \cite[Lemma~1.2]{EH-autGeomACF}.
  Say $b = \widetilde b'$ and $d = \widetilde d'$.
  Then by the coheir property of independence in the stable theory $ACF$
  any formula in $\tp(c/C_0b'd'g)$ is satisfiable in $C_0$.
  So since $b'$ resp. $d'$ is interalgebraic over $c$ with points $x$ resp. 
  $y$ of $G$ satisfying $x+g=y$, they are already interalgebraic over $C_0$ 
  with such points, as required.

  (Alternatively, this argument can be phrased purely algebraically, by taking 
  a specialisation of $c$ to $C_0$ fixing $b'd'g$;
  see \cite[Lemma~2.1.1]{EH-projACF}.)
\end{proof}

\begin{lemma} \label{lem:lines}
  Suppose $a,b \in P$, $a\neq b$.
  Then there is an abelian algebraic group $G$ over $C_0$ with $\dim(G)=k$,
  and there exist $g,h \in G(K)$
  such that $a = \widetilde {g}$, and $b = \widetilde {h}$, and
  $\cl(a,b) \setminus \{a,b\} = P \cap \{ \widetilde {g + h} :
  g,h \in G(K), a = \widetilde {g}, b = \widetilde {h} \}$.
\end{lemma}
\begin{proof}
  This proof closely follows the argument of \cite[Theorem~3.3.1]{EH-projACF}.
  As there, this proof actually only needs $\dim(P) \geq 3$, but we will make 
  the argument slightly more concrete by using that $P$ is a projective 
  geometry over $F$. 

  Take $c \in P \setminus \cl(a,b)$, and identify $\cl(a,b,c)$ with $\P(F^3)$, 
  placing $a,b,c$ at $[0:1:0],[0:0:1],[1:0:0]$ respectively.
  Consider the following configuration.
  
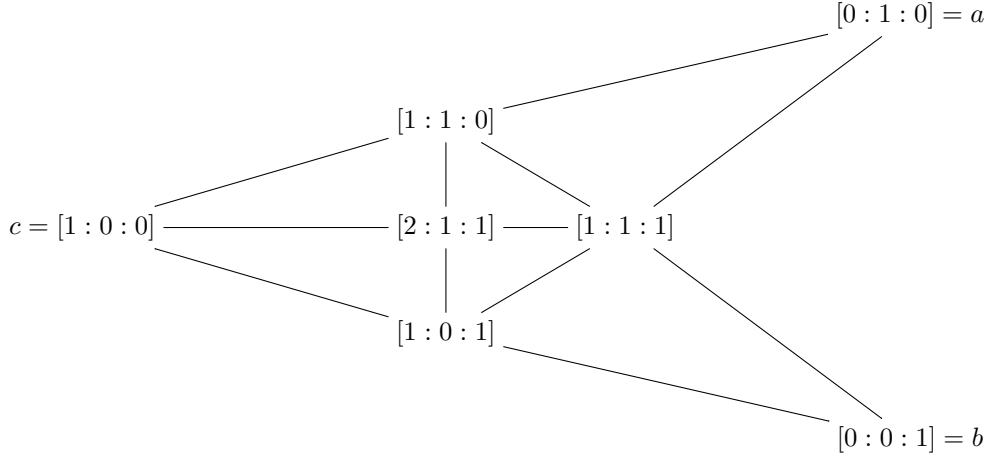
\[ \xymatrix{
  & & &&&& [0:1:0]=a \ar@{-}'[ddll][dddlll] \ar@{-}'[dlll][ddllllll] \\
  & & & [1:1:0] \ar@{-}'[d][dd] & & & \\
  c=[1:0:0] & & & [2:1:1] & [1:1:1] \ar@{-}'[l][llll] & & \\
  & & & [1:0:1] && & \\
  && & & & & [0:0:1]=b \ar@{-}'[uull][uuulll] \ar@{-}'[ulll][uullllll] 
  } \]
  By the assumption that $P$ is fully embedded in $\G_K$, this satisfies the 
  assumptions of Fact~\ref{fact:abGrpConf}.
  Hence there exist an abelian algebraic group $G$ over $C_0$ and generics 
  $g,k \in G(K)$ such that
  \begin{align*}
    a=[0:1:0] &= \widetilde {g}\\
    c=[1:0:0] &= \widetilde {k}\\
    [1:1:0] &= \widetilde {g+k} .
  \end{align*}

  Now let $d \in \cl(a,b) \setminus \{a,b\}$. Then by 
  Lemma~\ref{lem:quadrCollapse} applied to $g,k,b$, we have $d=\widetilde 
  {g+h}$ for some $h \in G(K)$ with $\widetilde {h}=b$. The converse inclusion 
  is clear.
\end{proof}

Now let $a_0,b_0 \in P$, $a_0 \neq b_0$, and fix $G$ and $g_0,h_0 \in G(K)$ as 
provided by this lemma, with $\widetilde {g_0}=a_0$ and $\widetilde 
{h_0}=b_0$.

Say $g \in G(K)$ is \defnstyle{$P$-aligned} if there exists $g' \in G(K)$ with 
$\widetilde {g'} \neq \widetilde {g}$ such that $\widetilde {g},\widetilde 
{g'},\widetilde {g+g'} \in P$. Such a $g'$ is a \defnstyle{$P$-alignment 
witness} for $g$.

\begin{lemma} \label{lem:alignedEx}
  Every point of $P$ is of the form $\widetilde {g}$ for some $P$-aligned $g 
  \in G(K)$.
\end{lemma}
\begin{proof}
  Let $c \in P$.
  If $c \in \cl(a_0,b_0)$, such a $g$ exists by Lemma~\ref{lem:lines}.
  Else, it follows from Lemma~\ref{lem:quadrCollapse} applied to $g_0,h_0,c$.
\end{proof}

\begin{lemma} \label{lem:alignedWitness}
  If $g \in G(K)$ is $P$-aligned and $b \in P \setminus \widetilde {g}$,
  then there exists a $P$-alignment witness $g' \in G(K)$ for $g$ with 
  $\widetilde {g'} = b$.
\end{lemma}
\begin{proof}
  By Lemma~\ref{lem:quadrCollapse},
  if $g'' \in G(K)$ is a $P$-alignment witness for $g$ and
  $b \notin \cl(\widetilde {g},\widetilde {g''})$, then there exists a 
  $P$-alignment witness $g'$ for $g$ with $\widetilde {g'} = b$.

  To handle the case that $b \in \cl(\widetilde {g},\widetilde {g''})$, apply 
  this first to $b' \in P \notin \cl(\widetilde {g},\widetilde {g''})$ to 
  obtain a witness $g'$, and then again to $b \notin \cl(\widetilde 
  {g},\widetilde {g'})$.
\end{proof}

\begin{lemma} \label{lem:alignedUnique}
  If $g,h \in G(K)$ are $P$-aligned and $\widetilde {g} = \widetilde {h}$, 
  then $E^0_G(g/G(C_0))=E^0_G(h/G(C_0))$.
\end{lemma}
\begin{proof}
  Say $g',h' \in G(K)$ are $P$-alignment witnesses for $g,h$ respectively.

  By Lemma~\ref{lem:alignedWitness}, we may assume $\widetilde {g'} \notin 
  \cl(\widetilde {h},\widetilde {h'})$.

  By Lemma~\ref{lem:quadrCollapse}, there is $h'' \in G(K)$ such that 
  $\widetilde {h''} = \widetilde {h'}$ and $\widetilde {g+h''} = \widetilde 
  {h+h'}$.
  Then by Fact~\ref{fact:corrIsog},
  there is $n \in \N$ and an isogeny $\alpha \in E_G$ and $k \in G(C_0)$ such 
  that $\alpha g = n h + k$,
  as required.
\end{proof}

\begin{lemma} \label{lem:alignedLines}
  Suppose $g,h \in G(K)$ are $P$-aligned and $\widetilde {g} \neq \widetilde 
  {h}$.
  Then $\cl(\widetilde {g},\widetilde {h}) = P \cap \{ \widetilde {k} : 
  k/G(C_0) \in E_G^0(g/G(C_0)) + E_G^0(h/G(C_0)) \}$.
\end{lemma}
\begin{proof}
  By Lemma~\ref{lem:alignedWitness}, we may take a $P$-alignment witness $g''$ 
  for $g$ with $\widetilde {h} \notin \cl(\widetilde {g},\widetilde {g''})$.

  Let $d \in \cl(\widetilde {g},\widetilde {h}) \setminus \{\widetilde 
  {g},\widetilde {h}\}$.
  Then by Lemma~\ref{lem:quadrCollapse},
  $d = \widetilde {g+h'}$ for some $h' \in G(K)$ with $\widetilde {h'} = 
  \widetilde {h}$.
  Then $\widetilde {h'}$ is $P$-aligned, so by Lemma~\ref{lem:alignedUnique},
  $h'/G(C_0) \in E_G^0(h/G(C_0))$.
\end{proof}

Our aim is to recognise $F$ as a subring of $E^0_G$, and $P$ as embedded in 
the corresponding $F$-projectivisation of a subspace of $G(K)/G(C)$.
This is a matter of the fundamental theorem of projective geometry.
However, since $E^0_G$ is not necessarily a field, this is not the classical 
case of the fundamental theorem. We use instead a version for projective 
spaces over rings obtained by Faure \cite{Faure-projGeomRings}.

The following definitions are adapted from \cite{Faure-projGeomRings}.
\begin{definition}
  The \defnstyle{projectivisation} $\P(M)$ of a module $M$ over a ring $R$ is 
  the set of non-zero 1-generated submodules $Rx$ equipped with the closure 
  operator $\cl_{\P(M)}$ induced from $R$-linear span,
  \[ \cl_{\P(M)}( (Ry_i)_i ) := \{ Rx : x \in \left<{(y_i)_i}\right>_R \} .\]

  A point $B \in \P(M)$ is \defnstyle{free} if $B = Rx$ for some $x \in M$ for 
  which $\{ \lambda \in R : \lambda x = 0 \} = \{0\}$.

  If $N$ is a module over a ring $S$, a map $g : \P(M) \rightarrow \P(N)$ is a
  \defnstyle{projective morphism} if $g(\cl_\P(M)(A)) \subset \cl_\P(N)(g(A))$ 
  for any $A \subset \P(M)$.

  If $\sigma : R \rightarrow S$ is a homomorphism, an additive map $f : M 
  \rightarrow N$ is \defnstyle{$\sigma$-semilinear} if $f(\lambda x) = 
  \sigma(\lambda) f(x)$ for any $x \in M$ and $\lambda \in R$. If $f$ is 
  injective, $\P(f) : \P(M) \rightarrow \P(N)$ is the induced projective 
  morphism.

  A ring $R$ is \defnstyle{directly finite} if $(\forall \lambda,\mu \in R) 
  (\lambda\mu=1 \Rightarrow \mu\lambda=1)$.
\end{definition}

\begin{fact} \label{fact:faure}
  Suppose $M$ and $N$ are modules over rings $R$ and $S$ respectively,
  and $S$ is directly finite.
  Suppose $g : \P(M) \rightarrow \P(N)$ is a projective morphism,
  and $\im(g)$ contains free points $B_1,B_2,B_3$
  such that for any $C_1,C_2 \in \im(g)$, for some $i \in \{1,2,3\}$, we have 
  \begin{equation}
    \tag{C3}
    \cl_{\P(N)}(C_1,C_2) \cap \cl_{\P(N)}(B_i) = \emptyset .
  \end{equation}

  Then there exists an embedding $\sigma : R \rightarrow S$ and a 
  $\sigma$-semilinear embedding $f : M \rightarrow N$ such that $g = \P(f)$.
\end{fact}
\begin{proof}
  This is the statement of \cite[Theorem~3.2]{Faure-projGeomRings} in the case 
  $E=\emptyset $,
  except that there $C_1$ and $C_2$ are not restricted to $\im(g)$; however, in 
  the proof the condition is used only when $C_1$ and $C_2$ are in $\im(g)$.

\end{proof}

Let $\P(G) := \P_{E^0_G}(G)$ be the projectivisation of the $E^0_G$-module 
$G(K)/G(C_0)$.

Then Lemma~\ref{lem:alignedEx} and Lemma~\ref{lem:alignedUnique}
establish a map $\iota : P 
\ensuremath{\lhook\joinrel\relbar\joinrel\rightarrow} \P_{E^0_G}(G)$, which by
Lemma~\ref{lem:alignedLines} is a projective morphism.
We proceed to verify the assumptions of Fact~\ref{fact:faure}.

$E^0_G$ is directly finite since if $\mu\lambda=1$ with $\mu,\lambda \in E_G$ 
then $\mu$ is an isogeny so has a quasi-inverse $\mu' \in E_G$ with $n := 
\mu'\mu \in \N_{>0}$;
then $\lambda\mu n = n\lambda\mu=\mu'\mu\lambda\mu=\mu'\mu=n$,
so $\lambda\mu=1$ since $G$ is $n$-divisible.

Now $\dim(P) \geq 3$, so say $\widetilde {g}_1,\widetilde {g}_2,\widetilde 
{g}_3 \in P$ are independent with $g_i$ being $P$-aligned.
Let $B_i := \iota(a_i) = E^0_G g_i$. Since each $g_i$ is generic in $G(K)$, 
each $B_i$ is free.
To check (C3), suppose $C_i = \iota(\widetilde {h}_i)$ for $i=1,2$ with $h_i$ 
being $P$-aligned.
Then some $\widetilde {g}_i \notin \cl(\widetilde {h}_1,\widetilde {h}_2)$, 
and so since $P$ is fully embedded in $\G_K$, we have $g_i \ind^0 h_1h_2$, 
from which (C3) follows.

Now say $V$ is an $F$-vector space such that $P \cong \P(V)$.
Then by Fact~\ref{fact:faure},
for some embedding $\sigma : F \rightarrow E^0_G$ and $\sigma$-semilinear 
embedding $f : V \rightarrow G(K)/G(C_0)$, we have $\iota = \P(f)$.

The main statement of Proposition~\ref{prop:EHeq} follows.

For the uniqueness up to isogeny of $G$, suppose Proposition~\ref{prop:EHeq} 
also holds for a group $G'$.
then if $g,h \in G(K)$ with $\widetilde {g},\widetilde {h} \in P$ and 
$\widetilde {g} \neq \widetilde {h}$, then, as in the proof of 
Lemma~\ref{lem:alignedLines}, there are $g',h' \in G'(K)$ with $\widetilde 
{g'}=\widetilde {g}$ and $\widetilde {h'}=\widetilde {h}$ and $\widetilde 
{g+h}=\widetilde {g'+h'}$.
So by Fact~\ref{fact:corrIsog}, $G'$ is isogenous to $G$.

\bibliographystyle{alpha}
\newcommand{\etalchar}[1]{$^{#1}$}

\end{document}